\documentclass{tac}

\author{Eugenia Cheng and Tom Leinster}

\thanks{TL was partially supported by an EPSRC Advanced Research Fellowship}

\address{EC: School of the Art Institute of Chicago, USA\\
TL: School of Mathematics, University of Edinburgh}

\title{Weak $\omg$-categories via terminal coalgebras}

\copyrightyear{2019}

\keywords{$\infty$-category, $\omega$-category, $n$-category, higher
  category, terminal coalgebra, final coalgebra}
\amsclass{18D05 (primary), 18C15}

\eaddress{echeng4@saic.edu, Tom.Leinster@ed.ac.uk}

\usepackage{amssymb}
\usepackage{pstricks}
\usepackage{pst-node}
\usepackage{bbm}

\psset{unit=1mm,linewidth=0.5pt,
arrowlength=1.1, arrowinset=0.7,arrowsize=4.5pt,
doublesep=1pt,
labelsep=2pt,nodesep=2pt}
  
\newcommand\1{\ensuremath{\mathbbm{1}}}%
\newcommand{\N}{\ensuremath{\mathbb N}} %xx
\newcommand{\omg}{\infty} 
\newcommand{\omgadj}{infinite} 
\newcommand{\cA}{{\cl{A}}}
\newcommand{\cC}{{\cl{C}}}
\newcommand{\cV}{{\cl{V}}}
\newcommand{\bI}{\bb{I}}
\newcommand{\iso}{\cong} 
\newcommand{\catequiv}{\simeq} 
\renewcommand{\:}{\colon}
\newcommand{\fn}[1]{\ensuremath{\mbox{\bfseries {\upshape {#1}\hspace{1.3pt}}}}}
\newcommand{\cat}[1]{\ensuremath{\textrm{\bfseries {\upshape {#1}}}}}
\newcommand{\Set}{{\cat{Set}}}
\newcommand{\Cat}{{\cat{Cat}}}
\newcommand{\Top}{{\cat{Top}}}
\newcommand{\Gph}{{\cat{Gph}}}
\newcommand{\wGph}{{\cat{$\omg$-Gph}}}
\newcommand{\CAT}{{\cat{CAT}}}
\newcommand{\Alg}{\fn{Alg}}
\newcommand{\cl}[1]{\ensuremath{\mathcal {#1}}}
\newcommand{\bb}[1]{\ensuremath{\mathbb {#1}}}
\newcommand{\ed}{\end{document}}
\newcommand{\id}{\ensuremath{\mbox{\em id}}}
\newcommand{\demph}[1]{{\bfseries #1}}

\newcommand{\fc}{\cat{fc}}
\newcommand{\fcv}{\ensuremath{\cat{fc}_{\cl{V}}}}
\newcommand{\vv}[1]{\vspace*{#1}}
\newcommand{\hh}[1]{\hspace*{#1}}
\newcommand{\ladj}{\ensuremath{\dashv}}

\newcommand{\tra}{\psset{unit=0.1cm,nodesep=0pt} \pspicture(8,0)
\pcline{->}(1,1.1)(7,1.1) \endpspicture}

\newcommand{\ltra}{\psset{unit=0.1cm,nodesep=0pt} \pspicture(15,0)
\pcline{->}(1.5,1.4)(13.5,1.4) \endpspicture}

\newcommand{\tmapsto}{\psset{unit=0.1cm,nodesep=0pt} \pspicture(8,0) %shorter arrow
\pcline{|->}(1,1.2)(7,1.2) \endpspicture}

\newcommand{\ltmapsto}{\psset{unit=0.1cm,nodesep=0pt} \pspicture(14,0)
\pcline{|->}(2.5,1.2)(11.5,1.2) \endpspicture}

\newcommand{\lthickmapsto}{\psset{unit=0.1cm,nodesep=0pt} \pspicture(14,0)
\pcline[linewidth=1.1pt,tbarsize=2pt 2]{|->}(2.5,1.2)(11.5,1.2) \endpspicture}

\newcommand{\hra}{
\psset{unit=0.1cm,labelsep=2pt,nodesep=0pt}
\pspicture(9,0)

\rput(2,2.4){\rnode{a1}{}}  % named node, with something placed there
\rput(2,1.1){\rnode{a2}{}}  % named node, with something placed there
\rput(8,1.1){\rnode{a3}{}}  % named node, with something placed there

\ncline{->}{a2}{a3}

\nccurve[angleA=180,angleB=180,ncurv=1.8]{-}{a1}{a2}
\endpspicture}

\newcommand{\tramap}[1]{\psset{unit=0.1cm,nodesep=0pt,labelsep=2pt} \pspicture(8,4)
\pcline{->}(1,1.1)(7,1.1)\naput{\ensuremath{\scriptstyle{#1}}} \endpspicture}

\newcommand{\tmap}{\tramap}

\newcommand{\mtmap}[1]{\psset{unit=0.1cm,nodesep=0pt,labelsep=2pt} \pspicture(10,4)
\pcline{->}(1,1.1)(9,1.1)\naput{\ensuremath{\scriptstyle{#1}}} \endpspicture}

\newcommand{\ltramap}[1]{\psset{unit=0.1cm,nodesep=0pt} \pspicture(15,4)
\pcline{->}(1.5,1.1)(13.5,1.1)\naput{\ensuremath{\scriptstyle{#1}}} \endpspicture}

\newcommand{\ltmap}{\ltramap}
\newcommand{\map}{\ltmap}

\newcommand{\trta}{\psset{unit=0.1cm,nodesep=0pt} \pspicture(8,0)
\pcline[doubleline=true,arrowinset=0.6,arrowlength=0.8]{->}(1,1.1)(7,1.1) \endpspicture}

\newcommand{\Tra}{\trta}

\newcommand{\noi}{\noindent}
\newcommand{\numroman}{\renewcommand{\labelenumi}{\roman{enumi})}}
\newcommand{\numarabic}{\renewcommand{\labelenumi}{\arabic{enumi}.}}
\newcommand{\numAlph}{\renewcommand{\labelenumi}{\Alph{enumi}.}}

\newenvironment{mydefinition}{\begin{definition}} {\end{definition}}
\newenvironment{myremark}{\begin{remark}} {\end{remark}}
\newenvironment{myexample}{\begin{example}} {\end{example}}

\newenvironment{prfof}[1]{\vspace{1ex}\begin{sloppypar}{\noindent
\upshape{\bfseries Proof of {#1}. }}} {{\hspace*{\fill}
$\Box$}\end{sloppypar}\vspace{2ex}}

\newcommand{\CATp}{\ensuremath{\cat{CAT}_\mathrm{p}}}
\newcommand{\CATd}{\ensuremath{\cat{CAT}_\mathrm{d}}}
\newcommand{\CATc}{\ensuremath{\cat{CAT}_\mathrm{c}}}
\newcommand{\CATpc}{\ensuremath{\cat{CAT}_\mathrm{pc}}}
\newcommand{\MND}{\ensuremath{\cat{MND}}}
\newcommand{\MNDd}{\ensuremath{\cat{MND}_\mathrm{d}}}
\newcommand{\MNDwk}{\ensuremath{\cat{MND}^\mathrm{wk}}}
\newcommand{\MNDdwk}{\ensuremath{\cat{MND}_\mathrm{d}^\mathrm{wk}}}
\newcommand{\FG}{\ensuremath{F_\mathrm{G}}}	
\newcommand{\FC}{\ensuremath{F_\mathrm{C}}}
\newcommand{\FM}{\ensuremath{F_\mathrm{M}}}
\newcommand{\DG}{\ensuremath{D_\mathrm{G}}}
\newcommand{\DC}{\ensuremath{D_\mathrm{C}}}
\newcommand{\DM}{\ensuremath{D_\mathrm{M}}}
\newcommand{\Algd}{\ensuremath{\cat{Alg}_\mathrm{d}}}
\newcommand{\Und}{\ensuremath{\cat{Und}}}
\newcommand{\Tom}{\ensuremath{T_\omg}}
\newcommand{\Pom}{\ensuremath{P_\omg}}
\newcommand{\Piom}{\ensuremath{\Pi_\omg}}
\newcommand{\iT}[1]{T^\mathrm{i}_{#1}}
\newcommand{\iP}[1]{P^\mathrm{i}_{#1}}
\newcommand{\iPi}[1]{\Pi^\mathrm{i}_{#1}}
\newcommand{\downarr}{\!\downarrow\!}
\newcommand{\tdalgd}{\ensuremath{\Top \downarr \Algd}}
\newcommand{\tdalg}{\ensuremath{\Top \downarr \Alg}}
\newcommand{\scat}[1]{\bb{#1}}  % Typeface for small categories
\renewcommand{\fc}[1]{\ensuremath{\mathbf{fc}_{#1}}} % Makes it the right
\renewcommand{\fcv}{\fc{\cl{V}}}
\newcommand{\tr}{\mathrm{tr}}   % Truncation
\newcommand{\proofbox}{\hspace*{\fill}$\Box$}
\newcommand{\lra}{\tra}
\renewcommand{\mapsto}{\tmapsto}

\begin{document}

\maketitle

\begin{abstract}

Higher categorical structures are often defined by induction on dimension, which \emph{a priori} produces only finite-dimensional structures.  In this paper we show how to extend such definitions to infinite dimensions using the theory of terminal coalgebras, and we apply this method to Trimble's notion of weak $n$-category.  Trimble's definition makes explicit the relationship between $n$-categories and topological spaces; our extended theory produces a definition of Trimble $\omg$-category and a fundamental $\omg$-groupoid construction.

Furthermore, terminal coalgebras are often constructed as limits of a
certain type.  We prove that the theory of Batanin--Leinster weak
\mbox{$\omg$-categories} arises as just such a limit, justifying our approach to Trimble \mbox{$\omg$-categories}.   In fact we work at the level of monads for $\omg$-categories, rather than $\omg$-categories themselves; this requires more sophisticated technology but also provides a more complete theory of the structures in question.
\end{abstract}

\tableofcontents

\section*{Introduction}

\addcontentsline{toc}{section}{Introduction}

\numarabic

The main aim of this work is to give a coinductive definition of Trimble-style $\omg$-categories.  We do this using the theory of terminal coalgebras for endofunctors.  

In 1999, Trimble proposed a definition of weak $n$-category which he called
``flabby'' $n$-category~\cite{lei7}.  The definition proceeds by iterated
enrichment.  It is well known that \emph{strict} $n$-categories can be
defined by iterated enrichment, that is, for $n \geq 1$, a strict
$n$-category is precisely a category enriched in strict $(n-1)$-categories.
However, for \emph{weak} $n$-categories, a notion of weak enrichment is
required.

Trimble does not attempt a fully general definition of weak enrichment, but focusses on the motivation from topology: the idea
that every topological space $X$ should give rise to a weak
$\omg$-category, its so-called fundamental $\omg$-groupoid $\Piom X$.
This was Trimble's main goal, although he only dealt with the finite-dimensional case, and to this end, he built the fundamental $n$-groupoid construction into the definition of
higher-dimensional category.  

Given a space $X$, it is clear what the underlying globular set of its
fundamental $\omg$-groupoid $\Piom X$ should be: the $0$-cells are the
points of $X$, the $1$-cells are the paths, the $2$-cells are the homotopies
of paths, and so on.  This underlying data comes equipped with various
operations, arising, for instance, from composition of paths and homotopies; the question of defining fundamental $\omg$-groupoids comes down to how to express these operations.

In Trimble's definition, the operations available in an abstract
higher-dimensional category are dictated by the homotopical operations generically available in topological spaces as above.  Thus, the definition depends inherently on the structure of $\Top$, a suitable category of topological spaces.

The definition uses operads to parametrise the weakly associative
composition in the higher-dimensional structures.  It is well-known how
operads can be used to describe homotopy monoids, in which the
multiplication is associative and unital only up to homotopy.  Trimble uses
a generalised version of this to make enriched categories in which
composition is associative and unital only ``up to homotopy''.  The idea is
as follows.

Given a symmetric monoidal category $\cl{V}$ and an operad $P$ in $\cl{V}$,
we may consider both $\cl{V}$-categories and $P$-algebras.  These concepts
can be amalgamated to obtain a third concept: that of $(\cl{V},
P)$-category, or category ``enriched in $\cl{V}$ and weakened by $P$''.
This concept has been defined and used under several names by several
authors \cite{may1,batweb1}.  An explanation with examples can be found in
\cite{che16}; an alternative point of view is given in \cite{batweb1}, in
which it is made particularly clear the sense in which this is a form of
weak enrichment.

In order to iterate enrichment of this form, it seems that we need a
\emph{series} of operads, one for each stage of the enrichment.  While this
is true \emph{a priori}, Trimble cleverly produces these operads as part of
the enrichment. For each $n$ he defines not just a category \cat{$n$-Cat}
of $n$-categories, but a fundamental $n$-groupoid functor
\[\Pi_n \: \Top \tra \cat{$n$-Cat}.\]
He then only needs to specify one operad $E$ in $\Top$; this produces a
series of operads $\Pi_n E$ in $\cat{$n$-Cat}$ which can be used for the
iterated enrichment.

As this definition is inductive, Trimble only defined a notion of
$n$-category for finite $n$. However by taking an appropriate limit we can
define a notion of $\omg$-category for this theory.  

We emphasise that we are using the term ``$\omg$-category'' in its fully
general sense (often also called ``$\omega$-categories'').  We do not
require that any of the cells in an $\omg$-category are in any way
invertible.  This is in contrast to the restricted meaning given to the
term by some authors, who use it as a synonym for what are properly called
$(\infty, 1)$-categories ($\omg$-categories in which all cells of dimension
greater than $1$ are weakly invertible).  Thus, $2$-categories,
$3$-categories and so on are all special cases of our $\omg$-categories.

Our limit-based approach to defining Trimble-style weak $\omg$-categories is
motivated by analogy with
\begin{enumerate}
 \item strict $\omg$-categories, and
\item Batanin/Leinster weak $\omg$-categories.
\end{enumerate}
In both cases we can build $\omg$-categories as a limit of $n$-dimensional
truncations; however in the weak case some delicacy is required as the
truncation of a weak $\omg$-category is not a weak $n$-category because of
coherence issues at the top dimension.

We can take further inspiration from strict $\omg$-categories as the construction can also be made via an enrichment endofunctor
\[\cV \mapsto \cat{$\cV$-Cat}.\]
The limit in question is then the terminal coalgebra for this endofunctor, via Ad\'amek's theorem (\cite{ada1}, see \cite{ada2} for a modern survey); the limit can be thought of as an ``infinite iteration'' of the endofunctor.  The theorem tells us that under certain conditions, the terminal coalgebra can be computed as the limit of the following sequential diagram
\begin{equation}\label{diagzero}
\psset{unit=0.1cm,labelsep=2pt,nodesep=2pt}
\pspicture(80,5)

% x1 x2 x3 x4

\rput(0,2){$\cdots$}

\rput(0,2){\rnode{x0}{\white{$\DG^3\1$}}}
\rput(22,2){\rnode{x1}{$F^3 1$}}
\rput(42,2){\rnode{x2}{$F^2 1$}}
\rput(62,2){\rnode{x3}{$F1$}}
\rput(80,2){\rnode{x4}{$1$}}

\psset{nodesep=2pt}
\ncline{->}{x0}{x1} \naput{{\scriptsize $F^3!$}}
\ncline{->}{x1}{x2} \naput{{\scriptsize $F^2!$}}
\ncline{->}{x2}{x3} \naput{{\scriptsize $F!$}}
\ncline{->}{x3}{x4} \naput{{\scriptsize $!$}}

\endpspicture
\end{equation}
The point is that in our examples, this diagram consists of (up to
isomorphism) the $n$-dimensional truncations of our \omgadj-dimensional
structure, so Ad\'amek's Theorem gives us the terminal coalgebra as a limit
of certain $n$-dimensional structures, which is what we would like the
\omgadj-dimensional structure to be. 

The Batanin-Leinster definition is not by iterated enrichment (at least,
not directly) but the Trimble definition is, in the appropriate weak sense.
So our aim is: 

\begin{enumerate}
 \item Define the category of Trimble $\omg$-categories as a limit of $n$-dimensional truncations.
\item Express this limit as the terminal coalgebra for a ``weak enrichment'' endofunctor.
\end{enumerate}

There are three further subtleties.  First, we prefer to work at the level of theories, so instead of simply constructing the \emph{category} of Trimble $\omg$-categories from the categories of $n$-dimensional truncations, we wish to construct the \emph{monad} for Trimble $\omg$-categories from the $n$-dimensional monads.

Secondly, Trimble's construction of $n$-categories is dependent on, and given simultaneously with, the ``fundamental $n$-groupoid'' functor
\[\Pi_n \: \Top \tra \cat{$n$-Cat}\]
at each stage of the enrichment.  This means that we need to work in a base category that expresses this data, that is, some sort of slice category.  Roughly speaking, instead of having an endofunctor for enrichment
\[\cV \tmapsto \cat{$\cV$-Cat}\]
we need one that includes $\Pi$, as in
\[\Top \tmap{\Pi} \cV \hh{1em} \tmapsto \hh{1em} \Top \tmap{\Pi^+} \cat{$\cV$-Cat}.\]
So this endofunctor seems to act on $\Top/\CAT$; for the monad version the situation needs to be adjusted accordingly.

Finally, in order for the inductive constructions to go through, we need some conditions on our categories and functors.  To construct the sort of enriched categories we require, we need to enrich in categories with finite products.  To construct the \emph{monads} that freely generate such enriched categories, we further need small coproducts and some distributivity conditions.  Thus we need to restrict the categories $\CAT$ and $\MND$ (a category of monads) variously.  This creates some technical questions when we come to take limits, but it is not too hard to check that this does not actually create any obstacles.  We include these technical results in an Appendix in order not to interrupt the flow of the paper; at a first approximation the main body of the paper can be read under the understanding that all limits required can be computed as expected.

The paper has two analogous halves, for the strict and weak cases.  As a preliminary, in Section~\ref{background} we recall the basic theory of terminal coalgebras that we will be using throughout.  We then apply this theory to analyse strict $\omg$-categories in three steps.  In Section~\ref{strictgph} we study the underlying data, $\omg$-graphs, in Section~\ref{strictcat} the categories of $n$-categories, and in Section~\ref{strictmnd} the monads for $n$-categories. This part serves partly as a warm-up for the more technically involved constructions of the second half, and partly as a basis---the second half will be not only analogous but will also build technically on these basic results.

In Section~\ref{batanin} we take a brief interlude from terminal coalgebras and study Batanin-Leinster weak $\omg$-categories by means of limits of $n$-dimensional truncations; this section thus assumes the most background.  Then we use terminal coalgebras to study Trimble $\omg$-categories analogously to Sections~\ref{strictgph}--\ref{strictmnd}: in Section~\ref{trimblegph} we study the underlying data, $\omg$-graphs, in Section~\ref{trimblecat} the categories of $n$-categories, and in Section~\ref{trimblemnd} the monads for $n$-categories.  

In each case the analysis follows broadly the same steps (though of course as things get more complicated, more definitions and intermediate lemmas are required):

\begin{enumerate}
 \item Define the category in which we work.
\item Define the endofunctor we study.
\item Define the $n$-dimensional structures we use, and truncation maps from $n$ to $n-1$ dimensions.
\item Show that these can be obtained by applying our endofunctor to the terminal object repeatedly, giving the diagram that appears in Ad\'amek's Theorem (\ref{diagzero}). 
\item Use Ad\'amek's Theorem to obtain \omgadj-dimensional structure as a limit of $n$-dimensional structures, hence as a terminal coalgebra for the endofunctor in question.
\end{enumerate}

Some of the work in this paper building up to the coalgebra results overlaps with the work of Weber \cite{web2} which was developed independently at around the same time.  Weber studies the enriched graph endofunctor and gives an inductive definition of the strict $n$-category monads, and the Trimble $n$-category monads, and he also makes widespread use of ``The formal theory of monads'' \cite{str1}.  However the emphasis of \cite{web2} and the related series of papers is different: there the aim is to give a framework for general processes of enrichment analogous to the construction and use of the Gray tensor product.  The aim of the current work is specifically to produce infinite-dimensional structures from finite-dimensional ones.  It is quite likely that Weber's work can be extended to give a concise approach to infinite-dimensional constructions in the manner of the present paper, but this is beyond our current scope.  

\subsection*{Assumed background}

\subsection*{Notation and terminology}

We will use various subscripts throughout on both our categories $\CAT$ and
$\MND$, and our endofunctors $F$ and $D$.  Lower case subscript ``p''
indicates finite products, and ``d'' indicates finite products and small
coproducts with distributivity.  On our endofunctors we will use ``G'',
``C'' and ``M'' to indicate when we are studying underyling graphs,
$n$-categories or monads respectively.  Our six main examples are summed up
in the following table; although some of the notation will not be explained
until  the main part of the paper we feel it will be useful to sum up the
structures at this point.

\[\begin{array}{c|cc}
  & \mbox{\textsf{strict}} & \mbox{\textsf{Trimble}} \\[6pt]
\hline \\[-6pt]
%%
%%%%%%%%%%%%%%%%%%%%%%%%%%%%%%%%%%%%%%%%%%%%%%%%%%%%%%%%%%%%%%%%%%%%%%%%
& \mbox{Section 2} & \mbox{Section 6} \\[4pt]
%%%%%%%%%%%%%%%%%%%%%%%%%%%%%%%%%%%%%%%%%%%%%%%%%%%%%%%%%%%%%%%%%%%%%%%%%%
\mbox{\textsf{underlying data}} 
& 
\begin{array}{c}
  \FG  \mbox{\ on\ } \CAT\\
  \cV \mapsto\cat{$\cV$-Gph}
  \end{array}
&
\begin{array}{c}
   \DG \mbox{\ on\ } \Top/\CAT\\
       (\cV,\Pi)  \mapsto  (\cat{$\cV$-Gph}, \Pi^+)
   \end{array} \\[20pt]
\hline \\[-6pt]
%%%%%%%%%%%%%%%%%%%%%%%%%%%%%%%%%%%%%%%%%%%%%%%%%%%%%%%%%%%%%%%%%%%%%
%
& \mbox{Section 3} & \mbox{Section 7}\\[4pt]
%
%%%%%%%%%%%%%%%%%%%%%%%%%%%%%%%%%%%%%%%%%%%%%%%%%%%%%%%%%%%%%%%%%%%%%%
%				
\begin{array}{l}		       
\mbox{\textsf{categories of}} \\[-5pt]
\mbox{\textsf{$\omg$-categories}}
\end{array}
&
\begin{array}{c}
  \FC \mbox{\ on\ } \CATp\\
  \cV \mapsto \cat{$\cV$-Cat}
   \end{array}
&
\begin{array}{c}
   \DC \mbox{\ on\ } \Top/\CATp \\
 (\cV, \Pi)\mapsto \big(\cat{$(\cV,\Pi E)$-Cat}, \Pi^+\big)
   \end{array}\\[20pt]
\hline \\[-6pt]
%%%%%%%%%%%%%%%%%%%%%%%%%%%%%%%%%%%%%%%%%%%%%%%%%%%%%%%%%%%%%%%%%%%%%%%%
%
& \mbox{Section 4} & \mbox{Section 8} \\[4pt]
%
%%%%%%%%%%%%%%%%%%%%%%%%%%%%%%%%%%%%%%%%%%%%%%%%%%%%%%%%%%%%%%%%%%%%%%%%
%
\begin{array}{l}
  \mbox{\textsf{monads for}} \\[-5pt]
  \mbox{\textsf{$\omg$-categories}}
  \end{array}
&
\begin{array}{c}
   \FM \mbox{\ on\ } \MNDd \\
        (\cV,T)  \mapsto  (\cat{$\cV$-Gph},T^+)
   \end{array}
&
\begin{array}{c}
   \DM \mbox{\ on\ } (\Top,1)/\MNDd \\
        (\cV, T, \Pi)  \mapsto  \big(\cat{$\cV$-Gph}, T^+, \Pi^+ \big)
   \end{array}\\[12pt]
\end{array}\]

We will often need to refer to structures defined by equations that hold
only up to coherent isomorphism or equivalence.  In the literature, a
complex system of modifiers has evolved, all expressing this same
condition, including ``weak'', ``strong'', ``pseudo'', and no modifier at
all.  We consistently use ``weak'' throughout, which has the disadvantage
of differing from some of the literature, but the advantages of simplicity
and uniformity.

\paragraph*{Acknowledgements}  We thank Martin Hyland, Steve Lack, Carlos
Simpson and Mark Weber for useful conversations.

%%%%%%%%%%%%%%%%%%%%%%%%%%%%%%%%%%%%%%%%%%%%%%%%%%%%%%%%%%%%%%%%%%%%%%%%%
%
% Section: Background on terminal coalgebras
%
%%%%%%%%%%%%%%%%%%%%%%%%%%%%%%%%%%%%%%%%%%%%%%%%%%%%%%%%%%%%%%%%%%%%%%%%%

\section{Background on terminal coalgebras}\label{termback}\label{background}

\begin{mydefinition}%1.1

A \demph{coalgebra} for an endofunctor $F\: \cl{C} \lra \cl{C}$ consists of

\begin{itemize}
\item an object $A \in \cl{C}$ \vspace{5pt}
\item a morphism 
$\psset{unit=0.1cm,labelsep=2pt,nodesep=2pt}
\pspicture(-5,8)(5,15)
% a1 
% a2 
\rput(0,15){\rnode{a1}{$A$}}  %  top
\rput(0,2){\rnode{a2}{$FA$}}  % bottom

\ncline{->}{a1}{a2} \naput{{\scriptsize $$}} % top

\endpspicture$
\end{itemize}
\vv{1.5em}satisfying no axioms. Coalgebras for $F$ form a category with the obvious morphisms
\[
\psset{unit=0.1cm,labelsep=3pt,nodesep=3pt}
\pspicture(20,20)

% a1 a2
% a3 a4

%%%%%%%%%% top

\rput(2,18){\rnode{a1}{$A$}}  % top left
\rput(18,18){\rnode{a2}{$B$}}  % top right
\rput(2,2){\rnode{a3}{$FA$}}  % bottom left
\rput(18,2){\rnode{a4}{$FB$}}  % bottom right

\ncline{->}{a1}{a2} \naput{{\scriptsize $h$}} % top
\ncline{->}{a3}{a4} \nbput{{\scriptsize $Fh$}} % bottom
\ncline{->}{a1}{a3} \nbput{{\scriptsize $$}} % left
\ncline{->}{a2}{a4} \naput{{\scriptsize $$}} % right

\endpspicture
\]
so we can look for terminal coalgebras.

\end{mydefinition}

\begin{myexample}\label{exone} %1.2

Given a set $M$ we have an endofunctor
\[\begin{array}{ccc}
\cat{Set} & \ltmap{M \times \_} & \cat{Set} \\[4pt]
A & \tmapsto & M \times A. \\[6pt]
\end{array}\]
A coalgebra for this is a function
\[\begin{array}{ccc}
A & \ltmap{(m,f)} & M \times A \\[4pt]
a & \tmapsto & (m(a), f(a)).
\end{array}\]
The terminal coalgebra is given by the set $M^{\bb{N}}$ of
``infinite words''
\[(m_0, m_1, m_2, \ldots)\]
in $M$. The structure map 
\[
\psset{unit=0.1cm,labelsep=2pt}
\pspicture(5,17)

% a1 
% a2 
\rput(0,15){\rnode{a1}{$M^\N$}}  %  top
\rput(0,2){\rnode{a2}{$M \times M^\N$}}  % bottom

\ncline[nodesepA=3pt,nodesepB=1pt]{->}{a1}{a2} \naput{{\scriptsize $$}} % top

\endpspicture
\]
is given by a canonical isomorphism. To show that this is terminal, given any
coalgebra 
\[
\psset{unit=0.1cm,labelsep=2pt}
\pspicture(5,17)

% a1 
% a2 
\rput(0,15){\rnode{a1}{$A$}}  %  top
\rput(0,2){\rnode{a2}{$M \times A$}}  % bottom

\ncline[nodesepA=3pt,nodesepB=2pt]{->}{a1}{a2} \naput{{\scriptsize $(m,f)$}} % top

\endpspicture
\]
we need a unique map
\[
\psset{unit=0.1cm,labelsep=3pt,nodesep=3pt}
\pspicture(30,25)

%%%%%%%%%% top

\rput(0,20){\rnode{a1}{$A$}}  % top left
\rput(30,20){\rnode{a2}{$M^\N$}}  % top right
\rput(0,0){\rnode{a3}{$M \times A$}}  % bottom left
\rput(30,0){\rnode{a4}{$M \times M^\N$}}  % bottom right

\ncline{->}{a1}{a2} \naput{{\scriptsize $t$}} % top
\ncline{->}{a3}{a4} \nbput{{\scriptsize $1 \times t$}} % bottom
\ncline{->}{a1}{a3} \nbput{{\scriptsize $$}} % left
\ncline{->}{a2}{a4} \naput{{\scriptsize $$}} % right

\endpspicture
\]
and we have
\[t\: a \ltmapsto \big(m(a), \hh{2pt} m(f(a)), \hh{2pt} m(f^2(a)), \hh{2pt} m(f^3(a)),\ldots\big).\]

\end{myexample}

\begin{myexample}\label{extwo} %1.3
Let $F$ be the free monoid monad on \cat{Set}.  The terminal coalgebra is
given by the set $\mbox{Tr}^{\infty}$ of ``infinite trees of finite arity''.
The structure map
\[
\psset{unit=0.1cm,labelsep=2pt,nodesep=2pt}
\pspicture(30,17)

% a1 
% a2 
\rput(0,15){\rnode{a1}{$\mbox{Tr}^\infty$}}  %  top
\rput(0,2){\rnode{a2}{$F(\mbox{Tr}^{\infty})$}}  % bottom

\rput[l](8,2){$=$ \mbox{``finite strings of infinite trees''}}

\ncline{->}{a1}{a2} \naput{{\scriptsize $$}} % top

\endpspicture
\]
is again given by a canonical isomorphism.  For details, see, for
instance,~3.9 and~3.10 of \cite{ada2}.

\end{myexample}

\begin{lemma}[Lambek \cite{lam1}]\label{lambek}  Let $F$ be an endofunctor of a category.
If 
\[\psset{unit=0.1cm,labelsep=2pt,nodesep=2pt}
\pspicture(5,17)

% a1 
% a2 
\rput(0,15){\rnode{a1}{$A$}}  %  top
\rput(0,2){\rnode{a2}{$FA$}}  % bottom

\ncline{->}{a1}{a2} \naput{{\scriptsize $f$}} % top

\endpspicture\] 
is a terminal coalgebra for $F$ then $f$ is an isomorphism.

\end{lemma}

\begin{theorem}[Ad\'{a}mek \cite{ada1}]\label{adamek} Let $F$ be an endofunctor on a category with a terminal object $1$.  Suppose
that the limit of the following diagram exists and is preserved by $F$:
\[
\psset{unit=0.1cm,labelsep=2pt,nodesep=2pt}
\pspicture(80,5)

% x1 x2 x3 x4

\rput(0,2){$\cdots$}

\rput(0,2){\rnode{x0}{\white{$\DG^3\1$}}}
\rput(22,2){\rnode{x1}{$F^3 1$}}
\rput(42,2){\rnode{x2}{$F^2 1$}}
\rput(62,2){\rnode{x3}{$F1$}}
\rput(80,2){\rnode{x4}{$1$}}

\psset{nodesep=2pt}
\ncline{->}{x0}{x1} \naput{{\scriptsize $F^3!$}}
\ncline{->}{x1}{x2} \naput{{\scriptsize $F^2!$}}
\ncline{->}{x2}{x3} \naput{{\scriptsize $F!$}}
\ncline{->}{x3}{x4} \naput{{\scriptsize $!$}}

\endpspicture
\]
Then this limit is (the underlying object of) a terminal coalgebra for $F$.
\end{theorem}

We now revisit the previous two examples using this theorem.

\begin{myexample} (cf.\ Example~\ref{exone})  %1.6
Given a set $M$ we considered
the endofunctor 
\[\begin{array}{ccc}
\cat{Set} & \map{M \times \_} & \cat{Set} \\[4pt]
A & \mapsto & M \times A \\[6pt]
\end{array}\]
and saw that the terminal coalgebra was given by the set $M^{\bb{N}}$ of
``infinite words'' in $M$.  Using Theorem~\ref{adamek} we can construct a
terminal coalgebra as the limit of
\[
\psset{unit=0.1cm,labelsep=2pt,nodesep=2pt}
\pspicture(85,5)

% x1 x2 x3 x4

\rput(0,2){$\cdots$}

\rput(0,2){\rnode{x0}{\white{$\DG^3\1$}}}
\rput(22,2){\rnode{x1}{$F^3 1$}}
\rput(46,2){\rnode{x2}{$F^2 1$}}
\rput(70,2){\rnode{x3}{$F1$}}
\rput(85,2){\rnode{x4}{$1$}}

\psset{nodesep=2pt}
\ncline{->}{x0}{x1} \naput{{\scriptsize $F^3!$}}
\ncline{->}{x1}{x2} \naput{{\scriptsize $F^2!$}}
\ncline{->}{x2}{x3} \naput{{\scriptsize $F!$}}
\ncline{->}{x3}{x4} \naput{{\scriptsize $!$}}

\endpspicture
\]
For the functor $F$ in this example, this diagram becomes
\[
\psset{unit=0.1cm,labelsep=2pt,nodesep=2pt}
\pspicture(85,5)

% x1 x2 x3 x4

\rput(0,2){$\cdots$}

\rput(0,2){\rnode{x0}{\white{$\DG^3\1$}}}
\rput(22,2){\rnode{x1}{$M^3 \times 1$}}
\rput(46,2){\rnode{x2}{$M^2 \times  1$}}
\rput(70,2){\rnode{x3}{$M\times 1$}}
\rput(85,2){\rnode{x4}{$1$}}

\psset{nodesep=2pt}
\ncline{->}{x0}{x1} \naput{{\scriptsize $M^3\times !$}}
\ncline{->}{x1}{x2} \naput{{\scriptsize $M^2 \times !$}}
\ncline{->}{x2}{x3} \naput{{\scriptsize $M\times !$}}
\ncline{->}{x3}{x4} \naput{{\scriptsize $!$}}

\endpspicture
\]
which is isomorphic to
\[
\psset{unit=0.1cm,labelsep=2pt,nodesep=2pt}
\pspicture(85,5)

% x1 x2 x3 x4

\rput(0,2){$\cdots$}

\rput(0,2){\rnode{x0}{\white{$\DG^3\1$}}}
\rput(22,2){\rnode{x1}{$M^3$}}
\rput(46,2){\rnode{x2}{$M^2$}}
\rput(70,2){\rnode{x3}{$M$}}
\rput(85,2){\rnode{x4}{$1$}}

\psset{nodesep=2pt}
\ncline{->}{x0}{x1} \naput{{\scriptsize $M^3$}}
\ncline{->}{x1}{x2} \naput{{\scriptsize $M^2$}}
\ncline{->}{x2}{x3} \naput{{\scriptsize $M$}}
\ncline{->}{x3}{x4} \naput{{\scriptsize $!$}}

\endpspicture
\]
and taking the limit of this diagram does indeed give infinite words in $M$.

Example~\ref{extwo} works similarly.
\end{myexample}

\begin{myexample}\label{onepointeight} %1.8

Let \CAT\ be the category of locally small categories and functors between
them.  For any locally small category \cl{V}, let $\cl{V}\cat{-Gph}$ be the
(locally small) category of \cl{V}-graphs and their morphisms.  A \cl{V}-graph
$A$ consists of
\begin{itemize}
\item a set $\cat{ob}A$
\item for all $x,y \in \cat{ob}A$ \ an object $A(x,y) \in \cl{V}$.
\end{itemize}
A morphism $f\: A \lra B$ of \cl{V}-graphs is given by a function $f\: \cat{ob}A
\lra \cat{ob}B$ and for all pairs of objects $a, a' \in \cat{ob}A$, a morphism
$A(a,a') \lra B(fa, fa')$ in $\cl{V}$. 

There is an endofunctor
\[\begin{array}{ccccc}
(\ )\cat{-Gph} &:& \CAT & \lra & \CAT \\[4pt]
&& \cl{V} & \mapsto & \cl{V}\cat{-Gph} \\
&& H & \mapsto & H_*
\end{array}\]
where $H_*$ ``acts locally as $H$''.  That is, given a functor $H\: \cl{V} \lra
\cl{W}$ and a \cl{V}-graph $A$, the \cl{W}-graph $H_*A$ has the same objects
as $A$, and 
\[(H_*A)(a,a') = H(A(a,a')).\] 

This functor has a terminal coalgebra given by the category $\omg\cat{-Gph}$
of \emph{$\omg$-graphs}, which is equivalent to the category \cat{GSet} of
globular sets.  We will study this example in detail in the next section.

\end{myexample}

\begin{myexample}   \label{eg:Simpson} %1.9
We learnt the following example from Carlos Simpson.  Write $\CATp$
for the category of locally small finite product categories and finite product
preserving functors.  There is an endofunctor
\[\begin{array}{ccc}
\CATp & \lra & \CATp \\[4pt]
\cl{V} & \mapsto & \cl{V}\cat{-Cat} \\
H & \mapsto & H_*.
\end{array}\]
This has a terminal coalgebra given by the category $\omg\cat{-Cat}$ of
\emph{strict $\omg$-categories.}  Section~\ref{sec:strict} consists of a
detailed study of this example, as both motivation and warm-up for the main
part of the paper.
\end{myexample}

The general idea, then, is that terminal coalgebras give us a way of
constructing infinite versions of gadgets whose finite versions we can
construct simply by induction.  Our aim is to apply this to Trimble's version
of weak $n$-categories.

%%%%%%%%%%%%%%%%%%%%%%%%%%%%%%%%%%%%%%%%%%%%%%%%%%%%%%%%%%%%%%%%%%%%%%%%%
%%%%%%%%%%%%%%%%%%%%%%%%%%%%%%%%%%%%%%%%%%%%%%%%%%%%%%%%%%%%%%%%%%%%%%%%%
% Section: omega-graphs
%%%%%%%%%%%%%%%%%%%%%%%%%%%%%%%%%%%%%%%%%%%%%%%%%%%%%%%%%%%%%%%%%%%%%%%%%
%%%%%%%%%%%%%%%%%%%%%%%%%%%%%%%%%%%%%%%%%%%%%%%%%%%%%%%%%%%%%%%%%%%%%%%%%

\section{$\omg$-graphs}
\label{sec:graphs}\label{strictgph}

This section contains the simplest of our terminal coalgebra theorems.  It
acts as a warm-up, in that the arguments in the later sections are similar to
the arguments here but take place in more sophisticated settings.  Some later
sections will also depend on the theorem proved here.

Our purpose is to characterise $\omg\cat{-Gph}$ as a terminal coalgebra.   As $\omg$-graphs give the underlying data for all our $\omg$-categories (strict and weak), this construction will be the basis of all further constructions.  We use the following endofunctor, but introduce notation to fit with the overall system of the paper.

\begin{mydefinition} %3.1
\label{defn:FG}\label{FG}

We write $\FG$ for the endofunctor 
\[ (\ )\cat{-Gph} \: \CAT \lra \CAT \] 
defined in Example~\ref{onepointeight}. 
\end{mydefinition}

In fact, $\FG$ is an endo-2-functor on the 2-category $\cat{CAT}$, with the
action of natural transformations defined in the obvious way.  For the most
part, though, we will continue to treat $\CAT$ as a 1-category.

A \cl{V}-graph is the underlying data for a \cl{V}-category, and an $n$-graph
is the underlying data for an $n$-category.  We define $n$-graphs by iterated
enrichment as follows.

\begin{mydefinition} %3.3
\label{defn:nGph}\label{threepointthree}

For all $n \geq 0$ we define the category $n\cat{-Gph}$ of \demph{$n$-graphs}
as follows: 

\begin{itemize}

\item $0\cat{-Gph} = \cat{Set}$, and
\item for all $n \geq 1$, $n\cat{-Gph} = \FG\big((n-1)\cat{-Gph}\big) =
\big((n-1)\cat{-Gph}\big)\cat{-Gph}$. 

\end{itemize}

\noi We also define for all $n \geq 1$ the \demph{truncation} functor 
\[
U_n\:
n\cat{-Gph} \lra (n-1)\cat{-Gph}, 
\]
as follows. 

\begin{itemize}

\item $U_1\: 1\cat{-Gph} \lra \cat{Set}$ is the functor that sends a graph
to its set of objects, and 
\item for all $n \geq 2$, $U_n = \FG(U_{n-1}) = {U_{n-1}}_*$.

\end{itemize}

\end{mydefinition}

Alternatively, $n$-graphs can also be understood as $n$-globular sets; the categories of such are equivalent.  This is a little tangential to our main argument so we will discuss it at the end of the section.

It is straightforward to make an \omgadj-dimensional version of
$n$-globular sets directly; this is the category \cat{GSet} of globular
sets.  However for the \omgadj-dimensional version of $n\cat{-Gph}$ we
cannot proceed by induction.  Instead we define \omgadj-dimensional graphs
as follows.  It is the simplest of our various coinductive constructions.

\begin{mydefinition} \label{threepointfive}\label{omegagraphdef} %3.5 
The category $\omg\cat{-Gph}$ of \demph{$\omg$-graphs} is the limit of the
diagram
\[ 
\cdots 
\map{U_3} 2\cat{-Gph} 
\map{U_2} 1\cat{-Gph} 
\map{U_1} 0\cat{-Gph} 
\]
in $\CAT$.  We write $\tr_n\: \omg\cat{-Gph} \lra n\cat{-Gph}$ for the
$n$th projection of the limit cone, and refer to $\tr_n$ (as well as $U_n$) as
\demph{truncation}.
\end{mydefinition}

Thus the limit cone looks like
\[
\psset{unit=0.095cm,labelsep=2pt,nodesep=2pt}
\pspicture(0,-5)(114,24)
% standard limit diagram
%   a
% x3 x2 x1 x0

\rput(20,20){\rnode{a}{\cat{$\omg$-Gph}}}

\rput(0,0){$\cdots$}

\rput(0,0){\rnode{x3}{\white{$\DG^3\1$}}}
\rput(30,0){\rnode{x2}{$\cat{$2$-Gph}$}}
\rput(70,0){\rnode{x1}{$\cat{$1$-Gph}$}}
\rput(108,0){\rnode{x0}{$\cat{$0$-Gph}$}}

\psset{nodesep=2pt}
\ncline{->}{x3}{x2} \nbput{{\scriptsize $U_2$}}
\ncline{->}{x2}{x1} \nbput{{\scriptsize ${U_2}$}}
\ncline{->}{x1}{x0} \nbput{{\scriptsize ${U_1}$}}

\psset{labelsep=0pt}
\ncline{->}{a}{x0} \naput[npos=0.7]{{\scriptsize $\mbox{tr}_0$}}
\ncline{->}{a}{x1} \naput[npos=0.75]{{\scriptsize $\mbox{tr}_1$}}
\ncline{->}{a}{x2} \naput[npos=0.6]{{\scriptsize $\mbox{tr}_2$}}
\ncline{->}{a}{x3} \naput{{\scriptsize $$}}

\endpspicture
\]
and we may express an $\omg$-graph as an infinite sequence
\[(\ldots, X_2, X_1, X_0)\]
where each $X_n$ is an $n$-graph and is the truncation of $X_{n+1}$.

\begin{theorem} %3.10
\label{thm:gph-tc}\label{threepointten}

The category $\omg\cat{-Gph}$ is the terminal coalgebra
for the endofunctor 
\[\FG \: \CAT \tra \CAT.\]

\end{theorem}

\begin{proof}

We use Ad\'amek's theorem.  We need to consider the limit in $\CAT$ of
\begin{equation}\label{diaga}
\psset{unit=0.1cm,labelsep=2pt,nodesep=2pt}
\pspicture(0,25)(80,35)

% x1 x2 x3

\rput(0,30){$\cdots$}

\rput(0,30){\rnode{x0}{\white{$\DG^3\1$}}}
\rput(30,30){\rnode{x1}{$\FG^3\1$}}
\rput(55,30){\rnode{x2}{$\FG^2\1$}}
\rput(80,30){\rnode{x3}{$\FG\1$}}

\psset{nodesep=4pt}
\ncline{->}{x0}{x1} \naput{{\scriptsize $\FG^3!$}}
\ncline{->}{x1}{x2} \naput{{\scriptsize $\FG^2!$}}
\ncline{->}{x2}{x3} \naput{{\scriptsize $\FG!$}}

\endpspicture
\end{equation}
where $\1$ is the terminal object of $\CAT$.  We need to show
\begin{enumerate}
 \item the limit is \cat{$\omg$-Gph}, and
\item the limit is preserved by $\FG$.
\end{enumerate}
First we note that the above diagram (\ref{diaga}) is isomorphic to the one whose limit defines \cat{$\omg$-Gph}; formally, there are canonical isomorphisms $(I_n)_{n \geq 0}$ such that the following
diagram commutes:
\[
\psset{unit=0.1cm,labelsep=2pt,nodesep=2pt}
\pspicture(0,5)(80,35)

% x1 x2 x3
% y1 y2 y3

\rput(0,30){$\cdots$}
\rput(0,10){$\cdots$}

\rput(0,30){\rnode{x0}{\white{$\DG^3\1$}}}
\rput(30,30){\rnode{x1}{$\FG^3\1$}}
\rput(55,30){\rnode{x2}{$\FG^2\1$}}
\rput(80,30){\rnode{x3}{$\FG\1$}}

\psset{nodesep=4pt}
\ncline{->}{x0}{x1} \naput{{\scriptsize $\FG^3!$}}
\ncline{->}{x1}{x2} \naput{{\scriptsize $\FG^2!$}}
\ncline{->}{x2}{x3} \naput{{\scriptsize $\FG!$}}

\rput(3,10){\rnode{y0}{\white{$\DG^3\1$}}}

\rput(30,10){\rnode{y1}{$\cat{2-Gph}$}}
\rput(55,10){\rnode{y2}{$\cat{1-Gph}$}}
\rput(80,10){\rnode{y3}{$\cat{0-Gph}$}}

\psset{nodesep=2pt}
\ncline{->}{y0}{y1} \nbput{{\scriptsize $U_3$}}
\ncline{->}{y1}{y2} \nbput{{\scriptsize $U_2$}}
\ncline{->}{y2}{y3} \nbput{{\scriptsize $U_1$}}

\psset{nodesepA=5pt,nodesepB=3pt,labelsep=2pt}
\ncline{->}{x1}{y1} \naput{{\scriptsize $I_2$}} \nbput{{\scriptsize $\iso$}}
\ncline{->}{x2}{y2} \naput{{\scriptsize $I_1$}} \nbput{{\scriptsize $\iso$}}
\ncline{->}{x3}{y3} \naput{{\scriptsize $I_0$}} \nbput{{\scriptsize $\iso$}}

\endpspicture
\]
This is true by a straightforward induction: we have
\[
\FG\1 = \1\cat{-Gph} \iso \cat{Set} = 0\cat{-Gph},
\]
giving the isomorphism $I_0$, and we put $I_n = \FG^n(I_0)$.  The rightmost square of the diagram commutes; hence by induction, the whole diagram commutes.

This shows that the the limit we need to take is over the following diagram
\[ 
\cdots 
\map{U_3} 2\cat{-Gph} 
\map{U_2} 1\cat{-Gph} 
\map{U_1} 0\cat{-Gph} 
\]
and this is by definition $\cat{$\omg$-Gph}$.

We now show that $\FG$ preserves this limit.  (In fact, $\FG$ preserves all
connected limits.  This can be proved by a similar argument, and also
follows from the fact that $\FG$ is a local right adjoint, as explained in
Section~3 of Weber~\cite{web2}.)  We need to show that
$\cat{($\omg$-Gph)-Gph}$ is the limit of the following diagram
\[
\psset{unit=0.095cm,labelsep=2pt,nodesep=2pt}
\pspicture(0,-5)(114,5)
% x3 x2 x1 x0

\rput(0,0){$\cdots$}

\rput(0,0){\rnode{x3}{\white{$\DG^3\1$}}}
\rput(30,0){\rnode{x2}{$\cat{($2$-Gph)-Gph}$}}
\rput(70,0){\rnode{x1}{$\cat{($1$-Gph)-Gph}$}}
\rput(108,0){\rnode{x0}{$\cat{($0$-Gph)-Gph}$.}}

\psset{nodesep=2pt}
\ncline{->}{x3}{x2} \naput{{\scriptsize $$}}
\ncline{->}{x2}{x1} \naput{{\scriptsize ${U_2}_*$}}
\ncline{->}{x1}{x0} \naput{{\scriptsize ${U_1}_*$}}
\endpspicture
\]
Let $X$ be an object of the limit, so $X$ consists of an infinite sequence
\[
(\ldots, X_2, X_1, X_0)
\]
with $X_n \in \cat{($n$-Gph)-Gph}$ and 
\begin{equation}\label{equ1}
(U_{n + 1})_* X_{n + 1} = X_n
\end{equation}
for all $n \geq 0$.  On objects this gives
\[
\cat{ob} X_{n + 1}
=
\cat{ob} \big((U_{n + 1})_* X_{n + 1}\big) 
= 
\cat{ob} X_n
\]
for all $n \geq 0$, so the graphs $X_n$ all have the same object-set, $S$,
say.  Then on homs, (\ref{equ1}) gives, for all
$a, b \in S$, 
\[
U_{n + 1} \big(X_{n + 1}(a, b)\big)
=
X_n(a, b).
\]
So $X$ consists of a set $S$ together with, for each $a, b \in S$, a sequence
\[\big(\ldots, \ X_{2}(a,b), \ X_1(a,b), \ X_0(a,b)\big) \in \cat{$\omg$-Gph}.\]
In other words, $X$ is an object of \cat{($\omg$-Gph)-Gph}.  A similar argument applies to morphisms.

We have now shown that the limit of the original diagram is \cat{$\omg$-Gph} and that it is preserved by $\FG$, so by Ad\'amek's Theorem \cat{$\omg$-Gph} is the terminal coalgebra for $\FG$ as required.
\end{proof}

Note that by Lambek's Lemma~(\ref{lambek}) we then have:

\begin{corollary} %3.11
\label{cor:gph-fixpt}

There is a canonical isomorphism $(\omg\cat{-Gph})\cat{-Gph} \iso
\omg\cat{-Gph}$.
\proofbox
\end{corollary}

Finally we discuss the relationship between $n$-graphs and $n$-globular sets; this is not directly part of our terminal coalgebra narrative, but will later help us deduce that the categories of $n$-graphs and $\omg$-graphs have the limit/colimit properties required.

%%%%%%%%%%%%%%%%%%% comparison with GSets

Recall that an $n$-globular set is a diagram
% b0   b1 b2 b3 ... b4 b5
\[
\psset{labelsep=2pt}
\pspicture(0,4)(70,12)

\rput(3,8){\rnode{b1}{$X_n$}}  
\rput(25,8){\rnode{b2}{$X_{n-1}$}}  
\rput(47,8){\rnode{b3}{\white{$X_n$}}}  
\rput(53,8){\rnode{b4}{\white{$X_n$}}}  
\rput(70,8){\rnode{b5}{$X_0$}}  

\rput(50,8){$\cdots$}

\ncline[offset=3pt]{->}{b1}{b2}\naput{{\scriptsize $s$}}
\ncline[offset=-3pt]{->}{b1}{b2}\nbput{{\scriptsize $t$}}
\ncline[offset=3pt]{->}{b2}{b3}\naput{{\scriptsize $s$}}
\ncline[offset=-3pt]{->}{b2}{b3}\nbput{{\scriptsize $t$}}
\ncline[offset=3pt]{->}{b4}{b5}\naput{{\scriptsize $s$}}
\ncline[offset=-3pt]{->}{b4}{b5}\nbput{{\scriptsize $t$}}

\endpspicture
\]
in \cat{Set} satisfying $ss = st$ and $ts = tt$ \cite{bat1}.  Write
$n\cat{-GSet}$ for the category of $n$-globular sets. Clearly \cat{$n$-GSet} is the category of presheaves on a finite category
\[
\psset{labelsep=2pt}
\pspicture(0,4)(70,12)

\rput(3,8){\rnode{b1}{$x_n$}}  
\rput(25,8){\rnode{b2}{$x_{n-1}$}}  
\rput(47,8){\rnode{b3}{\white{$X_n$}}}  
\rput(53,8){\rnode{b4}{\white{$X_n$}}}  
\rput(70,8){\rnode{b5}{$x_0$}}  

\rput(50,8){$\cdots$}

\ncline[offset=3pt]{<-}{b1}{b2}\naput{{\scriptsize $$}}
\ncline[offset=-3pt]{<-}{b1}{b2}\nbput{{\scriptsize $$}}
\ncline[offset=3pt]{<-}{b2}{b3}\naput{{\scriptsize $$}}
\ncline[offset=-3pt]{<-}{b2}{b3}\nbput{{\scriptsize $$}}
\ncline[offset=3pt]{<-}{b4}{b5}\naput{{\scriptsize $$}}
\ncline[offset=-3pt]{<-}{b4}{b5}\nbput{{\scriptsize $$}}

\endpspicture
\]
(with equations dual to those above)
and there is an evident
truncation functor
\[
n\cat{-GSet} \lra (n-1)\cat{-GSet},
\]
for each $n \geq 1$.  

\begin{proposition} %3.4
\label{prop:ngph-ngset}

For all $n \geq 0$, there is a canonical equivalence of categories
\[
n\cat{-Gph} \catequiv n\cat{-GSet},
\]
and these equivalences of categories commute with the truncation functors.
\end{proposition}

\begin{proof}
See the proof of Proposition 1.4.9 of~\cite{lei8}.
\end{proof}

By contrast with $\omg$-graphs, it is straightforward to make an
\omgadj-dimensional version of $n\cat{-GSet}$ directly; this is the
category \cat{GSet} of globular sets.

\begin{proposition} %3.6
\label{prop:gph-gset}

There is a canonical equivalence of categories
\[
\omg\cat{-Gph} \catequiv \cat{GSet}.
\]
Under this equivalence and those of Proposition~\ref{prop:ngph-ngset}, the
truncation functor $\tr_n$
% \[
% \tr_n: \omg\cat{-Gph} \lra n\cat{-Gph}
% \]
corresponds to the evident truncation functor $\cat{GSet} \lra
n\cat{-GSet}$.   
\end{proposition}

\begin{proof}
See, for instance, Section~F.2 of~\cite{lei8}.
\end{proof}

The fact that this is an equivalence and not an isomorphism comes down to a
slightly technical but nonetheless crucial point: the comparison functor 
\[\omg\cat{-Gph} \lra \cat{GSet}\]
involves moving from a structure with locally defined sets of $k$-cells, to one where a single set of cells is given for each dimension.  This means making some arbitrary choices of
coproducts at each dimension $k$, and thus the
composite
\[\cat{GSet} \lra \omg\cat{-Gph} \lra \cat{GSet}\]
is destined to be only isomorphic, and not equal, to the identity.  This means
that \cat{GSet} is only equivalent to a limit of the sequential diagram, and
is not itself a limit.  As a consequence, \cat{GSet} is \emph{not} a terminal coalgebra for $\FG$.

One remedy would be to view $\CAT$ as a 2-category and consider weak 2-dimensional limits, but we have made the choice in this
paper to stick to 1-categorical endofunctors and their terminal coalgebras,
constructed as 1-dimensional limits.  Thus terminal
coalgebras are unique up to isomorphism, not just equivalence.   For this reason, while it is often standard to use $n$-graphs and $n$-globular sets interchangeably, we shall take care not to do so in this work.

%%%%%%%%%%%%%%%%%%%%%%%%%% end of extra gset comments

This completes our analysis of $\omg\cat{-Gph}$; we are now ready to study
$\omg$-categories.

%%%%%%%%%%%%%%%%%%%%%%%%%%%%%%%%%%%%%%%%%%%%%%%%%%%%%%%%%%%%%%%%%%%%%%%%%
%%%%%%%%%%%%%%%%%%%%%%%%%%%%%%%%%%%%%%%%%%%%%%%%%%%%%%%%%%%%%%%%%%%%%%%%%
% Section: Strict omega-categories
%%%%%%%%%%%%%%%%%%%%%%%%%%%%%%%%%%%%%%%%%%%%%%%%%%%%%%%%%%%%%%%%%%%%%%%%%
%%%%%%%%%%%%%%%%%%%%%%%%%%%%%%%%%%%%%%%%%%%%%%%%%%%%%%%%%%%%%%%%%%%%%%%%%

\section{The category of strict $\omg$-categories}
\label{sec:strictcat}\label{strictcat}\label{sec:strict}

In this section we will study the construction of strict $\omg$-categories
via limits, and hence via terminal coalgebras.  Most of the material in this
section is not new, but we give it the emphasis we need in order to motivate
and illuminate the generalisation to the weak case. 

The basic idea is that we have for each $n \in \bb{N} \cup \{\omg\}$ a
category $\cat{Str-}n\cat{-Cat}$ of strict $n$-categories and strict
$n$-functors.  Then we have the following commuting diagram of truncation
functors, building on the diagram for graphs given just after Definition~\ref{omegagraphdef}
%standard limit diagram
\[
\psset{unit=0.095cm,labelsep=2pt,nodesep=2pt}
\pspicture(114,24)

%   a
% x0 x1 x2 x3 x4

\rput(20,20){\rnode{a}{\cat{Str-$\omg$-Cat}}}

\rput(0,0){$\cdots$}

\rput(0,0){\rnode{x0}{\white{$\DG^3\1$}}}
\rput(25,0){\rnode{x1}{$\cat{Str-$n$-Cat}$}}
\rput(55,0){\rnode{x2}{$\cat{Str-$(n-1)$-Cat}$}}
\rput(85,0){\rnode{x3}{$\cdots$}}
\rput(110,0){\rnode{x4}{$\cat{Str-$0$-Cat}$}}

\psset{nodesep=2pt}
\ncline{->}{x0}{x1} \naput{{\scriptsize $$}}
\ncline{->}{x1}{x2} \naput{{\scriptsize $$}}
\ncline[nodesepB=12pt]{->}{x2}{x3} \naput{{\scriptsize $$}}
\ncline[nodesepA=12pt]{->}{x3}{x4} \naput{{\scriptsize $$}}

\ncline{->}{a}{x0} \naput{{\scriptsize $$}}
\ncline{->}{a}{x1} \naput{{\scriptsize $$}}
\ncline{->}{a}{x2} \naput{{\scriptsize $$}}
\ncline{->}{a}{x4} \naput{{\scriptsize $$}}

\endpspicture
\]
Moreover, this is a limit cone in \cat{Cat}.

The precise story is a little more subtle, largely because we must take care over the difference between $n$-globular sets and $n$-graphs, as in the previous section.  Thus we have  the following
two distinct (but equivalent) ways of defining strict $\omg$-categories. 

\begin{enumerate}

\item Directly: a strict $n$-category is an $n$-globular set equipped with
composition and identities satisfying certain axioms. 

\item By induction: $\cat{Str-}0\cat{-Cat}= \cat{Set}$, and for $n \geq 1$
\[\cat{Str-}(n+1)\cat{-Cat} = (\cat{Str-}n\cat{-Cat})\cat{-Cat}.\] 

\end{enumerate}

\noindent The resulting categories of strict $n$-categories are equivalent for
all finite $n$, and we are accustomed to using these definitions somewhat interchangeably.  However for the \omgadj-dimensional case only the first
definition works \emph{a priori}; for the version by iterated enrichment, we must \emph{define}
$\cat{Str-}\omg\cat{-Cat}$ as the above limit, whereas in the first case the
fact that it is a limit is a theorem that follows from the direct (non-inductive) definition.

Standard results show that $\cat{Str-}n\cat{-Cat}$ is monadic over the
category $n\cat{-GSet}$ of $n$-dimensional globular sets or $n\cat{-Gph}$ for each $n$, and in the next section we perform the above
constructions at the level of monads rather than their algebras.

As discussed in the introduction, throughout this section we will need
finite products.  Recall from Example~\ref{eg:Simpson} that $\CATp$ denotes
the category of locally small categories with finite products, and functors
preserving them.  Observe that if $\cV$ has finite products then
$\cV\cat{-Gph}$ does too.  This extends to the following result.

\begin{lemma}
\label{lemma:FG-restricts}
$\FG$ restricts to an endofunctor of $\CATp$.
\proofbox
\end{lemma}

We now define the enriched category endofunctor that is the object of study in this section.

\begin{mydefinition} %4.1 
\label{FC}

We write $\FC$ for the endofunctor 
\[(\ )\cat{-Cat} \: \CATp \lra \CATp\] 
described in Example~\ref{eg:Simpson}. 

\end{mydefinition}

In this section, ``$n$-category'' means \emph{strict} $n$-category, and
similarly $n\cat{-Cat}$ denotes the category of \emph{strict} $n$-categories, defined iteratively as follows; this is analogous to our iterative definition of $\omg$-graphs.

\begin{mydefinition} %4.3

For all $n \geq 0$ we define the category $n\cat{-Cat}$ of
(small, strict) \demph{$n$-categories} as follows: 

\begin{itemize}

\item $0\cat{-Cat} = \cat{Set}$, and
\item for all $n \geq 1$, $n\cat{-Cat} = \FC\big((n-1)\cat{-Cat}\big) =
\big((n-1)\cat{-Cat}\big)\cat{-Cat}$. 

\end{itemize}

\noi As before, we also define for all $n \geq 1$ the \demph{truncation} functor $U_n:
n\cat{-Cat} \lra (n-1)\cat{-Cat}$, as follows. 

\begin{itemize}

\item $U_1 \: 1\cat{-Cat} \lra \cat{Set}$ is the functor that sends a category
to its set of objects, and 
\item for all $n \geq 2$, $U_n = \FC(U_{n-1}) = ({U_{n-1}})_\ast$.

\end{itemize}

\end{mydefinition}

We now give the definition of the category of $\omg$-categories as a limit of the $n$-dimensional versions. 

\begin{mydefinition} %4.4
\label{newfourpointfour}\label{omegacatdef}

The category $\omg\cat{-Cat}$ of (small, strict) \demph{$\omg$-categories}
is the limit of the diagram 
\[ 
\cdots 
\map{U_3} 2\cat{-Cat} 
\map{U_2} 1\cat{-Cat} 
\map{U_1} 0\cat{-Cat} 
\]
in $\CAT$.  By abuse of notation we write $\tr_n\: \omg\cat{-Cat} \lra n\cat{-Cat}$ for the
$n$th projection of this limit cone (as well as the one for $\omg$-graphs), and refer to this version of $\tr_n$ and $U_n$ as
\demph{truncation}.  
\end{mydefinition}

Thus we may express a strict $\omg$-category as an infinite sequence
\[(\ldots, X_2, X_1, X_0)\]
where each $X_n$ is in $\cat{$n$-Cat}$ and is the truncation of $X_{n+1}$.

In the next section we prove that $\omg\cat{-Cat}$ is the terminal
coalgebra for $\FC$; we will deduce it from the analogous result for monads.  Given Ad\'amek's Theorem and the definition of
$\omg\cat{-Cat}$ as a limit, the following lemma provides a strong hint towards the result.

\begin{lemma} %4.5
\label{lemma:FC-ladder}

There are canonical isomorphisms $(I_n)_{n \geq 0}$ such that the following
diagram commutes:
\[
\psset{unit=0.1cm,labelsep=2pt,nodesep=2pt}
\pspicture(80,35)

% x1 x2 x3
% y1 y2 y3

\rput(0,30){$\cdots$}
\rput(0,10){$\cdots$}

\rput(0,30){\rnode{x0}{\white{$\DG^3\1$}}}
\rput(30,30){\rnode{x1}{$\FC^3\1$}}
\rput(55,30){\rnode{x2}{$\FC^2\1$}}
\rput(80,30){\rnode{x3}{$\FC\1$}}

\psset{nodesep=4pt}
\ncline{->}{x0}{x1} \naput{{\scriptsize $\FC^3!$}}
\ncline{->}{x1}{x2} \naput{{\scriptsize $\FC^2!$}}
\ncline{->}{x2}{x3} \naput{{\scriptsize $\FC!$}}

\rput(3,10){\rnode{y0}{\white{$\DG^3\1$}}}

\rput(30,10){\rnode{y1}{$\cat{2-Cat}$}}
\rput(55,10){\rnode{y2}{$\cat{1-Cat}$}}
\rput(80,10){\rnode{y3}{$\cat{0-Cat}$}}

\psset{nodesep=2pt}
\ncline{->}{y0}{y1} \nbput{{\scriptsize $U_3$}}
\ncline{->}{y1}{y2} \nbput{{\scriptsize $U_2$}}
\ncline{->}{y2}{y3} \nbput{{\scriptsize $U_1$}}

\psset{nodesepA=5pt,nodesepB=3pt,labelsep=2pt}
\ncline{->}{x1}{y1} \naput{{\scriptsize $I_2$}} \nbput{{\scriptsize $\iso$}}
\ncline{->}{x2}{y2} \naput{{\scriptsize $I_1$}} \nbput{{\scriptsize $\iso$}}
\ncline{->}{x3}{y3} \naput{{\scriptsize $I_0$}} \nbput{{\scriptsize $\iso$}}

\endpspicture
\]
\end{lemma}

\begin{proof}
As in the proof of Theorem~\ref{thm:gph-tc}.
\end{proof}

\begin{theorem} %4.6
\label{thm:strict-cat-tc}

The category $\omg\cat{-Cat}$ of strict $\omg$-categories is the terminal
coalgebra for $\FC$.

\end{theorem}

This can be proved directly in a manner similar to that for $\omg$-graphs (by showing that it is a limit in \CATp\ and $\FC$ preserves the limit);
however we defer the proof as we will be able to deduce it from results on
monads in the next section.  

\begin{remark}

The category of $n$-categories obtained non-inductively (as $n$-globular sets with extra structure) is equivalent to
$n\cat{-Cat}$; see Proposition~1.4.9 of~\cite{lei8}, for instance.  Likewise the category of $\omg$-categories defined as globular sets with extra structure is equivalent to $\omg\cat{-Cat}$; see for instance Proposition~1.4.12 of~\cite{lei8}.

Note that as in the previous section, these are equivalences and not isomorphisms, so we should not expect the category of non-inductively defined $\omg$-categories to be a terminal coalgebra for $\FC$. 

\end{remark}

%%%%%%%%%%%%%%%%%%%%%%%%%%%%%%%%%%%%%%%%%%%%%%%%%%%%%%%%%%%%%%%%%%%%%%%%%
%%%%%%%%%%%%%%%%%%%%%%%%%%%%%%%%%%%%%%%%%%%%%%%%%%%%%%%%%%%%%%%%%%%%%%%%%
% Section: The monad for strict omega-categories
%%%%%%%%%%%%%%%%%%%%%%%%%%%%%%%%%%%%%%%%%%%%%%%%%%%%%%%%%%%%%%%%%%%%%%%%%
%%%%%%%%%%%%%%%%%%%%%%%%%%%%%%%%%%%%%%%%%%%%%%%%%%%%%%%%%%%%%%%%%%%%%%%%%

\section{The monad for strict $\omg$-categories}
\label{sec:strictmonad}\label{strictmnd}\label{four}

It is well known that strict $n$-categories are the algebras for a certain
monad, and in fact there is a more abstract way of performing the
constructions of the previous section, using the \emph{monad} for $n$-categories rather than the
\emph{category} of $n$-categories.  This will be useful in
Section~\ref{sec:weak}, where we study structures that are \emph{only} defined
as algebras for a monad.  

In this section we will see that there is a forgetful functor
\[\omg\cat{-Cat} \lra \omg\cat{-Gph}\]
that it is monadic, that the induced monad is the limit of the monads for $n$-categories, and hence, that it is the terminal coalgebra for a certain endofunctor. 

As discussed in the introduction will need to impose some conditions on the categories and monads involved.  To guarantee the existence of a free enriched category monad we demand finite products and small coproducts that interact well, as follows.

\begin{mydefinition}	\label{defn:inf-dist} %2.4

An \demph{infinitely distributive category} is a category with finite products
and small coproducts in which the former distribute over the latter. We write
\CATd\ for the subcategory of \CAT\ whose objects are the locally small
infinitely distributive categories and whose morphisms are functors preserving
the finite products and small coproducts. 

\end{mydefinition}

Note that this property is inherited by enrichment: if $\cV$ is infinitely distributive then $\cV\cat{-Gph}$ and $\cV\cat{-Cat}$ are too.  This extends to the following result.

\begin{lemma}
The endofunctors $\FG$ and $\FC$ restrict to endofunctors of $\CATd$.
\proofbox
\end{lemma}

We must also make precise some background on monads.

\begin{mydefinition} %2.1
\label{defn:MND}
We write \MND\ for the category of monads and lax morphisms of monads (as in
\cite{str1}, where lax morphisms of monads are called monad functors).
\end{mydefinition}

Thus, an object of \MND\ is a pair $(\cl{V}, T)$ where $\cl{V} \in \CAT$
and $T$ is a monad on $\cl{V}$,
and a morphism $(\cl{V}, T) \lra (\cl{V}', T')$ is a pair $(H, \theta)$ where
$H\: \cl{V} \lra \cl{V}'$ is a functor and $\theta$ is a natural transformation 
\[\psset{unit=0.08cm,labelsep=0pt,nodesep=3pt}
\pspicture(20,22)

% a1 a2
% b1 b2

\rput(0,19){\rnode{a1}{$\cV$}} % top left
\rput(20,19){\rnode{a2}{$\cV'$}} % top right

\rput(0,1){\rnode{b1}{$\cV$}}   % bottom left
\rput(20,1){\rnode{b2}{$\cV'$}}  % bottom right

\psset{nodesep=3pt,labelsep=2pt,arrows=->}
\ncline{a1}{a2}\naput{{\scriptsize $H$}} % top
\ncline{b1}{b2}\nbput{{\scriptsize $H$}} % bottom
\ncline{a1}{b1}\nbput{{\scriptsize $T$}} % left
\ncline{a2}{b2}\naput{{\scriptsize $T'$}} % right

\psset{labelsep=1.5pt}
\pnode(13,13){a3}
\pnode(7,7){b3}
\ncline[doubleline=true,arrowinset=0.6,arrowlength=0.8,arrowsize=0.5pt 2.1]{a3}{b3} \nbput[npos=0.4]{{\scriptsize $\theta$}}

\endpspicture\]
satisfying the coherence axioms in \cite{str1}.  Such a morphism is called
\demph{weak} if $\theta$ is an isomorphism, and we write $\MNDwk$ for the
subcategory of $\MND$ consisting of all objects but only the weak morphisms; we will need this restricted category for technical reasons later.

We will sometimes use an alternative point of view on monad morphisms.
By the results in~\cite{str1}, a lax morphism of monads $(\cl{V}, T) \lra
(\cl{V}', T')$ can equivalently be described as a pair of functors $H,K$ making the following square commute
\[\psset{unit=0.08cm,labelsep=0pt,nodesep=3pt}
\pspicture(20,22)

% a1 a2
% b1 b2

\rput(0,19){\rnode{a1}{$\cV^T$}} % top left
\rput(20,19){\rnode{a2}{${\cV'}^{T'}$}} % top right

\rput(0,1){\rnode{b1}{$\cV$}}   % bottom left
\rput(20,1){\rnode{b2}{$\cV'$}}  % bottom right

\psset{nodesep=3pt,labelsep=2pt,arrows=->}
\ncline{a1}{a2}\naput{{\scriptsize $K$}} % top
\ncline{b1}{b2}\nbput{{\scriptsize $H$}} % bottom
\ncline{a1}{b1}\nbput{{\scriptsize $G^T$}} % left
\ncline{a2}{b2}\naput{{\scriptsize ${G'}^{T'}$}} % right

\endpspicture\]
where the vertical arrows are the forgetful functors.

\begin{myremark} %2.2

Note that the category $\MND$ could be made into a 2-category using monad
transformations, but we prefer to stay at the 1-categorical level and use the
theory of terminal coalgebras for endofunctors rather than develop the
2-dimensional version of that theory.
\end{myremark}

\begin{definition}

We write \MNDd\ for the subcategory of \MND\ given as follows:

\begin{itemize}

\item objects are pairs $(\cl{V}, T)$ where $\cl{V} \in \CATd$ and $T$ is a
monad on $\cl{V}$ whose functor part preserves small coproducts 

\item a morphism $(\cl{V},T) \lra (\cl{V}', T')$ is a lax morphism of monads
$(H, \theta)$ such that the functor $H$ preserves finite products and small coproducts.
\end{itemize}

\end{definition}

Note that for an object $(\cl{V}, T)$ of \MNDd, the functor part of $T$ is
\emph{not} required to preserve finite products.

We now define the crucial endofunctor on \MNDd; this can be thought of as the ``monad version'' of the previous endofunctor $\FC = (\ )\cat{-Cat}$, in a
sense we will make precise using the functor 
\[\Alg \: \MND \tra \CAT\]  
that sends a monad to its category of algebras. The idea is that given a monad $T$ on $\cV$, our endofunctor produces the monad on $\cV\cat{-Gph}$ for ``categories enriched in $T$-algebras'' induced by the forgetful functor
\[\cl{V}^T\cat{-Cat} \lra \cl{V}\cat{-Gph}.\]
The next two propositions construct this monad.

\begin{proposition}[Kelly {\cite[Theorem 23.4]{kel4}}] %4.8
\label{prop:kelly}\label{fourpointeight} Let \cl{V} be an infinitely distributive category.  Then the forgetful functor
\[\cl{V}\cat{-Cat} \lra \cl{V}\cat{-Gph}\]
is monadic.  The induced monad is the ``free \cl{V}-category monad''
$\fcv$, which is the identity on underlying sets of objects, and
is given on hom-sets as follows: for a \cl{V}-graph $A$ and objects $a, a' \in
A$, 
\[(\fcv A)(a,a') = 
\coprod_{k \in \mathbb{N}, a=a_0, a_1, \ldots, a_{k-1}, a_k=a'}
A(a_{k-1},a_k) \times \cdots \times A(a_0,a_1).\] 

\end{proposition}

Note that in the case $\cV = \Set$ the monad $\fcv$ is the free category monad on $\Gph$, which we write as $\cat{fc}$.

Now consider a monad $T$ on an infinitely distributive category \cl{V}.
Since $\FG$ extends to a 2-functor on $\CAT$ (as remarked after
Definition~\ref{defn:FG}), there is an induced monad $\FG(T) = T_*$ on
$\FG(\cl{V}) = \cl{V}\cat{-Gph}$.  Hence $\cl{V}\cat{-Gph}$ carries two
monads, $\fcv$ and $T_*$.

\numroman
\begin{proposition} %4.9
\label{prop:fc-T}\label{fourpointnine}

Let $(\cl{V}, T) \in \MNDd$.  Then:
\begin{enumerate}
\item The diagonal of the commutative square of forgetful functors
\[
\psset{unit=0.1cm,labelsep=3pt,nodesep=3pt}
\pspicture(25,23)

% a1 a2
% a3 a4

%%%%%%%%%% top

\rput(0,20){\rnode{a1}{\cat{$\cV^T$-Cat}}}  % top left
\rput(25,20){\rnode{a2}{\cat{$\cV^T$-Gph}}}  % top right
\rput(0,0){\rnode{a3}{\cat{$\cV$-Cat}}}  % bottom left
\rput(25,0){\rnode{a4}{\cat{$\cV$-Gph}}}  % bottom right

\ncline{->}{a1}{a2} \naput{{\scriptsize $$}} % top
\ncline{->}{a3}{a4} \nbput{{\scriptsize $$}} % bottom
\ncline{->}{a1}{a3} \nbput{{\scriptsize $$}} % left
\ncline{->}{a2}{a4} \naput{{\scriptsize $$}} % right

\endpspicture
\]

is monadic.

\item There is a canonical distributive law of $T_*$ over $\fcv$, and the
resulting monad $\fcv \circ T_*$ on $\cl{V}\cat{-Gph}$ is the monad induced by
this diagonal functor.

\item $(\cl{V}\cat{-Gph}, \fcv \circ T_*) \in \MNDd$.

\end{enumerate}

\end{proposition}

\begin{proof}
For parts (i) and (ii) see Proposition~F.1.1 of~\cite{lei8}.  There it is assumed that \cl{V} is a
presheaf category, but the proof needs only that \cl{V} is infinitely
distributive.  The last part is straightforward to check.
\end{proof}

This gives the action of our endofunctor on objects; the full definition is as follows.

\begin{mydefinition} %4.10
\label{FM}\label{fourpointten}

We define an endofunctor \[\FM\: \MNDd\ \lra \MNDd\] as follows.

\begin{itemize}

\item On objects, $\FM(\cl{V},T) = \big(\cl{V}\cat{-Gph}, \ \cat{fc}_{\cl{V}}
\circ T_*\big)$.

\item On morphisms, a lax morphism of monads
\[\psset{unit=0.08cm,labelsep=0pt,nodesep=3pt}
\pspicture(20,22)

% a1 a2
% b1 b2

\rput(0,20){\rnode{a1}{$\cV$}} % top left
\rput(20,20){\rnode{a2}{$\cV'$}} % top right

\rput(0,0){\rnode{b1}{$\cV$}}   % bottom left
\rput(20,0){\rnode{b2}{$\cV'$}}  % bottom right

\psset{nodesep=3pt,labelsep=2pt,arrows=->}
\ncline{a1}{a2}\naput{{\scriptsize $H$}} % top
\ncline{b1}{b2}\nbput{{\scriptsize $H$}} % bottom
\ncline{a1}{b1}\nbput{{\scriptsize $T$}} % left
\ncline{a2}{b2}\naput{{\scriptsize $T'$}} % right

\psset{labelsep=1.5pt}
\pnode(13,13){a3}
\pnode(7,7){b3}
\ncline[doubleline=true,arrowinset=0.6,arrowlength=0.8,arrowsize=0.5pt 2.1]{a3}{b3} \nbput[npos=0.4]{{\scriptsize $\theta$}}

\endpspicture\]
is mapped to the composite
\[\psset{unit=0.1cm,labelsep=0pt,nodesep=3pt}
\pspicture(0,-4)(23,44)

% a1 a2
% b1 b2
% c1 c2

\rput(-1,40){\rnode{a1}{\cat{$\cV$-Gph}}} % top left
\rput(21,40){\rnode{a2}{\cat{$\cV'$-Gph}}} % top mid

\rput(-1,20){\rnode{b1}{\cat{$\cV$-Gph}}}   % bottom left
\rput(21,20){\rnode{b2}{\cat{$\cV'$-Gph}}}  % bottom mid

\rput(-1,0){\rnode{c1}{\cat{$\cV$-Gph}}}   % bottom left
\rput(21,0){\rnode{c2}{\cat{$\cV'$-Gph}}}  % bottom mid

\psset{nodesep=3pt,labelsep=2pt,arrows=->}
\ncline{a1}{a2}\naput{{\scriptsize $H_*$}} % top
\ncline{b1}{b2}\nbput{{\scriptsize $H_*$}} % mid
\ncline{c1}{c2}\nbput{{\scriptsize $H_*$}} % bottom

\ncline{a1}{b1}\nbput{{\scriptsize $T_*$}} % left
\ncline{b1}{c1}\nbput{{\scriptsize $\cat{fc}_{\cV}$}} % left

\ncline{a2}{b2}\naput{{\scriptsize $T'_*$}} % right
\ncline{b2}{c2}\naput{{\scriptsize $\cat{fc}_{\cV'}$}} % right

\psset{labelsep=1.5pt,nodesep=4pt}

\pnode(13,33){a3}
\pnode(7,27){b3}
\ncline[doubleline=true,arrowinset=0.6,arrowlength=0.8,arrowsize=0.5pt 2.1]{a3}{b3} \naput[npos=0.4]{{\scriptsize $\theta_*$}}

\pnode(13,13){a3}
\pnode(7,7){b3}
\ncline[doubleline=true,arrowinset=0.6,arrowlength=0.8,arrowsize=0.5pt 2.1]{a3}{b3} \naput[npos=0.4]{{\scriptsize $\iso$}}

\endpspicture\]
where the bottom square is the isomorphism induced by $H$ preserving
products and coproducts. 

\end{itemize}

\end{mydefinition}

As noted after Definition~\ref{defn:MND}, a lax morphism of monads $(\cl{V},
T) \lra (\cl{V}', T')$ may be viewed as a commutative square
\[
\psset{unit=0.1cm,labelsep=3pt,nodesep=3pt}
\pspicture(20,25)

% a1 a2
% a3 a4

%%%%%%%%%% top

\rput(0,20){\rnode{a1}{$\cV^T$}}  % top left
\rput(20,20){\rnode{a2}{${\cV'}^{T'}$}}  % top right
\rput(0,0){\rnode{a3}{$\cV$}}  % bottom left
\rput(20,0){\rnode{a4}{$\cV'$}}  % bottom right

\ncline{->}{a1}{a2} \naput{{\scriptsize $K$}} % top
\ncline{->}{a3}{a4} \nbput{{\scriptsize $H$}} % bottom
\ncline{->}{a1}{a3} \nbput{{\scriptsize $G^T$}} % left
\ncline{->}{a2}{a4} \naput{{\scriptsize ${G'}^{T'}$}} % right

\endpspicture
\]
When monad morphisms are viewed in this way, the image of such a morphism
under $\FM$ is the commutative square
\[
\psset{unit=0.1cm,labelsep=3pt,nodesep=3pt}
\pspicture(0,-4)(25,24)

% a1 a2
% a3 a4

%%%%%%%%%% top

\rput(0,20){\rnode{a1}{\cat{$\cV^T$-Cat}}}  % top left
\rput(25,20){\rnode{a2}{\cat{${\cV'}^{T'}$-Cat}}}  % top right
\rput(0,0){\rnode{a3}{\cat{$\cV$-Gph}}}  % bottom left
\rput(25,0){\rnode{a4}{\cat{$\cV'$-Gph}}}  % bottom right

\ncline{->}{a1}{a2} \naput{{\scriptsize $K_*$}} % top
\ncline{->}{a3}{a4} \nbput{{\scriptsize $H_*$}} % bottom
\ncline{->}{a1}{a3} \nbput{{\scriptsize $$}} % left
\ncline{->}{a2}{a4} \naput{{\scriptsize $$}} % right

\endpspicture
\]
where the vertical arrows are the forgetful functors; sometimes this point of view will make our constructions easier.

At the end of this section the results we prove about $\FM$ will be transferred to results about $\FC$ via the functor $\Alg$ as below.

\begin{proposition}[Street {\cite[Theorems~1 and~7]{str1}}] \label{prop:FTM-adjns} %2.3 
\label{twopointthree}

There is a chain of adjunctions
\[
\psset{unit=0.11cm,labelsep=2pt,nodesep=3pt}
\pspicture(0,20)

% a2
% a1

\rput(0,0){\rnode{a1}{$\CAT$}}  % named node, with something placed there
\rput(0,20){\rnode{a2}{$\MND$}}  % named node, with something placed there

\ncline[nodesepA=3pt,nodesepB=3pt]{->}{a1}{a2}\ncput*[labelsep=0pt]{{\scriptsize $\cat{Inc}$}} % right

\nccurve[angleA=215,angleB=145,ncurv=0.9]{->}{a2}{a1}\nbput{{\scriptsize $\Und$}} \naput[labelsep=5.5pt]{{\scriptsize $\ladj$}}
\nccurve[angleA=-35,angleB=35,ncurv=0.9]{->}{a2}{a1}\naput{{\scriptsize $\Alg$}} \nbput[labelsep=5.5pt]{{\scriptsize $\ladj$}}

\endpspicture
\]
\noi where

\[\begin{array}[t]{ccccc}
\Und &:& \MND &\lra& \CAT\\[0pt]
&& (\cl{V}, T) & \mapsto & \cl{V} \\[10pt]
%&& (U, \alpha) & \mapsto & U
\cat{Inc} &:& \CAT & \lra & \MND \\[0pt]
&& \cl{V} & \mapsto & (\cl{V}, 1) \\[10pt]
%&& U & \mapsto & (U, \cat{id})
\Alg &:& \MND & \lra & \CAT \\[0pt]
&& (\cl{V},T) & \mapsto & \ \cl{V}^T. 
\end{array}\]

\end{proposition}

The proof of the following proposition can be found in the appendix.

\begin{proposition} %2.6
\label{prop:FTM-dist}\label{twopointsix}

The chain of adjunctions in Proposition~\ref{prop:FTM-adjns}
restricts to a chain of adjunctions
\[
\psset{unit=0.11cm,labelsep=2pt,nodesep=3pt}
\pspicture(0,22)

% a2
% a1

\rput(0,0){\rnode{a1}{$\CATd$}}  % named node, with something placed there
\rput(0,20){\rnode{a2}{$\MNDd$}}  % named node, with something placed there

\ncline[nodesepA=3pt,nodesepB=3pt]{->}{a1}{a2}\ncput*[labelsep=0pt]{{\scriptsize $\cat{Inc}$}} % right

\nccurve[angleA=215,angleB=145,ncurv=0.9]{->}{a2}{a1}\nbput{{\scriptsize $\Und$}} \naput[labelsep=5.5pt]{{\scriptsize $\ladj$}}
\nccurve[angleA=-35,angleB=35,ncurv=0.9]{->}{a2}{a1}\naput{{\scriptsize $\Alg$}} \nbput[labelsep=5.5pt]{{\scriptsize $\ladj$}}

\endpspicture
\]

\end{proposition}

The endofunctor $\FM$ is a ``lift'' of the endofunctors $\FG$ and $\FC$, in the following sense; this will enable us to deduce the results about $\FC$ from those about $\FM$.

\begin{lemma} %4.12
\label{lemma:comparing-Fs}\label{fourpointtwelve}

The squares
\[
\psset{unit=0.1cm,labelsep=3pt,nodesep=3pt}
\pspicture(0,-4)(24,25)

% a1 a2
% a3 a4

%%%%%%%%%% top

\rput(0,20){\rnode{a1}{$\MNDd$}}  % top left
\rput(24,20){\rnode{a2}{$\MNDd$}}  % top right
\rput(0,0){\rnode{a3}{$\CATd$}}  % bottom left
\rput(24,0){\rnode{a4}{$\CATd$}}  % bottom right

\ncline{->}{a1}{a2} \naput{{\scriptsize $\FM$}} % top
\ncline{->}{a3}{a4} \nbput{{\scriptsize $\FG$}} % bottom
\ncline{->}{a1}{a3} \nbput{{\scriptsize $\Und$}} % left
\ncline{->}{a2}{a4} \naput{{\scriptsize $\Und$}} % right

\endpspicture
\hh{6em}
\pspicture(0,-4)(24,25)

% a1 a2
% a3 a4

%%%%%%%%%% top

\rput(0,20){\rnode{a1}{$\MNDd$}}  % top left
\rput(24,20){\rnode{a2}{$\MNDd$}}  % top right
\rput(0,0){\rnode{a3}{$\CATd$}}  % bottom left
\rput(24,0){\rnode{a4}{$\CATd$}}  % bottom right

\ncline{->}{a1}{a2} \naput{{\scriptsize $\FM$}} % top
\ncline{->}{a3}{a4} \nbput{{\scriptsize $\FC$}} % bottom
\ncline{->}{a1}{a3} \nbput{{\scriptsize $\Alg$}} % left
\ncline{->}{a2}{a4} \naput{{\scriptsize $\Alg$}} % right

\endpspicture
\]
commute, the first strictly and the second up to a canonical isomorphism.

\end{lemma}

\begin{proof}
The first assertion is trivial.  For the second: by the lower-left leg we have 
\[ (\cl{V},T) \ \mapsto \ \cl{V}^T \ \mapsto \ \cl{V}^T\cat{-Cat}.\]
By the upper-right leg we have
\[ (\cl{V},T) \ \mapsto \ \big(\cl{V}\cat{-Gph}, \ \fcv \circ T_*\big) \ \mapsto \
\cl{V}\cat{-Gph}^{\fcv \circ T_*}.\] 
Thus the result follows from Proposition~\ref{prop:fc-T}, part (ii).  
\end{proof}

We can now use $\FM$ to define the monads for $n$-categories.  This is analogous to our definitions of \cat{$n$-Gph} and \cat{$n$-Cat}.

\begin{definition} %new 
\label{newfourpointtwelve}

We define for each
$n \geq 0$ a monad $T_n$ on $n\cat{-Gph}$ as follows:
\begin{itemize}
\item $T_0$ is the identity monad on $\cat{Set} = 0\cat{-Gph}$

\item for all $n \geq 1$, $(n\cat{-Gph}, T_n) = \FM\big((n-1)\cat{-Gph}, \
T_{n-1}\big)$. 
\end{itemize}
By Proposition~\ref{fourpointtwelve} and induction, we have
\[\Alg(T_n) \iso \cat{$n$-Cat} \]
for each $n$.  Further, since the monads $T_n$ are in the image of $\FM$, they all preserve coproducts.

We also define for each $n \geq 1$ a lax morphism of monads
\[
(U_n, \gamma_n)\: 
(n\cat{-Gph}, T_n) \lra \big((n-1)\cat{-Gph}, T_{n-1}\big)
\]
in $\MNDd$ as follows:
\begin{itemize}
\item $(U_1, \gamma_1)\: (1\cat{-Gph}, \fc{}) \lra (\cat{Set},
1)$ is the monad morphism corresponding to the objects functor
$\cat{ob}\: \cat{Cat} \lra \cat{Set}$.

\item For all $n \geq 2$, $(U_n, \gamma_n) = \FM\left(U_{n-1}, \gamma_{n-1}\right)$.
\end{itemize}
Note that each $U_n \: n\cat{-Gph} \lra (n-1)\cat{-Gph}$ is the truncation functor as in Definition~\ref{defn:nGph}.

\end{definition}

Now when morphisms of monads are viewed as commutative squares, the diagram 
\begin{equation}\label{diagone} 
\cdots 
\map{(U_3, \gamma_3)} (2\cat{-Gph}, T_2) 
\map{(U_2, \gamma_2)} (1\cat{-Gph}, T_1) 
\map{(U_1, \gamma_1)} (0\cat{-Gph}, T_0) 
\end{equation}
becomes the following commutative diagram
\[
\psset{unit=0.1cm,labelsep=2pt,nodesep=2pt}
\pspicture(0,5)(80,35)

% x0 x1 x2 x3
% y0 y1 y2 y3

\rput(0,30){$\cdots$}
\rput(0,10){$\cdots$}

\rput(0,30){\rnode{x0}{\white{$\DG^3\1$}}}
\rput(30,30){\rnode{x1}{$\cat{$2$-Cat}$}}
\rput(55,30){\rnode{x2}{$\cat{$1$-Cat}$}}
\rput(80,30){\rnode{x3}{$\cat{$0$-Cat}$}}

\psset{nodesep=4pt}
\ncline{->}{x0}{x1} \naput{{\scriptsize $U_3$}}
\ncline{->}{x1}{x2} \naput{{\scriptsize $U_2$}}
\ncline{->}{x2}{x3} \naput{{\scriptsize $U_1$}}

\rput(0,10){\rnode{y0}{\white{$\DG^3\1$}}}

\rput(30,10){\rnode{y1}{$\cat{2-Gph}$}}
\rput(55,10){\rnode{y2}{$\cat{1-Gph}$}}
\rput(80,10){\rnode{y3}{$\cat{0-Gph}$}}

\psset{nodesep=2pt}
\ncline{->}{y0}{y1} \nbput{{\scriptsize $U_3$}}
\ncline{->}{y1}{y2} \nbput{{\scriptsize $U_2$}}
\ncline{->}{y2}{y3} \nbput{{\scriptsize $U_1$}}

\psset{nodesepA=5pt,nodesepB=3pt,labelsep=2pt}
\ncline{->}{x1}{y1} \naput{{\scriptsize $$}} \nbput{{\scriptsize $$}}
\ncline{->}{x2}{y2} \naput{{\scriptsize $$}} \nbput{{\scriptsize $$}}
\ncline{->}{x3}{y3} \naput{{\scriptsize $$}} \nbput{{\scriptsize $$}}

\endpspicture
\]
Hence there is an induced forgetful functor
\[
\omg\cat{-Cat} \lra \omg\cat{-Gph}.
\]

\begin{lemma} %4.13

The forgetful functor $\omg\cat{-Cat} \lra \omg\cat{-Gph}$ is
monadic.  

\end{lemma}

\begin{proof}
See Theorem~F.2.2 of~\cite{lei8}.  (The result follows from the monadicity of
the functors $n\cat{-Cat} \lra n\cat{-Gph}$, using the fact that the
truncations 
\[U_n\: n\cat{-Cat} \lra (n-1)\cat{-Cat}\] 
are isofibrations in the sense of \cite{lac3}.)
\end{proof}

\begin{definition}
Write $\Tom$ for the induced monad on $\omg\cat{-Gph}$; thus, $\Tom$ is
the monad for $\omg$-categories.  
\end{definition}

Eventually we will take the limit of the above sequential diagram (\ref{diagone}) and show that it gives the monad $\Tom$; for now we show that we have a cone as below.
\begin{equation}\label{diagc}
\psset{unit=0.095cm,labelsep=2pt,nodesep=2pt}
\pspicture(0,-5)(114,24)
% standard limit diagram
%   a
% x3 x2 x1 x0

\rput(20,20){\rnode{a}{$(\cat{$\omg$-Gph}, \Tom)$}}

\rput(0,0){$\cdots$}

\rput(0,0){\rnode{x3}{\white{$\DG^3\1$}}}
\rput(30,0){\rnode{x2}{$(\cat{$2$-Gph}, T_2)$}}
\rput(70,0){\rnode{x1}{$(\cat{$1$-Gph}, T_1)$}}
\rput(108,0){\rnode{x0}{$(\cat{$0$-Gph}, T_0)$}}

\psset{nodesep=2pt}
\ncline{->}{x3}{x2} \naput{{\scriptsize $$}}
\ncline{->}{x2}{x1} \nbput{{\scriptsize $({U_2}, \gamma_2)$}}
\ncline{->}{x1}{x0} \nbput{{\scriptsize $(U_1, \gamma_1)$}}

\ncline{->}{a}{x0} \naput[npos=0.65,labelsep=1pt]{{\scriptsize $(\mbox{tr}_0, \tau_0)$}}
\ncline{->}{a}{x1} \nbput[npos=0.55,labelsep=1pt]{{\scriptsize $(\mbox{tr}_1, \tau_1)$}}
\ncline{->}{a}{x2} \nbput[npos=0.65,labelsep=1pt]{{\scriptsize $(\mbox{tr}_2, \tau_2)$}}
\ncline{->}{a}{x3} \naput{{\scriptsize $$}}

\endpspicture
\end{equation}
For each $n \geq 0$, the truncation functors $\tr_n$ participate in a
commutative square
\[
\psset{unit=0.1cm,labelsep=3pt,nodesep=3pt}
\pspicture(0,-4)(25,24)

% a1 a2
% a3 a4

%%%%%%%%%% top

\rput(0,20){\rnode{a1}{\cat{$\omg$-Cat}}}  % top left
\rput(25,20){\rnode{a2}{\cat{$n$-Cat}}}  % top right
\rput(0,0){\rnode{a3}{\cat{$\omg$-Gph}}}  % bottom left
\rput(25,0){\rnode{a4}{\cat{$n$-Gph}}}  % bottom right

\ncline{->}{a1}{a2} \naput{{\scriptsize $\mbox{tr}_n$}} % top
\ncline{->}{a3}{a4} \nbput{{\scriptsize $\mbox{tr}_n$}} % bottom
\ncline{->}{a1}{a3} \nbput{{\scriptsize $$}} % left
\ncline{->}{a2}{a4} \naput{{\scriptsize $$}} % right

\endpspicture
\]
This gives the graph truncation functor $\tr_n$ the structure of a lax
morphism of monads, which we write as
\[
(\tr_n, \tau_n):
(\omg\cat{-Gph}, \Tom) \lra (n\cat{-Gph}, T_n).
\]

We now embark on the proof that the free $\omg$-category monad
$(\omg\cat{-Gph}, \Tom)$ is the terminal coalgebra for $\FM$.

\begin{theorem} %4.16
\label{thm:strict-monad-tc}\label{fourpointsixteen}

The strict $\omg$-category monad $(\omg\cat{-Gph}, \Tom)$ is the terminal
coalgebra for $\FM$.

\end{theorem}

%2

\begin{proof}

We use Ad\'amek's theorem.  We need to consider the limit in $\MNDd$ of
\begin{equation}\label{diagtwo}
\psset{unit=0.1cm,labelsep=2pt,nodesep=2pt}
\pspicture(0,-1)(80,5)

% x1 x2 x3

\rput(0,0){$\cdots$}

\rput(0,0){\rnode{x0}{\white{$\DG^3\1$}}}
\rput(30,0){\rnode{x1}{$\FM^3\1$}}
\rput(55,0){\rnode{x2}{$\FM^2\1$}}
\rput(80,0){\rnode{x3}{$\FM\1$}}

\psset{nodesep=4pt}
\ncline{->}{x0}{x1} \naput{{\scriptsize $\FG^3!$}}
\ncline{->}{x1}{x2} \naput{{\scriptsize $\FG^2!$}}
\ncline{->}{x2}{x3} \naput{{\scriptsize $\FG!$}}

\endpspicture
\end{equation}

\vv{1em}

\noi where $\1$ is the terminal object $(\1, 1)$ of $\MNDd$. (We also continue to write the terminal object of $\CAT$ as $\1$, and the identity monad as $1$.)  We need to show 
\begin{enumerate}
 \item the limit is $(\cat{$\omg$-Gph}, \Tom)$, and
\item the limit is preserved by $\FM$.
\end{enumerate}

First we note that the above diagram (\ref{diagtwo}) is isomorphic to our diagram of $n$-dimensional monads and truncations; formally, there are canonical isomorphisms $(I_n, \iota_n)_{n \geq 0}$ of monads, with
$I_n$ as in the analogous result for graphs (see proof of Theorem~\ref{thm:gph-tc}), such that the following diagram
commutes:
\[
\psset{unit=0.1cm,labelsep=2pt,nodesep=2pt}
\pspicture(0,5)(100,39)

% x1 x2 x3
% y1 y2 y3

\rput(0,35){$\cdots$}
\rput(0,10){$\cdots$}

\rput(0,35){\rnode{x0}{\white{$\DG^3\1$}}}
\rput(30,35){\rnode{x1}{$\FM^3\1$}}
\rput(65,35){\rnode{x2}{$\FM^2\1$}}
\rput(100,35){\rnode{x3}{$\FM\1$}}

\psset{nodesep=4pt}
\ncline{->}{x0}{x1} \naput{{\scriptsize $\FM^3!$}}
\ncline{->}{x1}{x2} \naput{{\scriptsize $\FM^2!$}}
\ncline{->}{x2}{x3} \naput{{\scriptsize $\FM!$}}

\rput(3,10){\rnode{y0}{\white{$\DG^3\1$}}}

\rput(30,10){\rnode{y1}{$(\cat{2-Gph},T_2)$}}
\rput(65,10){\rnode{y2}{$(\cat{1-Gph},T_1)$}}
\rput(100,10){\rnode{y3}{$(\cat{0-Gph},T_0)$}}

\psset{nodesep=2pt}
\ncline{->}{y0}{y1} \nbput{{\scriptsize $(U_3,\gamma_3)$}}
\ncline{->}{y1}{y2} \nbput{{\scriptsize $(U_2,\gamma_2)$}}
\ncline{->}{y2}{y3} \nbput{{\scriptsize $(U_1,\gamma_1)$}}

\psset{nodesepA=5pt,nodesepB=3pt,labelsep=2pt}
\ncline{->}{x1}{y1} \naput{{\scriptsize $(I_2,\iota_2)$}} \nbput{{\scriptsize $\iso$}}
\ncline{->}{x2}{y2} \naput{{\scriptsize $(I_1,\iota_1)$}} \nbput{{\scriptsize $\iso$}}
\ncline{->}{x3}{y3} \naput{{\scriptsize $(I_0,\iota_0)$}} \nbput{{\scriptsize $\iso$}}

\endpspicture
\]
As before, this is true by a straightforward induction: we have
\[
\FM\1 = (\cat{$\1$-Gph}, 1) \iso (\cat{$0$-Gph},1)
\]
and as the rightmost square of the diagram commutes, the whole diagram commutes.

This shows that the limit we need to take is over the following diagram in $\MNDd$
\[ 
\cdots 
\map{(U_3, \gamma_3)} (\cat{$2$-Gph}, T_2) 
\map{(U_2, \gamma_2)} (\cat{$1$-Gph}, T_1) 
\map{(U_1, \gamma_0)} (\cat{$0$-Gph}, T_0)
\]
So our aim is to show that the cone (\ref{diagc}) is a limit in $\MNDd$.
We know the following.
\begin{enumerate}
 \item The underlying graph part is a limit cone in $\CAT$ (Definition~\ref{threepointfive}).
\item The underlying graph part is a cone in $\CATd$; this follows from the equivalences
\[\cat{$n$-Gph} \catequiv \cat{$n$-GSet}\]
\[\cat{$\omg$-Gph} \catequiv \cat{GSet}.\]

\item The diagram is in $\MNDwk$. For this we need to show that all the monad morphisms are weak, not just lax.  Certainly each $(\mbox{tr}_n, \tau_n)$ is weak by the proof of Theorem~F.2.2 of~\cite{lei8}.  (In fact, this proof shows that if we are willing to
replace each $T_n$ by an isomorphic monad, then we may arrange that each
$\gamma_n$ and $\tau_n$ is an identity.)  For the $(U_n, \gamma_n)$ it suffices to show that $\FM^n !$ is weak for all $n$.  Now $!$ is certainly weak, and $\FM$ preserves weakness---examining the definition of $\FM$ on morphisms (Definition~\ref{fourpointten}) we see that if $\theta$ is an isomorphism then $\theta_*$ is too, which gives the result.

\end{enumerate}

So by Proposition~\ref{twopointseven} our diagram is a limit in $\MNDd$ and also in $\MND$; we will use the latter in the next proof.

% \bigskip

We now show that $\FM$ preserves this limit, also using Proposition~\ref{twopointseven}.  We have the same three steps as above, for $\FM$ of the limit diagram as follows.

\begin{enumerate}
 \item The underlying graph part is a limit cone in \CAT; this amounts to $\FG$ preserving the underlying limit in \CAT\ which we know from Section~\ref{strictcat}.

\item The underlying graph part is also a cone in $\CATd$ as above.

\item The diagram is in $\MNDwk$ as we have seen above that $\FM$ preserves weakness.
\end{enumerate}

Hence the new cone is also a limit in $\MNDd$ and $\MND$. \end{proof}

%2

We can now deduce that $\omg\cat{-Cat}$ is the terminal coalgebra for $\FC$.

\begin{prfof}{Theorem~\ref{thm:strict-cat-tc}}
We use Ad\'amek's Theorem.  The isomorphisms of Lemma~\ref{lemma:FC-ladder} tell us it suffices to show
that the following is a limit in $\CATp$.
\begin{equation}\label{diagfour}
\psset{unit=0.1cm,labelsep=2pt,nodesep=2pt}
\pspicture(0,-4)(120,25)

%   a
% x0 x1 x2 x3 x4

\rput(20,20){\rnode{a}{\cat{$\omg$-Cat}}}

\rput(0,0){$\cdots$}

\rput(0,0){\rnode{x0}{\white{$\DG^3\1$}}}
\rput(25,0){\rnode{x1}{$\cat{$n$-Cat}$}}
\rput(55,0){\rnode{x2}{$\cat{$(n-1)$-Cat}$}}
\rput(85,0){\rnode{x3}{$\cdots$}}
\rput(110,0){\rnode{x4}{$\cat{$0$-Cat}$}}

\psset{nodesep=2pt}
\ncline{->}{x0}{x1} \naput{{\scriptsize $$}}
\ncline{->}{x1}{x2} \naput{{\scriptsize $$}}
\ncline[nodesepB=12pt]{->}{x2}{x3} \naput{{\scriptsize $$}}
\ncline[nodesepA=12pt]{->}{x3}{x4} \naput{{\scriptsize $$}}

\ncline{->}{a}{x0} \naput{{\scriptsize $$}}
\ncline{->}{a}{x1} \naput{{\scriptsize $$}}
\ncline{->}{a}{x2} \naput{{\scriptsize $$}}
\ncline{->}{a}{x4} \naput{{\scriptsize $$}}

\endpspicture
\end{equation}
and that $\FC$ preserves this limit.  We know the following facts.
\begin{enumerate}
 \item It is a limit in $\CAT$ (Definition~\ref{newfourpointfour}).
\item It is the image under $\Alg$ of a cone (\ref{diagc}) in $\MNDd$.
\item $\Alg$ restricts to a functor $\MNDd \lra \CATd$, so the diagram (\ref{diagfour}) is a cone in $\CATd$ and hence $\CATp$ (as \CATd\ is a subcategory of \CATp).
\item The inclusion $\CATp \hra \CAT$ reflects limits (Proposition~\ref{twopointeight}) so the diagram is a limit in \CATp.
\end{enumerate}
To show that $\FC \: \CATp \tra \CATp$ preserves this limit, it is again enough to show that $\FC$ of the limit cone is a limit cone in $\CAT$, as the inclusion reflects limits.  We use the
commutative square 
\[
\psset{unit=0.1cm,labelsep=3pt,nodesep=3pt}
\pspicture(0,-4)(24,24)

% a1 a2
% a3 a4

%%%%%%%%%% top

\rput(0,20){\rnode{a1}{$\MNDd$}}  % top left
\rput(24,20){\rnode{a2}{$\MNDd$}}  % top right
\rput(0,0){\rnode{a3}{$\CATd$}}  % bottom left
\rput(24,0){\rnode{a4}{$\CATd$}}  % bottom right

\ncline{->}{a1}{a2} \naput{{\scriptsize $\FM$}} % top
\ncline{->}{a3}{a4} \nbput{{\scriptsize $\FC$}} % bottom
\ncline{->}{a1}{a3} \nbput{{\scriptsize $\Alg$}} % left
\ncline{->}{a2}{a4} \naput{{\scriptsize $\Alg$}} % right

\endpspicture
\]
of Lemma~\ref{lemma:comparing-Fs}.  Now the limit cone we need to preserve is diagram (\ref{diagfour}), which we write as follows
\[(\cat{$\omg$-Cat} \map{\mbox{tr}_n} \cat{$n$-Cat})_{n\geq 0}.\]
This is in fact
\[\Alg 
\left(
(\omg\cat{-Gph}, \Tom) \map{(\tr_n, \tau_n)} (n\cat{-Gph}, T_n)
\right)_{n \geq 0}\]
so applying $\FC$ we get
\[\FC \Alg 
\left(
(\omg\cat{-Gph}, \Tom) \map{(\tr_n, \tau_n)} (n\cat{-Gph}, T_n)
\right)_{n \geq 0}\]
which, by the above commutative square is
\[\Alg \FM
\left(
(\omg\cat{-Gph}, \Tom) \map{(\tr_n, \tau_n)} (n\cat{-Gph}, T_n)
\right)_{n \geq 0}\]
and we know this is a limit as follows.  In the previous proof we showed that 
\[\FM
\left(
(\omg\cat{-Gph}, \Tom) \map{(\tr_n, \tau_n)} (n\cat{-Gph}, T_n)
\right)_{n \geq 0}\]
is a limit in $\MND$; furthermore we know that $\Alg$ has a left adjoint and so preserves limits (Proposition~\ref{twopointthree}). Thus we have the limit required.
\end{prfof}

%3

\begin{remark}
Our proof that $\omg\cat{-Cat}$ is monadic over $\omg\cat{-Gph}$ did not
reveal an explicit description of the monad.  But in fact this monad, the free
strict $\omg$-category monad $\Tom$, does have an explicit description, at
least if one changes $\omg\cat{-Gph}$ to the equivalent category
$\cat{GSet}$.  Indeed, re-using the notation $\Tom$ and $T_n$ for the free
$\omg$- and $n$-category monads, one may find in the paper~\cite{bat1} of
Batanin an explicit construction of the whole limit cone
\[
\big(
(\cat{GSet}, \Tom) \lra (n\cat{-GSet}, T_n)
\big)_{n \geq 0}.
\]
Transferring across the equivalence between globular sets and graphs gives
the limit cone
\[\left( 
(\omg\cat{-Gph}, \Tom) \map{(\tr_n, \tau_n)} (n\cat{-Gph}, T_n)
\right)_{n \geq 0}.\] 
\end{remark}

%%%%%%%%%%%%%%%%%%%%%%%%%%%%%%%%%%%%%%%%%%%%%%%%%%%%%%%%%%%%%%%%%%%%%%%%%
%%%%%%%%%%%%%%%%%%%%%%%%%%%%%%%%%%%%%%%%%%%%%%%%%%%%%%%%%%%%%%%%%%%%%%%%%
% Section: Weak omega-categories
%%%%%%%%%%%%%%%%%%%%%%%%%%%%%%%%%%%%%%%%%%%%%%%%%%%%%%%%%%%%%%%%%%%%%%%%%
%%%%%%%%%%%%%%%%%%%%%%%%%%%%%%%%%%%%%%%%%%%%%%%%%%%%%%%%%%%%%%%%%%%%%%%%%

\section{Batanin--Leinster weak $\omg$-categories}\label{sec:weak}
\label{batanin}

In the case of weak $\omg$-categories we cannot simply use a limit of the
categories of $n$-categories as above.  This is because truncating a weak
$\omg$-category to $n$ dimensions does \emph{not} produce a weak
$n$-category, as the top dimension will not be coherent.  The main idea of
this section is that we must instead make the construction using ``incoherent
$n$-categories''---these should not be completely incoherent, but just as
incoherent as a truncated weak $\omg$-category. 

The notion of weak $\omg$-category that we use here is Leinster's, which is
based in turn on Batanin's~\cite{bat1}.  See~\cite{lei7} for the definition in
concise terms, and~\cite{cl1} or~\cite{lei8} for explanation.

In later sections we will use the same idea to construct a theory of
Trimble-like weak $\omg$-categories, the point being that while we cannot
use Trimble's inductive definition to produce $\omg$-categories, we can use
it to produce incoherent $n$-categories for each finite $n$, and then
construct $\omg$-categories as a limit of those.

There are no terminal coalgebras in this section.  This is because in the Batanin--Leinster version, the theory
of incoherent $n$-categories is not obtained from the theory of incoherent $(n
- 1)$-categories by applying an endofunctor; at least, that is not how
\emph{we} obtain it.  The main results will, however, concern sequential
limits of the same type that we have seen repeatedly in applying Ad\'amek's
Theorem, and that we will see again when we come to Trimble $n$-categories.  Thus this section can be seen as motivation and justification for the use of incoherent $n$-categories in the Trimble case.

First we give a low-dimensional example to illuminate the idea of incoherent
$n$-categories. 

\begin{myexample} %5.1
\label{eg:2-trunc}

Let $A$ be a weak $\omg$-category.  If we truncate its underlying globular
set to 2 dimensions, we have the following data: a 2-globular set
%
%ABC
\[
\psset{unit=0.1cm,labelsep=0pt,nodesep=3pt}
\pspicture(0,-3)(40,3)
\rput(0,0){\rnode{A}{$A_2$}}
\rput(20,0){\rnode{B}{$A_1$}}
\rput(40,0){\rnode{C}{$A_0$}}
\psset{nodesep=3pt,arrows=->}
\ncline[offset=3.5pt]{A}{B}\naput[npos=0.5,labelsep=2pt]{{\scriptsize $s$}}
\ncline[offset=-3.5pt]{A}{B}\nbput[npos=0.5,labelsep=2pt]{{\scriptsize$t$}}
\ncline[offset=3.5pt]{B}{C}\naput[npos=0.5,labelsep=2pt]{{\scriptsize$s$}}
\ncline[offset=-3.5pt]{B}{C}\nbput[npos=0.5,labelsep=2pt]{{\scriptsize$t$}}

\endpspicture
\]
equipped with identities, composition and coherence cells as for a bicategory,
but \emph{satisfying no axioms}.  For example we would have specified 2-cells 
\[a\:(hg)f \trta h(gf)\]
and
\[a^*\:h(gf) \trta (hg)f\]
but no stipulation that these be inverses or satisfy the pentagon axiom.  In
this way the structure is ``not coherent'' (since it does not obey the usual
coherence laws) but is also ``not completely incoherent'' since it does give
\emph{some} relationship between different composites of the same diagram,
just not a strong enough one. 

\end{myexample}

At the heart of this section is the thought that the following two structures are the same even in the case of weak structures:
\begin{itemize}
 \item an $\omg$-category truncated to $n$ dimensions, and 
\item a $k$-category truncated to $n$ dimensions, for any $k > n$.
\end{itemize}
This is stated precisely after Corollary~\ref{cor:trunc-iopd}. 

In this section we write $\omg\cat{-Cat}$ and $n\cat{-Cat}$ for the
categories of \emph{weak} $\omg$- and $n$-categories according to the Batanin--Leinster definition.  We will recall the
definitions shortly, and we will also define the category $n\cat{-iCat}$ of
incoherent $n$-categories.  Theorem~\ref{thm:weak-cat-tc} states that weak $\omg$-categories can be built from incoherent $n$-categories in the sense that
$\omg\cat{-Cat}$ is the limit of the following diagram:
\[
\cdots
\lra 2\cat{-iCat}
\lra 1\cat{-iCat}
\lra 0\cat{-iCat}
\]
We will make the definition of incoherent $n$-category directly, using only
$n$-dimensional notions, and then show that it is the same as the structure
obtained by truncating weak $\omg$-categories.

We now recall some definitions.  Batanin--Leinster $\omg$-categories are defined via $T$-operads for a particular cartesian monad $T$.  Recall that a cartesian monad is one whose functor part preserves pullbacks and whose natural transformation parts are cartesian, that is, their naturality squares are all pullbacks.  In what follows, $\cl{V}$ is a category with finite
limits and $T$ is a {cartesian monad} on $\cl{V}$.  

\begin{mydefinition} %5.2
\label{defn:T-opd}

A \demph{$T$-operad} is a cartesian monad $P$ on $\cl{V}$ together with a
cartesian natural transformation $\pi\: P \Tra T$ commuting with the monad
structures.  By the standard abuse of notation, we usually refer to a
$T$-operad $(P, \pi)$ as $P$.

\end{mydefinition}

The history of this definition goes back to a 1971 paper of
Burroni~\cite{bur1} though the word ``operad'' comes from
May~\cite{may2}.  An alternative point of view on the definition, closer to
Burroni's and put forward in the work of Batanin~\cite{bat1}, is as follows.

The category $T\cat{-Coll}$ of \demph{$T$-collections} is defined to be the
slice category $\cl{V}/T1$.  This carries a natural monoidal structure (as
described in~\cite{lei7}, for instance).  A $T$-operad is exactly a monoid in
the monoidal category $T\cat{-Coll}$.  See Corollary~6.2.4 of~\cite{lei8} for
a proof, and Chapter~4 of~\cite{lei8} for general explanation and examples of
$T$-operads.

In this paper we will not need the details of this alternative point of view.
All we need to know is how a $T$-operad in the sense of
Definition~\ref{defn:T-opd} gives rise to a collection, namely: for a
$T$-operad $(P, \pi)$, the resulting collection is
\[
\psset{unit=0.1cm,labelsep=2pt,nodesep=2pt}
\pspicture(5,17)

% a1 
% a2 
\rput(0,15){\rnode{a1}{$P1$}}  %  top
\rput(0,2){\rnode{a2}{$T1$}}  % bottom

\ncline{->}{a1}{a2} \naput{{\scriptsize $\pi_1$}} % top

\endpspicture
\]
Let us now investigate what happens as the monad $T$ is varied.  This will be
applied when we come to consider $n$-categories for varying values of $n$.
Details can be found in Section~6.7 of~\cite{lei8}.

Let
\[
(H, \theta)\: (\cl{V}, T) \lra (\cl{V}', T')
\]
be a weak morphism of monads, where the categories $\cl{V}$ and $\cl{V}'$ 
have finite limits, the functor $H$ preserves them, and the monads $T$ and
$T'$ are cartesian.  Then any $T$-operad $P$ gives rise canonically to a
$T'$-operad $HP$.  (This is easier to see in the collection point of view on
operads.)  Moreover, there is an induced weak morphism of monads
\[
(H, \widetilde{\theta})\: (\cl{V}, P) \lra (\cl{V}', HP).
\]
We will use these constructions shortly.

We will always take $\cl{V}$ to be the category of $n$-globular sets, and $T$
to be the free strict $n$-category monad, for varying values of $n \in
\mathbb{N} \cup \{\omg\}$.  (In this section only, we work with $n$-globular
sets rather than $n$-graphs.)  This gives the notion of globular operad,
introduced by Batanin~\cite{bat1}.

\begin{mydefinition} %5.3

A \demph{globular operad} is a $\Tom$-operad, where $\Tom$ is the free strict
$\omg$-category monad on $\cat{GSet}$.  Similarly, for $n \in \mathbb{N}$, a
\demph{(globular) $n$-operad} is a $T_n$-operad, where $T_n$ is the free
strict $n$-category monad on $n\cat{-GSet}$.  An \demph{algebra} for a
globular operad or $n$-operad is an algebra for its underlying monad.  

\end{mydefinition}

The definition makes sense because $\cat{GSet}$ and $n\cat{-GSet}$ have finite
limits (being presheaf categories) and because the monads $\Tom$ and $T_n$ are
cartesian~\cite[Section F.2]{lei8}.

Later we will need the following technical fact.

\begin{lemma} %5.4
\label{lemma:opd-dist}

Let $n \in \mathbb{N} \cup \{\omg\}$ and let $P$ be an $n$-operad.  Then
$(n\cat{-GSet}, P) \in \MNDd$.

\end{lemma}
(When $n = \omg$, the notation $n\cat{-GSet}$ means $\cat{GSet}$.)

\begin{proof}
We have to show that $n\cat{-GSet}$ is an infinitely distributive category,
and that the functor part of $P$ preserves small coproducts.  The first
assertion is true since $n\cat{-GSet}$ is a presheaf category.  The second
is a consequence of the following facts: the free strict $n$-category functor
$T_n$ preserves coproducts~\cite[Section F.2]{lei8}, there exists a cartesian
natural transformation $P \Tra T_n$, and coproducts in a presheaf category are
stable under pullback.
\end{proof}

For each $n \geq 0$ we have a morphism of monads
\[
(\tr_n, \tau_n)\: (\cat{GSet}, \Tom) \lra (n\cat{-GSet}, T_n)
\]
and moreover it is weak (see proof of Theorem~\ref{fourpointsixteen}).  So by the remarks above, any globular
operad $P$ gives rise to an $n$-operad $\tr_n P$ and a weak morphism of monads
\begin{equation}
\label{eq:P-trunc}
(\tr_n, \widetilde{\tau_n}):
(\cat{GSet}, P) \lra (n\cat{-GSet}, \tr_n P).
\end{equation}
Similarly, for each $n \geq 1$ we have the weak morphism of monads 
\[
(U_n, \gamma_n):
(n\cat{-GSet}, T_n) \lra \big((n - 1)\cat{-GSet}, T_{n - 1}\big)
\]
so any $n$-operad $Q$ gives rise to an
$(n - 1)$-operad $U_n Q$ and a weak morphism of monads
\begin{equation}
\label{eq:Q-trunc}
(U_n, \widetilde{\gamma_n}):
(n\cat{-GSet}, Q) \lra \big((n - 1)\cat{-GSet}, U_n Q\big).
\end{equation}

We now come to the definitions of weak $\omg$-category, weak $n$-category,
and incoherent $n$-category.  The strategy for the first, due to
Batanin~\cite{bat1}, is to define a weak $\omg$-category as an algebra for a
certain globular operad.  The operad we use is different from Batanin's;
see~\cite{lei7}.  To define it, we use the concept of ``contraction'' (which
also has a different meaning here from in Batanin's work).

By \demph{collection} and \demph{$n$-collection}, we mean $\Tom$- and
$T_n$-collection, respectively.

\begin{mydefinition} %5.5
\label{defn:contraction}

Let 
\[
\psset{unit=0.1cm,labelsep=2pt,nodesep=2pt}
\pspicture(5,17)

% a1 
% a2 
\rput(0,15){\rnode{a1}{$A$}}  %  top
\rput(0,2){\rnode{a2}{$T_\omg 1$}}  % bottom

\ncline{->}{a1}{a2} \naput{{\scriptsize $p$}} % top

\endpspicture
\]
be a collection.  For $m \in \mathbb{N}$, a pair $a, b$ of $m$-cells in $A$ is
\demph{parallel} if \emph{either} $m = 0$ \emph{or} $m \geq 1$, $s(a) = s(b)$ and $t(a) = t(b)$.  A \demph{contraction} on the collection $p$ is a function
assigning to each
\begin{itemize}
 \item $m \in \mathbb{N}$, 
\item parallel pair $a, b \in A_m$, and 
\item $(m + 1)$-cell $y\: p(a) \lra p(b)$ in $\Tom 1$,
\end{itemize}
an $(m + 1)$-cell $x\: a \lra b$ in $A$ such that $p(x) = y$.

A (globular) \demph{operad with contraction} is a globular operad equipped
with a contraction on its underlying collection.

\end{mydefinition}

(This definition can also be phrased in terms of a right lifting property;
we leave this as an exercise for the interested reader.)

There is a category \cat{OWC} of operads with contraction, whose morphisms are
the morphisms of operads preserving those contractions.  This category has an
initial object~\cite[Proposition 9.2.2]{lei8}, whose underlying operad we call
$\Pom$.

\begin{mydefinition} %5.6

A \demph{weak $\omg$-category} is an algebra for $\Pom$, the initial operad
with contraction.  We write $\omg\cat{-Cat}$ for the category
$\cat{GSet}^{\Pom}$ of weak $\omg$-categories and strict $\omg$-functors.

\end{mydefinition}

We now turn to finite-dimensional structures.  In the definition of weak $n$-category, the notion of contraction is modified at the top dimension to ensure coherence.  For incoherent $n$-categories, this is exactly the part of the definition that we eliminate.  Thus the following definition is in essence just a straightforward truncation of the definition of contraction given above; the ``incoherence'' is relative to what would be necessary if we were defining weak $n$-categories.  

\begin{mydefinition} %5.7

An \demph{incoherent contraction} on an $n$-collection
\[
\psset{unit=0.1cm,labelsep=2pt,nodesep=2pt}
\pspicture(5,17)

% a1 
% a2 
\rput(0,15){\rnode{a1}{$A$}}  %  top
\rput(0,2){\rnode{a2}{$T_n1$}}  % bottom

\ncline{->}{a1}{a2} \naput{{\scriptsize $p$}} % top

\endpspicture
\]
is defined just as contractions were defined in
Definition~\ref{defn:contraction}, but replacing $\Tom$ by $T_n$ and ``$m \in
\mathbb{N}$'' by ``$m \in \{0, \ldots, n - 1\}$'' throughout.  The category
$n\cat{-iOWC}$ of \demph{$n$-operads with incoherent contraction} is defined
similarly.

\end{mydefinition}

The category $n\cat{-iOWC}$ has an initial object $\iP{n}$, as we shall see.

\begin{mydefinition} %5.8

An \demph{incoherent $n$-category} is an algebra for $\iP{n}$, the
initial $n$-operad with incoherent contraction.  We write $n\cat{-iCat}$ for
the category $n\cat{-GSet}^{\iP{n}}$ of incoherent $n$-categories.  

\end{mydefinition}

While incoherent $n$-categories are weak structures, the morphisms in the
category $n\cat{-iCat}$ should be thought of as \emph{strict} $n$-functors.

\begin{remark}
We will not need (coherent) weak $n$-categories in what follows, but it may
clarify matters to consider them briefly now.  An incoherent contraction on an
$n$-collection $A$ is a \demph{contraction} if for any parallel $n$-cells $x,
x'$ in $A$, 
\[
p(x) = p(x') \ \ \Rightarrow\  \  x = x'.
\]
So the ``incoherence'' of an incoherent contraction lies in the fact that
given a parallel pair $a, b$ of $(n - 1)$-cells in $A$ and an $n$-cell $y:
p(a) \lra p(b)$ in $T_n 1$, there may be \emph{many} $n$-cells $x\: a \lra b$
lifting $y$.  

The definition of weak $n$-category is given as for incoherent weak
$n$-category, but with contractions in the place of incoherent
contractions.  That is, the category of $n$-operads with contraction has an
initial object $P_n$, and a {weak $n$-category} is defined as a
$P_n$-algebra.  The operad $P_n$ can be viewed as $\iP{n}$ with some of its
$n$-dimensional operations identified, so a weak $n$-category can be viewed
as an incoherent $n$-category satisfying some $n$-dimensional equations.
The significance of this difference is illustrated in
Example~\ref{eg:2-trunc}. 
\end{remark}

Returning to the main development, let us consider how the theory of
(incoherent) $n$-categories varies as $n$ varies in $\mathbb{N} \cup
\{\omg\}$.  Truncation defines, for each $n$, functors
\[
\cat{OWC} \map{\tr_n} n\cat{-iOWC},
\qquad
n\cat{-iOWC} \map{U_n} (n - 1)\cat{-iOWC}.
\]
So we have, for each $n$, an operad $\tr_n \Pom$ with incoherent contraction.

\begin{proposition} %5.9
\label{prop:initial-iopd}

For each $n \geq 0$, $\tr_n \Pom$ is an initial object of $n\cat{-iOWC}$. 

\end{proposition}

\begin{proof}
This follows from Proposition~9.3.7 of~\cite{lei8} (where incoherent
contractions are called precontractions).
\end{proof}

Hence $\iP{n} = \tr_n\Pom$.  We then have an $(n - 1)$-operad $U_n \iP{n}$
with incoherent contraction.  

\begin{corollary} %5.10
\label{cor:trunc-iopd}

For each $n \geq 1$, we have $U_n \iP{n} = \iP{n - 1}$.

\end{corollary}

\begin{proof}
$U_n \iP{n} = U_n \tr_n \Pom = \tr_{n - 1} \Pom = \iP{n - 1}$.
\end{proof}

It can also be shown that $U_n P_n = \iP{n - 1}$: the theory of weak
$n$-categories truncated to dimension $(n - 1)$ is the theory of incoherent
$(n - 1)$-categories \cite[Corollary~9.3.10]{lei8}.  Hence $\tr_n \Pom =
U_{n + 1} P_{n + 1}$, and, more generally, $\tr_n \Pom = U_{n + 1} U_{n + 2}
\cdots U_k P_k$ for any $k > n$: the theory of weak $\omg$-categories
truncated to dimension $n$ is the same as the theory of weak $k$-categories
truncated to dimension $n$.  

We now exhibit $\omg\cat{-Cat}$ as the limit of the categories
$n\cat{-iCat}$.  For each $n \geq 0$, we have a weak morphism of monads
\[
(\tr_n, \widetilde{\tau_n}):
(\cat{GSet}, \Pom) \lra (n\cat{-GSet}, \tr_n \Pom = \iP{n})
\]
by~(\ref{eq:P-trunc}).  We also have, for each $n \geq 1$, a weak morphism
of monads
\[
(U_n, \widetilde{\gamma_n}):
(n\cat{-GSet}, \iP{n}) \lra \big((n - 1)\cat{-GSet}, U_n \iP{n} = \iP{n - 1}\big)
\]
by~(\ref{eq:Q-trunc}).  Together these form a cone
\begin{equation}
\label{eq:big-opd-cone}
\psset{unit=0.093cm,labelsep=2pt,nodesep=2pt}
\pspicture(0,-5)(122,24)

%   a
% x0 x1 x2 x3 x4

\rput(20,20){\rnode{a}{$(\cat{GSet},P_\omg)$}}

\rput(0,0){$\cdots$}

\rput(0,0){\rnode{x0}{\white{$\DG^3\1$}}}
\rput(22,0){\rnode{x1}{$(\cat{$n$-GSet}, \iP{n})$}}
\rput(60,0){\rnode{x2}{$\big(\cat{$(n-1)$-GSet}, \iP{n-1}\big)$}}
\rput(92,0){\rnode{x3}{$\cdots$}}
\rput(115,0){\rnode{x4}{$(\cat{$0$-GSet}, \iP{0})$}}

\psset{nodesep=2pt}
\ncline{->}{x0}{x1} \naput{{\scriptsize $$}}
\ncline{->}{x1}{x2} \naput{{\scriptsize $$}}
\ncline[nodesepB=8pt]{->}{x2}{x3} \naput{{\scriptsize $$}}
\ncline[nodesepA=8pt]{->}{x3}{x4} \naput{{\scriptsize $$}}

\ncline{->}{a}{x0} \naput{{\scriptsize $$}}
\ncline{->}{a}{x1} \naput{{\scriptsize $$}}
\ncline{->}{a}{x2} \naput{{\scriptsize $$}}
\ncline{->}{a}{x4} \naput{{\scriptsize $$}}

\endpspicture
\end{equation}
in $\MND$.  In fact it is a cone in $\MNDd$, by Lemma~\ref{lemma:opd-dist}.
Applying the functor $\Alg$ gives a cone
\begin{equation}
\label{eq:big-weak-cone}
\psset{unit=0.093cm,labelsep=2pt,nodesep=3pt}
\pspicture(0,-5)(122,24)

%   a
% x0 x1 x2 x3 x4

\rput(20,20){\rnode{a}{\cat{$\omg$-Cat}}}

\rput(0,0){$\cdots$}

\rput(0,0){\rnode{x0}{\white{$\DG^3\1$}}}
\rput(22,0){\rnode{x1}{$\cat{$n$-iCat}$}}
\rput(60,0){\rnode{x2}{$\cat{$(n-1)$-iCat}$}}
\rput(92,0){\rnode{x3}{$\cdots$}}
\rput(115,0){\rnode{x4}{$\cat{$0$-iCat}$}}

\psset{nodesep=2pt}
\ncline{->}{x0}{x1} \naput{{\scriptsize $$}}
\ncline{->}{x1}{x2} \naput{{\scriptsize $$}}
\ncline[nodesepB=10pt]{->}{x2}{x3} \naput{{\scriptsize $$}}
\ncline[nodesepA=10pt]{->}{x3}{x4} \naput{{\scriptsize $$}}

\ncline{->}{a}{x0} \naput{{\scriptsize $$}}
\ncline{->}{a}{x1} \naput{{\scriptsize $$}}
\ncline{->}{a}{x2} \naput{{\scriptsize $$}}
\ncline{->}{a}{x4} \naput{{\scriptsize $$}}

\endpspicture
\end{equation}
in $\CATd$. 

We now prove the main results of this section.

\begin{theorem} %5.11
\label{thm:weak-lim-cone}

The cone~(\ref{eq:big-opd-cone}) is a limit cone in both $\MND$ and $\MNDd$.

\end{theorem}

That is, the monad for weak $\omg$-categories is the limit of the monads for
incoherent $n$-categories.  This is a precise expression of the idea that the
theory of weak $\omg$-categories is built up dimension by dimension from the
theories of incoherent $n$-categories.

\begin{proof}
All the arrows in this cone are weak morphisms of monads, and the underlying
cone in $\CAT$ is a limit cone.  So by Proposition~\ref{prop:reflection}, the
cone~(\ref{eq:big-opd-cone}) is a limit in $\MND$ and $\MNDd$.
\end{proof}

\begin{theorem} %5.12
\label{thm:weak-cat-tc}

The cone~(\ref{eq:big-weak-cone}) is a limit cone in $\CAT$, $\CATp$ and
$\CATd$. 

\end{theorem}

That is, the category of weak $\omg$-categories is the limit of the
categories of incoherent $n$-categories.  

\begin{proof}
By Proposition~\ref{prop:FTM-adjns}, the functor $\Alg\: \MND \lra \CAT$
preserves limits, so Theorem~\ref{thm:weak-lim-cone} implies
that~(\ref{eq:big-weak-cone}) is a limit cone in $\CAT$.  Then 
by Proposition~\ref{prop:little-refl}, it is also a limit cone in $\CATp$ and
$\CATd$. 
\end{proof}

In the remaining sections we show how to copy this construction to
produce a definition of weak $\omg$-category for Trimble's theory.  As noted in the introduction, a similar point of view on
Trimble weak $n$-categories (for finite $n$) has been developed
independently in work of Weber \cite{web2}.

%%%%%%%%%%%%%%%%%%%%%%%%%%%%%%%%%%%%%%%%%%%%%%%%%%%%%%%%%%%%%%%%%%%%%%%%%
%%%%%%%%%%%%%%%%%%%%%%%%%%%%%%%%%%%%%%%%%%%%%%%%%%%%%%%%%%%%%%%%%%%%%%%%%
% Section: Trimble: prelude
%%%%%%%%%%%%%%%%%%%%%%%%%%%%%%%%%%%%%%%%%%%%%%%%%%%%%%%%%%%%%%%%%%%%%%%%%
%%%%%%%%%%%%%%%%%%%%%%%%%%%%%%%%%%%%%%%%%%%%%%%%%%%%%%%%%%%%%%%%%%%%%%%%%

\section{Trimble $\omg$-categories: underlying data}\label{sec:trimback}\label{trimblegph}

We now embark on our analysis of Trimble's weak $n$-categories.  The constructions in the next three sections resemble those in Sections \ref{strictgph}, \ref{strictcat} and \ref{strictmnd} but with the functor from \Top\ built in as discussed in the introduction.  We begin by incorporating this data into the ``enriched graph'' endofunctor, before proceeding to study the ``weakly enriched category'' endofunctor in the next section, and finally the monad version in Section~\ref{trimblemnd}.  

Recall that previously we used the endofunctor 
\[\FG \: \cV \mapsto \cV\cat{-Gph}.\]
Now we need an endofunctor of the form
\[(\Top \lra \cV) \hh{1em} \lthickmapsto \hh{1em} (\Top \lra \cV\cat{-Gph}).\]
This motivates us to work in the slice category $\Top/\CAT$.  We now proceed to define the endofunctor we need, and show that its terminal coalgebra is the underlying $\omg$-graph functor
\[\Top \tmap{\Pi_\omg} \omg\cat{-Gph}.\]

Let $\cat{Top}$ be the category of all topological spaces and continuous
maps.  The crucial property of $\cat{Top}$ here is that it admits the following canonical functor; this is an unbased analogue of the loop space functor
$\Omega$.

\begin{definition}\label{newsevenpointone}
There is a canonical functor
\[
\Gamma\: \cat{Top} \lra \cat{Top-Gph}
\]
defined as follows.  Given a space $X$, the $\cat{Top}$-graph $\Gamma X$ is given by:
\begin{itemize}
\item $\cat{ob}(\Gamma X)$ is the set of points of $X$
\item for points $x, y$ of $X$, the space $(\Gamma X)(x, y)$ is the path space
$X(x, y)$.
\end{itemize}
Note that paths are always parametrised
by $[0, 1]$ so that $X(x, y)$ is naturally a subspace of $X^{[0, 1]}$.
\end{definition}

We are now ready to define our endofunctor.

\begin{definition}\label{DG}

We define an endofunctor 
\[\DG \: \Top/\CAT \ltra \Top/\CAT\]
by
\[
\DG
\left( \cat{Top} \tmap{\Pi} \cl{V} \right)
=
\left( \cat{Top} \tmap{\Pi^+} \cl{V}\cat{-Gph} \right)
\]
where $\Pi^+$ is the composite functor
\[
\cat{Top} \map{\Gamma} \cat{Top-Gph} \map{\Pi_*} \cl{V}\cat{-Gph}.
\]
Explicitly, given a space $X$, the graph $\Pi^+ X$ has
\begin{itemize}
\item objects the set of points of $X$, and 
\item for points $x$ and $y$,  $(\Pi^+ X)(x, y) =
\Pi\big(X(x, y)\big)$  
\end{itemize}
On morphisms, $\DG$ is defined by
\[
\psset{unit=0.1cm,labelsep=2pt,nodesep=2pt}
\pspicture(0,5)(85,30)

\rput(3,20){\pspicture(20,20)
%    a2
% a1 
%    a3 
\rput(0,10){\rnode{a1}{$\Top$}}  %  left
\rput(17,17){\rnode{a2}{$\cV$}}  % top right
\rput(17,3){\rnode{a3}{$\cV'$}}  % bottom right

\ncline{->}{a1}{a2} \naput[npos=0.55]{{\scriptsize $\Pi$}} % top
\ncline{->}{a1}{a3} \nbput[npos=0.55]{{\scriptsize $\Pi'$}} % left
\ncline[labelsep=2pt]{->}{a2}{a3} \naput{{\scriptsize $H$}} % right

\endpspicture}

\pcline[linewidth=1.1pt,tbarsize=2pt 2]{|->}(20,20)(30,20)

\rput(50,20){
\pspicture(20,20)
%       a2
%  a0 a1 
%       a3 

\rput(0,10){\rnode{a0}{$\Top$}}  %  left

\rput(22,10){\rnode{a1}{$\cat{Top-Gph}$}}  %  left
\rput(45,19){\rnode{a2}{\cat{$\cV$-Gph}}}  % top right
\rput(45,1){\rnode{a3}{\cat{$\cV'$-Gph}}}  % bottom right

\psset{nodesep=1pt}

\ncline{->}{a0}{a1} \naput[npos=0.5]{{\scriptsize $\Gamma$}} % top
\ncline{->}{a1}{a2} \naput[npos=0.6]{{\scriptsize $\Pi_*$}} % top
\ncline{->}{a1}{a3} \nbput[npos=0.6]{{\scriptsize $\Pi'_*$}} % left
\ncline[labelsep=2pt]{->}{a2}{a3} \naput{{\scriptsize $H_*$}} % right

\endpspicture}

\endpspicture
\]

\end{definition}

Next we define what might be called the ``incoherent'' or ``truncated'' fundamental $n$-graph functor---instead of taking homotopy classes at the top dimension we simply truncate.   This will be the $n$-dimensional part of our eventual terminal coalgebra, that is, the $n$th term in the sequential limit.  As before, we use the superscript ``i'' to indicate the incoherence. 

\begin{definition}
\label{defn:Piin-Trim}
For each $n \geq 0$, define a functor $\iPi{n}\: \cat{Top} \lra
n\cat{-Gph}$ by:
\begin{itemize}
\item $\iPi{0}\: \cat{Top} \lra 0\cat{-Gph} = \cat{Set}$ is the functor mapping
a space to its set of points

\item for $n \geq 1$,
\[
\psset{unit=0.1cm,labelsep=2pt,nodesep=2pt}
\left(\pspicture(-6,7.5)(6,18)

% a1 
% a2 
\rput(0,15){\rnode{a1}{$\Top$}}  %  top
\rput(0,2){\rnode{a2}{$\cat{$n$-Gph}$}}  % bottom

\ncline{->}{a1}{a2} \naput{{\scriptsize $\iPi{n}$}} % top

\endpspicture\right) = \DG
\left(\pspicture(-11,7.5)(11.5,18)

% a1 
% a2 
\rput(0,15){\rnode{a1}{$\Top$}}  %  top
\rput(0,2){\rnode{a2}{$\cat{$(n-1)$-Gph}$}}  % bottom

\ncline{->}{a1}{a2} \naput{{\scriptsize $\iPi{n-1}$}} % top

\endpspicture\right)
\]

\end{itemize}
Thus $\iPi{n}(X)$ is the $n$-graph consisting of the points of $X$, the paths, the homotopies between paths, the homotopies between those, and so on up to the $n$th dimension, where we do \emph{not} take homotopy classes; this corresponds to the $n$-globular set whose $k$-cells are continuous maps from the $k$-ball into $X$.

Furthermore, for each $n \geq 1$ we have the morphism
\[ U_n \: \left(\pspicture(-6,7.5)(6,18)

% a1 
% a2 
\rput(0,15){\rnode{a1}{$\Top$}}  %  top
\rput(0,2){\rnode{a2}{$\cat{$n$-Gph}$}}  % bottom

\ncline[nodesep=2pt]{->}{a1}{a2} \naput{{\scriptsize $\iPi{n}$}} % top

\endpspicture\right)
\tra
\left(\pspicture(-11,7.5)(11.5,18)

% a1 
% a2 
\rput(0,15){\rnode{a1}{$\Top$}}  %  top
\rput(0,2){\rnode{a2}{$\cat{$(n-1)$-Gph}$}}  % bottom

\ncline[nodesep=2pt]{->}{a1}{a2} \naput{{\scriptsize $\iPi{n-1}$}} % top

\endpspicture\right)
\]
in $\cat{Top}/\CAT$, where the underlying functor 
\[U_n\: n\cat{-Gph} \lra (n-1)\cat{-Gph}\] 
is the truncation functor as in the graph case (Definition~\ref{threepointthree}). 
\end{definition}

We wish to define the \omgadj-dimensional version of $\iPi{n}$ as a limit
of the $n$-dimensional versions.  The following result tells us we can take
limits in the slice category $\Top/\CAT$ as we expect.

\begin{proposition}{\cite[Exercise V.1.1]{mac1}}
\label{prop:coslice-creates}\label{onepointten}  %1.10
Let $\cl{C}$ be a category and $C \in \cl{C}$.  Then the forgetful functor
$C/\cl{C} \lra \cl{C}$ creates limits.  

\end{proposition}

\begin{corollary}
The diagram
\[
\cdots 
\map{U_2}
\left(\pspicture(-6,7.5)(6,18)

% a1 
% a2 
\rput(0,15){\rnode{a1}{$\Top$}}  %  top
\rput(0,2){\rnode{a2}{$\cat{$1$-Gph}$}}  % bottom

\ncline[nodesep=2pt]{->}{a1}{a2} \naput{{\scriptsize $\iPi{1}$}} % top

\endpspicture\right)
\map{U_1}
\left(\pspicture(-6,7.5)(6,18)

% a1 
% a2 
\rput(0,15){\rnode{a1}{$\Top$}}  %  top
\rput(0,2){\rnode{a2}{$\cat{$0$-Gph}$}}  % bottom

\ncline[nodesep=2pt]{->}{a1}{a2} \naput{{\scriptsize $\iPi{0}$}} % top

\endpspicture\right)
\]
has a limit in $\cat{Top}/\CAT$ of the form
\begin{equation}
\label{eq:Pi-graph}
\psset{unit=0.1cm,labelsep=2pt,nodesep=2pt}
\pspicture(5,17)

% a1 
% a2 
\rput(0,15){\rnode{a1}{$\Top$}}  %  top
\rput(0,2){\rnode{a2}{$\cat{$\omg$-Gph}$}}  % bottom

\ncline{->}{a1}{a2} \naput{{\scriptsize $\Pi_\omg$}} % top

\endpspicture
\end{equation}
\end{corollary}

\begin{definition}
We call $\Piom$ the \demph{fundamental $\omg$-graph} functor.  We write $\tr_n$ for
the $n$th projection of this limit, and call it \demph{truncation}.  

\end{definition}

\begin{remark}\label{sixprecise}
The commuting triangles 
\[\pspicture(20,20)
%    a2
% a1 
%    a3 
\rput(0,10){\rnode{a1}{$\Top$}}  %  left
\rput(19,17){\rnode{a2}{$\cat{$\omg$-Gph}$}}  % top right
\rput(19,3){\rnode{a3}{$\cat{$n$-Gph}$}}  % bottom right

\ncline{->}{a1}{a2} \naput[npos=0.55]{{\scriptsize $\Piom$}} % top
\ncline{->}{a1}{a3} \nbput[npos=0.55]{{\scriptsize $\iPi{n}$}} % left
\ncline[labelsep=2pt]{->}{a2}{a3} \naput{{\scriptsize $\mbox{tr}_n$}} % right

\endpspicture\]
tell us how $\Piom$ behaves, that is
\[\iPi{n}(X) = \tr_n(\Piom(X)).\]
\end{remark}

\begin{remark}
 Note that $\Gamma$ gives \Top\ the structure of an $\FG$-coalgebra, and $\Piom$ is then the unique map of coalgebras to the terminal coalgebra \cat{$\omg$-Gph}.
\end{remark}

We can now show that \Piom\ is a
terminal coalgebra as required.  We are going to use our previous analysis of $\FG$ to help us, by means of the following two results.  First we note that $\DG$ is a ``lift'' of $\FG$ in the following sense.

\begin{lemma} %6.3
\label{lemma:DG-FG}

The square
\[
\psset{unit=0.1cm,labelsep=3pt,nodesep=3pt}
\pspicture(0,-5)(30,24)

% a1 a2
% a3 a4

%%%%%%%%%% top

\rput(0,20){\rnode{a1}{$\Top/\CAT$}}  % top left
\rput(30,20){\rnode{a2}{$\Top/\CAT$}}  % top right
\rput(0,0){\rnode{a3}{$\CAT$}}  % bottom left
\rput(30,0){\rnode{a4}{$\CAT$}}  % bottom right

\ncline{->}{a1}{a2} \naput{{\scriptsize $\DG$}} % top
\ncline{->}{a3}{a4} \nbput{{\scriptsize $\FG$}} % bottom
\ncline{->}{a1}{a3} \nbput{{\scriptsize $U$}} % left
\ncline{->}{a2}{a4} \naput{{\scriptsize $U$}} % right

\endpspicture
\]
commutes, where $U$ is the forgetful functor.
\proofbox

\end{lemma}

The following result will tell us that $\DG$ satisfies the hypotheses of Ad\'amek's Theorem, because $\FG$ does.

\begin{proposition} %1.11
\label{prop:coslice-tc}

Let $\cl{C}$ be a category and $C \in \cl{C}$.  Let $D$ and $F$ be
endofunctors such that the square
\[
\psset{unit=0.1cm,labelsep=3pt,nodesep=3pt}
\pspicture(0,-4)(20,24)

% a1 a2
% a3 a4

%%%%%%%%%% top

\rput(0,20){\rnode{a1}{$C/\cC$}}  % top left
\rput(20,20){\rnode{a2}{$C/\cC$}}  % top right
\rput(0,0){\rnode{a3}{$\cC$}}  % bottom left
\rput(20,0){\rnode{a4}{$\cC$}}  % bottom right

\ncline{->}{a1}{a2} \naput{{\scriptsize $D$}} % top
\ncline{->}{a3}{a4} \nbput{{\scriptsize $F$}} % bottom
\ncline{->}{a1}{a3} \nbput{{\scriptsize $U$}} % left
\ncline{->}{a2}{a4} \naput{{\scriptsize $U$}} % right

\endpspicture
\]
commutes, where $U$ is the forgetful functor.  Suppose that $F$ satisfies the
hypotheses of Ad\'amek's Theorem~(\ref{adamek}).  Then so does $D$, and the
image under $U$ of the terminal $D$-coalgebra is the terminal $F$-coalgebra.  

\end{proposition}

\begin{proof}
This follows from Ad\'amek's Theorem and $U$ creating limits
(Proposition~\ref{prop:coslice-creates}). 
\end{proof}

We are now ready to characterise the terminal coalgebra of $\DG$. Write $\1$ for the terminal object $(\cat{Top} \tmap{!} \1)$
of $\cat{Top}/\CAT$.

\begin{theorem} %6.5

The category of $\omg$-graphs together with the fundamental $\omg$-graph
functor 
\[\psset{unit=0.1cm,labelsep=2pt,nodesep=2pt}
\pspicture(5,17)

% a1 
% a2 
\rput(0,15){\rnode{a1}{$\Top$}}  %  top
\rput(0,2){\rnode{a2}{$\cat{$\omg$-Gph}$}}  % bottom

\ncline{->}{a1}{a2} \naput{{\scriptsize $\Pi_\omg$}} % top

\endpspicture
\]
is the terminal coalgebra for $\DG$.  

\end{theorem}

\begin{proof}
By Proposition~\ref{prop:coslice-tc} and Lemma~\ref{lemma:DG-FG}, the terminal
coalgebra for $\DG$ exists and is constructed as in Ad\'amek's Theorem.  That is, we take the limit of the diagram
\[
\psset{unit=0.1cm,labelsep=2pt,nodesep=2pt}
\pspicture(0,-5)(80,5)

% x1 x2 x3 x4

\rput(0,0){$\cdots$}

\rput(0,0){\rnode{x0}{\white{$\DG^3\1$}}}
\rput(22,0){\rnode{x1}{$\DG^3 \1$}}
\rput(42,0){\rnode{x2}{$\DG^2 \1$}}
\rput(62,0){\rnode{x3}{$\DG\1$}}
\rput(80,0){\rnode{x4}{$\1$}}

\psset{nodesep=2pt}
\ncline{->}{x0}{x1} \naput{{\scriptsize $\DG^3!$}}
\ncline{->}{x1}{x2} \naput{{\scriptsize $\DG^2!$}}
\ncline{->}{x2}{x3} \naput{{\scriptsize $\DG!$}}
\ncline{->}{x3}{x4} \naput{{\scriptsize $!$}}

\endpspicture
\]
But there are canonical isomorphisms $(I_n)_{n \geq 0}$ such that the following
diagram commutes: 
\[
\psset{unit=0.1cm,labelsep=2pt,nodesep=2pt}
\pspicture(100,44)

% x1 x2 x3
% y1 y2 y3

\rput(0,40){$\cdots$}
\rput(0,10){$\cdots$}

\rput(0,40){\rnode{x0}{\white{$\DG^3\1$}}}
\rput(30,40){\rnode{x1}{$\DG^3\1$}}
\rput(65,40){\rnode{x2}{$\DG^2\1$}}
\rput(100,40){\rnode{x3}{$\DG\1$}}

\psset{nodesep=4pt}
\ncline{->}{x0}{x1} \naput{{\scriptsize $\DG^3!$}}
\ncline{->}{x1}{x2} \naput{{\scriptsize $\DG^2!$}}
\ncline{->}{x2}{x3} \naput{{\scriptsize $\DG!$}}

\rput(3,10){\rnode{y0}{\white{$\DG^3\1$}}}

\rput(30,10){\rnode{y1}{
$\left(\pspicture(-6,7.5)(6,18)
% a1 
% a2 
\rput(0,15){\rnode{a1}{$\Top$}}  %  top
\rput(0,2){\rnode{a2}{$\cat{$2$-Gph}$}}  % bottom
\ncline{->}{a1}{a2} \naput{{\scriptsize $\iPi{2}$}} % top

\endpspicture\right)$}}

\rput(65,10){\rnode{y2}{
$\left(\pspicture(-6,7.5)(6,18)
% a1 
% a2 
\rput(0,15){\rnode{a1}{$\Top$}}  %  top
\rput(0,2){\rnode{a2}{$\cat{$1$-Gph}$}}  % bottom
\ncline{->}{a1}{a2} \naput{{\scriptsize $\iPi{1}$}} % top

\endpspicture\right)$}}

\rput(100,10){\rnode{y3}{
$\left(\pspicture(-6,7.5)(6,18)
% a1 
% a2 
\rput(0,15){\rnode{a1}{$\Top$}}  %  top
\rput(0,2){\rnode{a2}{$\cat{$0$-Gph}$}}  % bottom
\ncline{->}{a1}{a2} \naput{{\scriptsize $\iPi{0}$}} % top

\endpspicture\right)$}}

\psset{nodesep=2pt}
\ncline{->}{y0}{y1} \naput{{\scriptsize $U_3$}}
\ncline{->}{y1}{y2} \naput{{\scriptsize $U_2$}}
\ncline{->}{y2}{y3} \naput{{\scriptsize $U_1$}}

\psset{nodesepA=5pt,nodesepB=3pt}
\ncline{->}{x1}{y1} \naput{{\scriptsize $I_2$}} \nbput{{\scriptsize $\iso$}}
\ncline{->}{x2}{y2} \naput{{\scriptsize $I_1$}} \nbput{{\scriptsize $\iso$}}
\ncline{->}{x3}{y3} \naput{{\scriptsize $I_0$}} \nbput{{\scriptsize $\iso$}}

\endpspicture
\]
---the isomorphisms $I_n$ are the same as in the case of $\FG$ (Theorem~\ref{threepointten}).  But we have already shown that the limit of the lower part of the diagram is 
\[ \psset{unit=0.1cm,labelsep=2pt,nodesep=2pt}
\pspicture(5,17)

% a1 
% a2 
\rput(0,15){\rnode{a1}{$\Top$}}  %  top
\rput(0,2){\rnode{a2}{$\cat{$\omg$-Gph}$}}  % bottom

\ncline{->}{a1}{a2} \naput{{\scriptsize $\Pi_\omg$}} % top

\endpspicture
\]
thus it is the terminal coalgebra of $\DG$ as required.
\end{proof}

%%%%%%%%%%%%%%%%%%%%%%%%%%%%%%%%%%%%%%%%%%%%%%%%%%%%%%%%%%%%%%%%%%%%%%%%%
%%%%%%%%%%%%%%%%%%%%%%%%%%%%%%%%%%%%%%%%%%%%%%%%%%%%%%%%%%%%%%%%%%%%%%%%%
% Section: The category of Trimble omega-cats 
%%%%%%%%%%%%%%%%%%%%%%%%%%%%%%%%%%%%%%%%%%%%%%%%%%%%%%%%%%%%%%%%%%%%%%%%%
%%%%%%%%%%%%%%%%%%%%%%%%%%%%%%%%%%%%%%%%%%%%%%%%%%%%%%%%%%%%%%%%%%%%%%%%%

\section{The category of Trimble $\omg$-categories}\label{sec:trimcat}
\label{trimblecat}

In this section we define Trimble $n$- and $\omg$-categories, and begin the process of 
characterising the category of Trimble $\omg$-categories as a terminal
coalgebra.  The structure of this section follows that of
Section~\ref{sec:strictcat} (strict $\omg$-categories); as with the strict case we will study the monads in the next section, and deduce the result about terminal coalgebras from the more powerful result for monads.  The notation
$n\cat{-Cat}$ will now always refer to Trimble $n$-categories, and similarly
$\omg\cat{-Cat}$.  Operads will be non-symmetric classical operads.

First we recall the notion of ``weak enrichment'' used in the Trimble
theory; the idea is to enrich in $\cV$ in a way that is weakened by the
action of an operad $P$ (see \cite{che16}).

\begin{mydefinition}\label{sevenpointone} %7.1
Let \cl{V} be a symmetric monoidal category and $P$ an operad in \cl{V}.  A
\demph{$(\cl{V},P)$-category} $A$ is given by 
\begin{itemize}
\item a \cl{V}-graph $A$, equipped with
\item for all $k \geq 0$ and $a_0, \ldots, a_k \in \cat{ob}A$ a composition
morphism
\[\gamma\: P(k) \otimes A(a_{k-1}, a_k) \otimes \cdots \otimes A(a_0, a_1)
\lra A(a_0, a_k)\] 
\end{itemize}
in \cl{V}, compatible with the composition and identities of the operad in
the usual way (as for algebras).  Morphisms are defined in the obvious way,
giving a category $(\cl{V},P)\cat{-Cat}$.
\end{mydefinition}

\begin{examples} The following degenerate examples show how this notion generalises both $\cV$-categories and $P$-algebras.
 
\begin{enumerate}
 \item If $P$ is the terminal operad then a $(\cl{V}, P)$-category is
exactly a $\cl{V}$-category.

\item A $(\cl{V},
P)$-category with only one object is exactly a $P$-algebra.

\end{enumerate}

\end{examples}

There is a particular topological operad $E$ for which $(\Top, E)$-categories are in plentiful supply; this is the operad used by Trimble in his original definition \cite{tri1,lei7}.  We will fix this operad and use it for the rest of the work.

\begin{mydefinition} %7.2

We define the operad $E$ in \cat{Top} by setting
$E(k)$ to be the space of continuous endpoint-preserving maps 
    \[ [0,1] \lra [0,k] \]
for each $k \geq 0$.  The composition maps 
\[E(m) \times E(k_1) \times \cdots \times E(k_m) \lra E(k_1 + \cdots + k_m) \]
are given by reparametrisation and the unit is given by the identity map $[0,
1] \lra [0, 1]$ in $E(1)$.

\end{mydefinition}

\begin{example}\label{neweightpointfour}\label{sevenpointfour}
 Every topological space $X$ gives rise canonically to a $(\cat{Top},
E)$-category as follows.  

\begin{itemize}
\item Its underlying graph is $\Gamma X$, defined in Definition~\ref{newsevenpointone}.

\item Composition: given points $x_0, \ldots, x_k \in X$ we have a
canonical map
\[E(k) \times X(x_{k-1}, x_k) \times \cdots \times X(x_0, x_1) \lra X(x_0,
x_k)\] 
compatible with the operad composition: given an element of the domain space, we concatenate the $k$ paths and apply the reparametrisation specified by the element of $E(k)$, thus obtaining a single path.

\end{itemize}

By a slight abuse
of notation, we refer to the $(\cat{Top}, E)$-category as $\Gamma X$, too.

\end{example}

\begin{myremark} %7.3

In all that follows, we use this operad $E$ as our starting point.  Note that other operads could be used; the only property of $E$ that is required to make the constructions work is that it ``acts on path spaces'' as in Example~\ref{sevenpointfour}.  In fact, $E$ is in a precise sense the \emph{universal} operad acting on path spaces \cite{che19}; a smaller but somewhat less elegant operad is exhibited in \cite{cg2}.  

Note, further, that for each $k \geq 0$, the space $E(k)$ is
contractible.  This is what will give \emph{coherence} for the $n$-categories
we define; but from a technical point of view the induction will not depend on
this property of $E$.

\end{myremark}

Our aim is to combine
\begin{itemize}
 \item the move from $\cV$-graphs to $\cV$-categories, as in the move from
   Section~\ref{strictgph} to Section~\ref{strictcat}, and 

\item the move from $\CAT$ to the slice $\Top/\CAT$, as in the move from
  Section~\ref{strictgph} to Section~\ref{trimblegph}. 
\end{itemize}
That is we wish define an endofunctor of the form
\[\big(\Top \lra \cV\big) \hh{1em} \lthickmapsto \hh{1em} \big(\Top \lra (\cV,P)\cat{-Cat}\big)\]
where $P$ is an operad suitably derived from the starting data.  As in Section~\ref{strictcat} we must restrict to categories with finite products, so now our base category is the slice
\[\Top/\CATp.\]

We now define the endofunctor for which the category of Trimble
$\omg$-categories will be the terminal coalgebra.  We proceed in an exactly
analogous manner to Section~\ref{sec:strict}, defining an endofunctor for
``enrichment'', except now instead of producing ordinary enriched categories,
we produce enriched categories weakened by the action of an operad.  We begin by giving the direct explicit definition, with a more abstract characterisation afterwards.

\begin{definition} %7.7
\label{defn:DC}\label{DC}
We define an endofunctor 
\[\DC \: \cat{Top}/\CATp \ltra \cat{Top}/\CATp\]
by
\[
\DC \left( \cat{Top} \tmap{\Pi} \cl{V} \right)
\ = \ 
\left( \cat{Top} \tmap{\Pi^+} (\cl{V}, \Pi E)\cat{-Cat} \right)
\]
where $\Pi^+$ is given as follows.  For a space $X$, the $(\cl{V}, \Pi E)$-category $\Pi^+ X$ has as
its underlying $\cl{V}$-graph what was called $\Pi^+ X$ in
Section~\ref{sec:trimback}, and the action of the operad $\Pi E$ is given by

\[
\psset{unit=0.1cm,labelsep=2pt,nodesep=2pt}
\pspicture(40,44)

% x1 x2 x3 x4
% x2
% x3

\rput(20,40){\rnode{x1}{$\Pi\big(E(k)\big)\times\Pi\big(X(x_{k-1},x_k)\big)\times
\cdots \times \Pi\big(X(x_{0},x_1)\big)$}}

\rput(20,22){\rnode{x2}{$\Pi\big(E(k) \times X(x_{k-1},x_k)\times \cdots
\times X(x_{0},x_1)\big)$}}

\rput(20,4){\rnode{x3}{$\Pi\big(X(x_0,x_k)\big).$}}

\psset{nodesep=2pt}
\ncline{->}{x1}{x2} \nbput{{\scriptsize $\iso$}} \naput[labelsep=10pt]{$\Pi$ preserves products}
\ncline{->}{x2}{x3} \naput[labelsep=10pt]{$\Pi$ of the action of $E$
on path spaces}

\endpspicture
\]
Note that here $\Pi E$ is the operad defined by $(\Pi E)(k) = \Pi\big((E(k)\big)$; this is an operad as $\Pi$ preserves products.  It is straightforward to show that if $\Pi$ preserves products then $\Pi^+$ also does.  
\end{definition}

We now give a more abstract description of this endofunctor, in the style of the previous section.  Recall that we defined $\DG(\Top \tmap{\Pi} \cV)$ as a composite
\[\Top \map{\Gamma} \Top\cat{-Gph} \map{\Pi_*} \cV\cat{-Gph}.\]
Similarly, we will express $\DC(\Top \tmap{\Pi} \cV)$ as
\[\Top \map{\Gamma} (\Top,E)\cat{-Cat} \map{\Pi_*} (\cV,\Pi E)\cat{-Cat}.\]
By a mild abuse of notation we will generalise the functors $\Gamma$ and $(\ )_*$ of the previous section to the case of categories, without changing the notation.  We begin with $\Gamma$.

\begin{lemma}\label{newsevenpointseven}
 The assignation $X \mapsto \Gamma X$ of Example~\ref{neweightpointfour} extends to a functor 
\[
\Gamma \: \cat{Top} \lra (\cat{Top}, E)\cat{-Cat},
\]
and the following triangle commutes
\begin{equation}
\label{eq:Gamma-fgt}
\psset{unit=0.1cm,labelsep=2pt,nodesep=3pt,npos=0.4}
\pspicture(20,22)
%    a2
% a1 
%    a3 

\rput(0,10){\rnode{a1}{$\Top$}}  %  left
\rput(27,19){\rnode{a2}{$\cat{$(\Top,E)$-Cat}$}}  % top right
\rput(27,1){\rnode{a3}{$\cat{Top-Gph}$}}  % bottom right

\ncline{->}{a1}{a2} \naput[npos=0.5]{{\scriptsize $\Gamma$}} % top
\ncline{->}{a1}{a3} \nbput[npos=0.5]{{\scriptsize $\Gamma$}} % left
\ncline[labelsep=2pt]{->}{a2}{a3} \naput{{\scriptsize $$}} % right

\endpspicture
\end{equation}
where the vertical arrow is the forgetful functor. 
\end{lemma}

We now turn our attention to the functor $(\ )_*$. Straightforward calculations give the following.

\begin{lemma}
Consider $H \: \cl{V} \lra \cl{W}$ in $\CATp$ and an operad $P$ in
$\cl{V}$.  Then there are an induced operad $HP$ in $\cl{W}$ given by $(HP)(k) = H(P(k))$ and an induced
functor
\[
H_*\: (\cl{V}, P)\cat{-Cat} \lra (\cl{W}, HP)\cat{-Cat}.
\]
Furthermore $H_*$ preserves finite products, that is, lies in $\CATp$.  

\end{lemma}

The endofunctor $\DC$ of $\cat{Top}/\CATp$ can now be described abstractly
by 
\[
\DC \left( \cat{Top} \tmap{\Pi} \cl{V} \right)
\ = \ 
\left( \cat{Top} \tmap{\Pi^+} (\cl{V}, \Pi E)\cat{-Cat} \right)
\]
where $\Pi^+$ is the composite functor
\[
\cat{Top}
\map{\Gamma}
(\cat{Top}, E)\cat{-Cat}
\map{\Pi_*}
(\cl{V}, \Pi E)\cat{-Cat}
\]
(which is in $\CATp$ by the preceding results).  

We now give Trimble's definition of $n$-category inductively by applying $\DC$ repeatedly.

\begin{definition}

For
each $n \geq 0$, we define simultaneously
\begin{itemize}
\item a category $n\cat{-Cat}$ with finite products, whose objects are called
\demph{Trimble $n$-categories} (but whose morphisms should be thought of as
\emph{strict} $n$-functors)

\item a finite product preserving functor $\Pi_n\: \cat{Top} \lra n\cat{-Cat}$,
the \demph{fundamental $n$-groupoid} functor.
\end{itemize}
We use the fundamental $(n - 1)$-groupoid of each space $E(k)$ to parametrise
$k$-ary composition in an $n$-category.  The definitions are:
\begin{itemize}
\item $0\cat{-Cat} = \cat{Set}$, and $\Pi_0\: \cat{Top} \lra \cat{Set}$ is the
functor sending a space to its set of path components
\item for $n \geq 1$,
\[
\left( \cat{Top} \mtmap{\Pi_n} n\cat{-Cat} \right)
\ = \ 
\DC
\left( \cat{Top} \mtmap{\Pi_{n - 1}} (n - 1)\cat{-Cat} \right).
\]
\end{itemize}
\end{definition}

{Incoherent} Trimble $n$-categories are defined similarly.  As before,
the key is to find the part of the definition that gives the coherence axioms,
and remove it.  In this case it comes from defining $\Pi_0 X$ as the set of path components of $X$---this ensures coherence at
the top dimension.  So we simply remove this and replace it with the set of points, as follows.

\begin{definition}
For each $n \geq 0$, we define simultaneously
\begin{itemize}
\item a category $n\cat{-iCat}$ with finite products, whose objects are called
\demph{incoherent Trimble $n$-categories}
\item a finite product preserving functor $\iPi{n}\: \cat{Top} \lra
n\cat{-iCat}$
\end{itemize}
as follows.
\begin{itemize}
\item $\cat{$0$-iCat} = \cat{Set}$, and $\iPi{0}\: \cat{Top} \lra \cat{Set}$ is the
functor sending a space to its set of points
\item for $n \geq 1$,
\[
\left( \cat{Top} \mtmap{\iPi{n}} \cat{$n$-iCat} \right)
\ = \ 
\DC
\left( \cat{Top} \mtmap{\iPi{n - 1}} \cat{$(n-1)$-iCat} \right).
\]
\end{itemize}

(The re-use of the notation $\iPi{n}$ from Definition~\ref{defn:Piin-Trim}
will be justified shortly.)  We also define, for each $n \geq 1$, a
morphism
\[ U_n \: \left(\pspicture(-6.5,7.5)(6.5,18)

% a1 
% a2 
\rput(0,15){\rnode{a1}{$\Top$}}  %  top
\rput(0,2){\rnode{a2}{$\cat{$n$-iCat}$}}  % bottom

\ncline{->}{a1}{a2} \naput{{\scriptsize $\iPi{n}$}} % top

\endpspicture\right)
\tra
\left(\pspicture(-11.5,7.5)(11.5,18)

% a1 
% a2 
\rput(0,15){\rnode{a1}{$\Top$}}  %  top
\rput(0,2){\rnode{a2}{$\cat{$(n-1)$-iCat}$}}  % bottom

\ncline[nodesepB=2pt]{->}{a1}{a2} \naput{{\scriptsize $\iPi{n-1}$}} % top

\endpspicture\right)
\]
in $\Top/\CATp$ as follows:
\begin{itemize}
\item $U_1\: 1\cat{-iCat}  = (\cat{Set}, \iPi{0} E)\cat{-Cat} \lra \cat{Set}$ is
the functor that takes the set of objects
\item for $n \geq 2$, $U_n = \DC(U_{n - 1})$.
\end{itemize}
\end{definition}

In Section~\ref{sec:weak} we saw that, for one notion of weak
higher-dimensional category, the category of weak $\omg$-categories is the
limit over all $n$ of the categories of incoherent $n$-categories.  Motivated
by this, we make the definition of Trimble weak $\omg$-categories analogously; in fact we take a limit over all the $\iPi{n}$ so that we get a fundamental $\omg$-groupoid functor at the same time.

\begin{mydefinition} %7.9
\label{sevenpointnine}
Define
\[\psset{unit=0.1cm,labelsep=2pt,nodesep=2pt}
\pspicture(5,17)

% a1 
% a2 
\rput(0,15){\rnode{a1}{$\Top$}}  %  top
\rput(0,2){\rnode{a2}{$\cat{$\omg$-Cat}$}}  % bottom

\ncline{->}{a1}{a2} \naput{{\scriptsize $\Pi_\omg$}} % top

\endpspicture\]
to be the limit in $\cat{Top}/\CAT$ of the diagram 
\begin{equation}\label{diagnine}
\cdots 
\map{U_3}
\left(\pspicture(-6,7.5)(6,18)

% a1 
% a2 
\rput(0,15){\rnode{a1}{$\Top$}}  %  top
\rput(0,2){\rnode{a2}{$\cat{$2$-iCat}$}}  % bottom

\ncline[nodesep=2pt]{->}{a1}{a2} \naput{{\scriptsize $\iPi{2}$}} % top

\endpspicture\right)
\map{U_2}
\left(\pspicture(-6,7.5)(6,18)

% a1 
% a2 
\rput(0,15){\rnode{a1}{$\Top$}}  %  top
\rput(0,2){\rnode{a2}{$\cat{$1$-iCat}$}}  % bottom

\ncline[nodesep=2pt]{->}{a1}{a2} \naput{{\scriptsize $\iPi{1}$}} % top

\endpspicture\right)
\map{U_1}
\left(\pspicture(-6,7.5)(6,18)

% a1 
% a2 
\rput(0,15){\rnode{a1}{$\Top$}}  %  top
\rput(0,2){\rnode{a2}{$\cat{$0$-iCat}$}}  % bottom

\ncline[nodesep=2pt]{->}{a1}{a2} \naput{{\scriptsize $\iPi{0}$}} % top

\endpspicture\right)
\end{equation}
As in the previous section, Proposition~\ref{prop:coslice-creates} tells us that this limit exists and $\cat{$\omg$-Cat}$ is the limit of the categories $\cat{$n$-iCat}$.  Objects of
$\omg\cat{-Cat}$ are \demph{Trimble $\omg$-categories}, and $\Piom$ is the
\demph{fundamental $\omg$-groupoid} functor.  The $n$th projection of the
limit is written $\tr_n$ and called \demph{truncation}.

\end{mydefinition}

\begin{remark}
As for the underlying graph version (Remark~\ref{sixprecise}), the commuting triangles 
\[\pspicture(20,20)
%    a2
% a1 
%    a3 
\rput(0,10){\rnode{a1}{$\Top$}}  %  left
\rput(19,17){\rnode{a2}{$\cat{$\omg$-Cat}$}}  % top right
\rput(19,3){\rnode{a3}{$\cat{$n$-Cat}$}}  % bottom right

\psset{labelsep=2pt}
\ncline{->}{a1}{a2} \naput[npos=0.55]{{\scriptsize $\Piom$}} % top
\ncline{->}{a1}{a3} \nbput[npos=0.55]{{\scriptsize $\iPi{n}$}} % left
\ncline[labelsep=2pt]{->}{a2}{a3} \naput{{\scriptsize $\mbox{tr}_n$}} % right

\endpspicture\]
tell us how $\Piom$ behaves.
\end{remark}

We will prove that $\left(\cat{Top} \tmap{\Piom} \omg\cat{-Cat}\right)$ is the terminal
coalgebra for $\DC$, using Ad\'amek's Theorem.  As in the strict case, we will deduce this from the result for monads in the next section, but the following result gives a strong indication that the result is true.  As usual, we compare the sequential diagram whose limit defines our $\omg$-categories, with the sequential diagram whose limit appears in Ad\'amek's Theorem.

\begin{lemma} %7.10
\label{lemma:DC-ladder}

There are canonical isomorphisms $(I_n)_{n \geq 0}$ such that the following
diagram commutes:
\[
\psset{unit=0.1cm,labelsep=2pt,nodesep=2pt}
\pspicture(100,44)

% x1 x2 x3
% y1 y2 y3

\rput(0,40){$\cdots$}
\rput(0,10){$\cdots$}

\rput(0,40){\rnode{x0}{\white{$\DG^3\1$}}}
\rput(30,40){\rnode{x1}{$\DC^3\1$}}
\rput(65,40){\rnode{x2}{$\DC^2\1$}}
\rput(100,40){\rnode{x3}{$\DC\1$}}

\psset{nodesep=4pt}
\ncline{->}{x0}{x1} \naput{{\scriptsize $\DC^3!$}}
\ncline{->}{x1}{x2} \naput{{\scriptsize $\DC^2!$}}
\ncline{->}{x2}{x3} \naput{{\scriptsize $\DC!$}}

\rput(3,10){\rnode{y0}{\white{$\DG^3\1$}}}

\rput(30,10){\rnode{y1}{
$\left(\pspicture(-6,7.5)(6,18)
% a1 
% a2 
\rput(0,15){\rnode{a1}{$\Top$}}  %  top
\rput(0,2){\rnode{a2}{$\cat{2-iCat}$}}  % bottom
\ncline{->}{a1}{a2} \naput{{\scriptsize $\iPi{2}$}} % top

\endpspicture\right)$}}

\rput(65,10){\rnode{y2}{
$\left(\pspicture(-6,7.5)(6,18)
% a1 
% a2 
\rput(0,15){\rnode{a1}{$\Top$}}  %  top
\rput(0,2){\rnode{a2}{$\cat{1-iCat}$}}  % bottom
\ncline{->}{a1}{a2} \naput{{\scriptsize $\iPi{1}$}} % top

\endpspicture\right)$}}

\rput(100,10){\rnode{y3}{
$\left(\pspicture(-6,7.5)(6,18)
% a1 
% a2 
\rput(0,15){\rnode{a1}{$\Top$}}  %  top
\rput(0,2){\rnode{a2}{$\cat{0-iCat}$}}  % bottom
\ncline{->}{a1}{a2} \naput{{\scriptsize $\iPi{0}$}} % top

\endpspicture\right)$}}

\psset{nodesep=2pt}
\ncline{->}{y0}{y1} \naput{{\scriptsize $U_3$}}
\ncline{->}{y1}{y2} \naput{{\scriptsize $U_2$}}
\ncline{->}{y2}{y3} \naput{{\scriptsize $U_1$}}

\psset{nodesepA=5pt,nodesepB=3pt}
\ncline{->}{x1}{y1} \naput{{\scriptsize $I_2$}} \nbput{{\scriptsize $\iso$}}
\ncline{->}{x2}{y2} \naput{{\scriptsize $I_1$}} \nbput{{\scriptsize $\iso$}}
\ncline{->}{x3}{y3} \naput{{\scriptsize $I_0$}} \nbput{{\scriptsize $\iso$}}

\endpspicture
\]

\end{lemma}

\begin{proof}
As usual, it is straightforward to define $I_0$ and verify the commutativity
of the rightmost square, and the rest follows by repeated application of $\DC$.
\end{proof}

\begin{theorem} %7.11
\label{thm:trim-cat-tc}

The category of Trimble $\omg$-categories together with the fundamental
$\omg$-groupoid functor
\[\psset{unit=0.1cm,labelsep=2pt,nodesep=2pt}
\pspicture(5,17)

% a1 
% a2 
\rput(0,15){\rnode{a1}{$\Top$}}  %  top
\rput(0,2){\rnode{a2}{$\cat{$\omg$-Cat}$}}  % bottom

\ncline{->}{a1}{a2} \naput{{\scriptsize $\Pi_\omg$}} % top

\endpspicture\]
is the terminal coalgebra for $\DC$.  

\end{theorem}

We defer the proof until the end of the next section.

%%%%%%%%%%%%%%%%%%%%%%%%%%%%%%%%%%%%%%%%%%%%%%%%%%%%%%%%%%%%%%%%%%%%%%%%%
%%%%%%%%%%%%%%%%%%%%%%%%%%%%%%%%%%%%%%%%%%%%%%%%%%%%%%%%%%%%%%%%%%%%%%%%%
% Section: The monad for Trimble omega-cats
%%%%%%%%%%%%%%%%%%%%%%%%%%%%%%%%%%%%%%%%%%%%%%%%%%%%%%%%%%%%%%%%%%%%%%%%%
%%%%%%%%%%%%%%%%%%%%%%%%%%%%%%%%%%%%%%%%%%%%%%%%%%%%%%%%%%%%%%%%%%%%%%%%%

\section{The monad for Trimble $\omg$-categories}\label{sec:trimmonad}
\label{trimblemnd}

We continue with a process analogous to that for strict $\omg$-categories.
We will show that there is a monad for Trimble $\omg$-categories and that
it is the terminal coalgebra for a certain endofunctor.

As in the previous section, we will need to encode the ``fundamental $n$-category'' functors as part of this data.  The idea is that our endofunctor should act on data such as:
\begin{itemize}
 \item the category of $n$-graphs \cat{$n$-Gph}, together with
\item the monad for incoherent $n$-categories $\iT{n}$ on it, equipped with
\item the fundamental incoherent $n$-category functor $\Top \lra \cat{$n$-iCat}$
\end{itemize}
and produce the $(n+1)$-dimensional version.  Expressing \cat{$n$-iCat} as $\cat{Alg}(\iT{n})$ leads us to consider the comma category $\Top \downarr \Alg$, but as in the analysis of the monad for strict $\omg$-categories we must restrict to infinitely distributive categories and the appropriate functors.  Thus, we work with the following category.  Recall from
Proposition~\ref{prop:FTM-dist} that $\Alg\: \MND \lra \CAT$ restricts to a
functor $\MNDd \lra \CATd$, which we also call $\Alg$ or, for emphasis,
$\Algd$.  

\begin{definition}
We denote by $\cat{Top} \downarr \Algd$ the comma category whose objects are triples $(\cl{V}, T, \Pi)$ where
\begin{itemize}

\item \cl{V} is an infinitely distributive category

\item $T$ is a monad on \cl{V} whose functor part preserves coproducts, and

\item $\Pi$ is a functor $\cat{Top} \lra \cl{V}^T$ preserving coproducts and
finite products.

\end{itemize}
We often write an object $(\cl{V}, T, \Pi)$ as 
\[
\left( \cat{Top} \tmap{\Pi} \cl{V}^T \right).
\]
\end{definition}

\begin{example}

For a prototype example of such an object, consider $\cl{V} = \cat{Gph}$, $T=
\cat{fc}$, $\Pi =$ the fundamental groupoid functor, that is
\[\Pi \: \Top \tra \Gph^{\cat{fc}} = \Cat.\]

\end{example}

Viewing lax morphisms of monads as commutative squares, a morphism
\[
(\cl{V}, T, \Pi) \lra (\cl{V}', T', \Pi') 
\]
consists of functors $H$ and $K$, preserving finite products and
small coproducts, such that the following diagram commutes:
\begin{equation}
\label{eq:house}
\psset{unit=0.1cm,labelsep=2pt,nodesep=2pt}
\pspicture(0,-1)(20,37)

%   a0
% a1 a2
% a3 a4

%%%%%%%%%% top
\rput(10,35){\rnode{a0}{$\Top$}}  % top 

\rput(0,18){\rnode{a1}{$\cV^T$}}  % top left
\rput(20,18){\rnode{a2}{$\cV'^{T'}$}}  % top right
\rput(0,0){\rnode{a3}{$\cV$}}  % bottom left
\rput(20,0){\rnode{a4}{$\cV'$}}  % bottom right

\ncline{->}{a0}{a1} \nbput{{\scriptsize $\Pi$}} % left diag
\ncline{->}{a0}{a2} \naput{{\scriptsize $\Pi'$}} % right diag

\ncline{->}{a1}{a2} \naput{{\scriptsize $K$}} % top
\ncline{->}{a3}{a4} \nbput{{\scriptsize $H$}} % bottom
\ncline{->}{a1}{a3} \nbput{{\scriptsize $G^T$}} % left
\ncline{->}{a2}{a4} \naput{{\scriptsize $G^{T'}$}} % right

\endpspicture
\end{equation}
Alternatively, it is a morphism $(H, \theta)\: (\cl{V}, T) \lra (\cl{V}', T')$
making the triangle in~(\ref{eq:house}) commute, where $K$ is the
induced functor on algebras.  

A standard argument proves:

\begin{lemma}
The adjunction $\cat{Inc} \ladj \Alg$ of Proposition~\ref{prop:FTM-dist}
induces an isomorphism of categories
\[
\cat{Top} \downarr \Algd 
\iso
(\cat{Top}, 1) / \MNDd.
\]
(Here $1$ is the identity monad on $\cat{Top}$.)  
\end{lemma}

From this point of
view, an object of the category is an infinitely distributive monad $(\cl{V},
T)$ together with a lax morphism of monads $(\cat{Top}, 1) \lra
(\cl{V}, T)$ preserving finite products and small coproducts.  A morphism in
this category is a commutative triangle.  We will use both points of view; the slice category point of view will enable us to calculate certain limits more easily.

Next we define the endofunctor $\DM$ of $\cat{Top} \downarr \Algd$, whose
terminal coalgebra will turn out to be the monad for Trimble
$\omg$-categories.  First it is useful to note the following result.

 \begin{lemma} %7.8
\label{lemma:DC-restricts}

The endofunctors $\DG$ and $\DC$ of $\Top/\CAT$ each restrict to an endofunctor of $\cat{Top}/\CATd$.
\end{lemma}

\begin{proof}
Recall that $\DG$ is defined by
\[\DG(\Top \tmap{\Pi}\cV) = (\Top \tmap{\Gamma} \cat{Top-Gph} \tmap{\Pi_*} \cat{$\cV$-Gph}).\]
Now if $\cV$ and $\Pi$ are in $\CATd$ then certainly \cat{$\cV$-Gph} and
$\Pi_*$ are; we need to show that $\Gamma$ is in $\CATd$, that is, that it
commutes with finite products and small coproducts.  The latter is
straightforward; the former follows from the definition of the path space
$X(x,y)$ as the equaliser of the pair 
\[\psset{labelsep=2pt,nodesep=2pt}
\pspicture(0,-2)(35,7)
\rput(0,3){\rnode{a1}{\raisebox{3pt}{$X^{[0,1]}$}}}
\rput(28,3){\rnode{a2}{$X \times X$}}
\psset{nodesep=3pt,arrows=->}
\ncline[offset=3pt]{a1}{a2}\naput{{\scriptsize $(\cat{ev}_0, \cat{ev}_1)$}}
\ncline[offset=-3pt]{a1}{a2}\nbput{{\scriptsize $(\Delta x, \Delta y)$}}
\endpspicture
\]
where $X^{[0, 1]}$ is the exponential in $\cat{Top}$ (which exists since $[0,
1]$ is compact Hausdorff), $\mathrm{ev}_0$ and $\mathrm{ev}_1$ are evaluation
at $0$ and $1$, and $\Delta x$ and $\Delta y$ are constant at $x$ and $y$. 

We now turn our attention to $\DC$.  Recall that we have the definition
\[\DC\left(\Top \tmap{\Pi}\cV\right) \ = \ \left(\Top \tmap{\Gamma} \cat{$(\Top,E)$-Cat} \tmap{\Pi_*} \cat{$(\cV,\Pi E)$-Cat}\right)\]
with $\Gamma$ defined in Lemma~\ref{newsevenpointseven}.  If $\cV$ and $\Pi$ are in $\CATd$ then a straightforward calculation shows that \cat{$(\cV,\Pi E)$-Cat} and $\Pi_*$ are too.  We now show that $\Gamma \in \CATd$.  We use the commuting triangle from Lemma~\ref{newsevenpointseven}
\[\psset{unit=0.1cm,labelsep=2pt,nodesep=3pt,npos=0.4}
\pspicture(27,22)
%    a2
% a1 
%    a3 

\rput(0,10){\rnode{a1}{$\Top$}}  %  left
\rput(27,19){\rnode{a2}{$\cat{$(\Top,E)$-Cat}$}}  % top right
\rput(27,1){\rnode{a3}{$\cat{Top-Gph}$}}  % bottom right

\ncline{->}{a1}{a2} \naput[npos=0.5]{{\scriptsize $\Gamma$}} % top
\ncline{->}{a1}{a3} \nbput[npos=0.5]{{\scriptsize $\Gamma$}} % left
\ncline[labelsep=2pt]{->}{a2}{a3} \naput{{\scriptsize $$}} % right

\endpspicture\]
where the vertical arrow is the forgetful functor.  A routine calculation
shows that if $\cV \in \CATd$ then the forgetful functor
\[\cat{$(\cV,P)$-Cat} \lra \cat{$\cV$-Gph}\]
creates finite products and small coproducts, and the result then follows
from our analysis of $\Gamma: \Top \lra \cat{Top-Gph}$ above.  \end{proof} 

We now define the crucial endofunctor on $\Top \downarr \Algd$.  We proceed analogously to our process of defining $\FM$, with two propositions analogous to \ref{fourpointeight} and \ref{fourpointnine}.  The following result generalises Kelly's theorem~(\ref{prop:kelly}) on free
enriched categories.

\begin{proposition} %8.1
\label{prop:VP-kelly}

Let $\cl{V}$ be an infinitely distributive category and $P$ an operad in
$\cl{V}$.  Then the forgetful functor
\[
(\cl{V}, P)\cat{-Cat} \lra \cl{V}\cat{-Gph}
\]
is monadic.  The induced monad is the ``free $(\cl{V}, P)$-category monad''
$\fc{(\cl{V}, P)}$, which is the identity on underlying sets of objects, and
is given on hom-sets as follows: for a $\cl{V}$-graph $A$ and objects $a, a'
\in A$, 
\[\left(\fc{(\cl{V},P)} A\right)(a,a') = 
\coprod_{k \in \mathbb{N}, a=a_0, a_1, \ldots, a_{k-1}, a_k=a'}
P(k) \times A(a_{k-1},a_k) \times \cdots \times A(a_0,a_1).\] 

\end{proposition}

\begin{proof}
See Proposition~2.11 of~\cite{che16}.
\end{proof}

Suppose now that we are given an infinitely distributive category $\cl{V}$, an
operad $P$ in $\cl{V}^T$, and a coproduct-preserving monad $T$ on $\cl{V}$.
Since the forgetful functor
\[
G^T\: \cl{V}^T \lra \cl{V}
\]
preserves products, there is an induced operad $G^T P$ in $\cl{V}$, which we
shall often simply refer to as $P$.  We then have a functor
\[
G^T_* \:
(\cl{V}^T, P)\cat{-Cat}
\lra
(\cl{V}, G^T P)\cat{-Cat} = (\cl{V}, P)\cat{-Cat}.
\]
We also have two monads on $\cl{V}\cat{-Gph}$: the monad $\fc{(\cl{V}, P)}$ of
Proposition~\ref{prop:VP-kelly}, and the monad $T_* = \FG(T)$.  The following
result generalises the one for ordinary enriched categories (Proposition~\ref{prop:fc-T}).

\begin{proposition} %8.2
\label{prop:fc-VP}\label{eightpointtwo}

Let $(\cl{V}, T) \in \MNDd$ and let $P$ be an operad in $\cl{V}^T$.  Then:
\begin{enumerate}
\item 
\label{part:diag-monadic}
The diagonal of the commutative square of forgetful functors
\[
\psset{unit=0.1cm,labelsep=3pt,nodesep=3pt}
\pspicture(30,24)

% a1 a2
% a3 a4

%%%%%%%%%% top

\rput(0,20){\rnode{a1}{\cat{$(\cV^T,P)$-Cat}}}  % top left
\rput(30,20){\rnode{a2}{\cat{$\cV^T$-Gph}}}  % top right
\rput(0,0){\rnode{a3}{\cat{$(\cV,P)$-Cat}}}  % bottom left
\rput(30,0){\rnode{a4}{\cat{$\cV$-Gph}}}  % bottom right

\ncline{->}{a1}{a2} \naput{{\scriptsize $$}} % top
\ncline{->}{a3}{a4} \nbput{{\scriptsize $$}} % bottom
\ncline{->}{a1}{a3} \nbput{{\scriptsize $G^T_*$}} % left
\ncline{->}{a2}{a4} \naput{{\scriptsize $G^T_*$}} % right

\endpspicture
\]

is monadic.

\item 
\label{part:distrib-law}
There is a canonical distributive law of $T_*$ over $\fc{(\cl{V},P)}$,
and the resulting monad $\fc{(\cl{V},P)} \circ T_*$ on $\cl{V}\cat{-Gph}$ is
the monad induced by this diagonal functor.

\item 
\label{part:distrib-distrib}
$(\cl{V}\cat{-Gph}, \fc{(\cl{V},P)} \circ T_*) \in \MNDd$.

\end{enumerate}

\end{proposition}

\begin{proof}
Parts~(\ref{part:diag-monadic}) and~(\ref{part:distrib-law}) are Theorem~4.1
of~\cite{che16}.  Part~(\ref{part:distrib-distrib}) is a straightforward calculation (cf. the analogous part of Proposition~\ref{prop:fc-T}).
\end{proof}

\begin{definition}\label{DM}
 We define an endofunctor
\[\DM \: \tdalgd \lra \tdalgd\]
by
\[\DM(\cV, T, \Pi) = (\cat{$\cV$-Gph}, T^+, \Pi^+)\]
as follows.  Recall that this notation means $\Pi$ is a functor
\[\Top \map{\Pi} \cV^T.\]

\begin{itemize}
 \item $T^+$ is the monad $\fc{(\cl{V}, \Pi E)} \circ T_*$
on $\cl{V}\cat{-Gph}$, corresponding to the (monadic) forgetful functor
\[
(\cl{V}^T, \Pi E)\cat{-Cat} \lra \cl{V}\cat{-Gph}.
\]
Note that as $\Pi$ preserves products it does indeed induce an operad $\Pi E$ in $\cV^T$.

\item $\Pi^+$ is the functor
\[\Top \map{\Gamma} \cat{$(\Top, E)$-Cat} \map{\Pi_*} \cat{$(\cV^T,\Pi E)$-Cat} \iso \cat{$\cV$-Gph}^{T^+}.\]

\end{itemize}

Note that both $(\cl{V}^T, \Pi E)\cat{-Cat}$ and
$\Pi^+$ are in $\CATd$ (Lemma~\ref{lemma:DC-restricts}), so $(\cl{V}\cat{-Gph}, T^+, \Pi^+)$ is indeed an object of $\tdalgd$.

We now define $\DM$ on morphisms.  Given a morphism
\[
(\cl{V}, T, \Pi)
\lra
(\cl{V}', T', \Pi')
\]
expressed as
\begin{equation}
\psset{unit=0.1cm,labelsep=2pt,nodesep=2pt}
\pspicture(0,-3)(20,38)

%   a0
% a1 a2
% a3 a4

%%%%%%%%%% top
\rput(10,35){\rnode{a0}{$\Top$}}  % top 

\rput(0,18){\rnode{a1}{$\cV^T$}}  % top left
\rput(20,18){\rnode{a2}{$\cV'^{T'}$}}  % top right
\rput(0,0){\rnode{a3}{$\cV$}}  % bottom left
\rput(20,0){\rnode{a4}{$\cV'$}}  % bottom right

\ncline{->}{a0}{a1} \nbput{{\scriptsize $\Pi$}} % left diag
\ncline{->}{a0}{a2} \naput{{\scriptsize $\Pi'$}} % right diag

\ncline{->}{a1}{a2} \naput{{\scriptsize $K$}} % top
\ncline{->}{a3}{a4} \nbput{{\scriptsize $H$}} % bottom
\ncline{->}{a1}{a3} \nbput{{\scriptsize $G^T$}} % left
\ncline{->}{a2}{a4} \naput{{\scriptsize $G^{T'}$}} % right

\endpspicture
\end{equation}
in $\cat{Top} \downarr \Algd$, there
is an induced morphism
\[
(\cl{V}\cat{-Gph}, T^+, \Pi^+)
\lra
(\cl{V}'\cat{-Gph}, T'^+, \Pi'^+)
\]
in $\cat{Top} \downarr \Algd$, namely
\[
\psset{unit=0.1cm,labelsep=2pt,nodesep=2pt}
\pspicture(0,-3)(34,38)

%   a0
% a1 a2
% a3 a4

%%%%%%%%%% top
\rput(17,35){\rnode{a0}{$\Top$}}  % top 

\rput(0,18){\rnode{a1}{\cat{$(\cV^T \hh{-2pt},\Pi E)$-Cat}}}  % top left
\rput(34,18){\rnode{a2}{\cat{$(\cV'^{T'} \hh{-3pt},\Pi' E)$-Cat}}}  % top right
\rput(0,0){\rnode{a3}{\cat{$\cV$-Gph}}}  % bottom left
\rput(34,0){\rnode{a4}{\cat{$\cV'$-Gph}}}  % bottom right

\ncline{->}{a0}{a1} \nbput[labelsep=-1pt,npos=0.45]{{\scriptsize $\Pi^+$}} % left diag
\ncline{->}{a0}{a2} \naput[labelsep=2pt,npos=0.53]{{\scriptsize ${\Pi'}^+$}} % right diag

\ncline{->}{a1}{a2} \naput{{\scriptsize $K_*$}} % top
\ncline{->}{a3}{a4} \nbput{{\scriptsize $H_*$}} % bottom
\ncline{->}{a1}{a3} \nbput{{\scriptsize $$}} % left
\ncline{->}{a2}{a4} \naput{{\scriptsize $$}} % right

\endpspicture
\]
This is a morphism in $\tdalgd$ as $\DG$ and $\DC$ both restrict to $\Top/\CATd$.  This completes the definition of $\DM$.

\end{definition}

\begin{remark}
We can also directly describe the effect of $\DM$ on the isomorphic category
$(\cat{Top}, 1)/\MNDd$.  The image under $\DM$ of an object
\[
(\Lambda, \lambda)\: (\cat{Top}, 1) \lra (\cl{V}, T)
\]
is an object of the form
\[
(\cat{Top}, 1) \lra (\cl{V}\cat{-Gph}, T^+)
\]
whose underlying functor is $\Lambda^+\: \cat{Top} \lra \cl{V}\cat{-Gph}$ (as
defined in Section~\ref{sec:trimback}).  The image under $\DM$ of a morphism

\[
\psset{unit=0.1cm,labelsep=2pt,nodesep=2pt}
\pspicture(40,20)

%   a1 
% a2  a3 

%%%%%%%%%% top

\rput(20,15){\rnode{a1}{$(\Top,1)$}}  % top 
\rput(0,0){\rnode{a2}{$(\cV,T)$}}  % bottom left
\rput(40,0){\rnode{a3}{$(\cV',T')$}}  % bottom right

\ncline{->}{a1}{a2} \nbput{{\scriptsize $(\Lambda,\lambda)$}} % left
\ncline{->}{a1}{a3} \naput{{\scriptsize $(\Lambda',\lambda')$}} % right
\ncline{->}{a2}{a3} \nbput{{\scriptsize $(H,\theta)$}} % bottom

\endpspicture
\]
is
\[
\psset{unit=0.1cm,labelsep=2pt,nodesep=2pt}
\pspicture(0,-3)(40,20)

%   a1 
% a2  a3 

%%%%%%%%%% top

\rput(20,15){\rnode{a1}{$(\Top,1)$}}  % top 
\rput(0,0){\rnode{a2}{$(\cat{$\cV$-Gph}, T^+)$}}  % bottom left
\rput(40,0){\rnode{a3}{$(\cat{$\cV'$-Gph},{T'}^+)$}}  % bottom right

\ncline{->}{a1}{a2} \nbput{{\scriptsize $$}} % left
\ncline{->}{a1}{a3} \naput{{\scriptsize $$}} % right
\ncline{->}{a2}{a3} \nbput{{\scriptsize $(H_*,\theta^+)$}} % bottom

\endpspicture
\] 
where $\theta^+$ is the composite natural transformation
\begin{equation}
\label{eq:stacked-squares}
\psset{unit=0.1cm,labelsep=0pt,nodesep=3pt}
\pspicture(23,42)

% a1 a2
% b1 b2
% c1 c2

\rput(-1,40){\rnode{a1}{\cat{$\cV$-Gph}}} % top left
\rput(21,40){\rnode{a2}{\cat{$\cV'$-Gph}}} % top mid

\rput(-1,20){\rnode{b1}{\cat{$\cV$-Gph}}}   % bottom left
\rput(21,20){\rnode{b2}{\cat{$\cV'$-Gph}}}  % bottom mid

\rput(-1,0){\rnode{c1}{\cat{$\cV$-Gph}}}   % bottom left
\rput(21,0){\rnode{c2}{\cat{$\cV'$-Gph}}}  % bottom mid

\psset{nodesep=3pt,labelsep=2pt,arrows=->}
\ncline{a1}{a2}\naput{{\scriptsize $H_*$}} % top
\ncline{b1}{b2}\nbput{{\scriptsize $H_*$}} % mid
\ncline{c1}{c2}\nbput{{\scriptsize $H_*$}} % bottom

\ncline{a1}{b1}\nbput{{\scriptsize $T_*$}} % left
\ncline{b1}{c1}\nbput{{\scriptsize $\cat{fc}_{(\cV,\Lambda E)}$}} % left

\ncline{a2}{b2}\naput{{\scriptsize $T'_*$}} % right
\ncline{b2}{c2}\naput{{\scriptsize $\cat{fc}_{(\cV'\hh{-2pt},\hh{1pt} \Lambda' E)}$}} % right

\psset{labelsep=1.5pt,nodesep=4pt}

\pnode(13,33){a3}
\pnode(7,27){b3}
\ncline[doubleline=true,arrowinset=0.6,arrowlength=0.8,arrowsize=0.5pt 2.1]{a3}{b3} \naput[npos=0.4]{{\scriptsize $\theta_*$}}

\pnode(13,13){a3}
\pnode(7,7){b3}
\ncline[doubleline=true,arrowinset=0.6,arrowlength=0.8,arrowsize=0.5pt 2.1]{a3}{b3} \naput[npos=0.4]{{\scriptsize $\iso$}}

\endpspicture
\end{equation}
and the bottom square is the isomorphism induced by $H$ preserving products
and coproducts.
\end{remark}

As in Section~\ref{strictcat} we will transfer results about $\DM$ to results about $\DC$.  We now make precise the relationship between those functors.  First note that the functors
\[
\Und, \Alg\: \MNDd \lra \CATd
\]
of Proposition~\ref{prop:FTM-dist} induce functors
\[
\Und, \Alg\: 
(\cat{Top}, 1)/\MNDd \lra \cat{Top}/\CATd.
\]

\begin{lemma} %8.4
\label{lemma:comparing-Ds}\label{eightpointfour}

The squares
\[
\psset{unit=0.1cm,labelsep=3pt,nodesep=3pt}
\pspicture(40,25)

% a1 a2
% a3 a4

%%%%%%%%%% top

\rput(0,20){\rnode{a1}{$(\Top,1)/\MNDd$}}  % top left
\rput(40,20){\rnode{a2}{$(\Top,1)/\MNDd$}}  % top right
\rput(0,0){\rnode{a3}{$\Top/\CATd$}}  % bottom left
\rput(40,0){\rnode{a4}{$\Top/\CATd$}}  % bottom right

\ncline{->}{a1}{a2} \naput{{\scriptsize $\DM$}} % top
\ncline{->}{a3}{a4} \nbput{{\scriptsize $\DG$}} % bottom
\ncline{->}{a1}{a3} \nbput{{\scriptsize $\Und$}} % left
\ncline{->}{a2}{a4} \naput{{\scriptsize $\Und$}} % right

\endpspicture
\]

\[
\psset{unit=0.1cm,labelsep=3pt,nodesep=3pt}
\pspicture(0,-4)(40,25)

% a1 a2
% a3 a4

%%%%%%%%%% top

\rput(0,20){\rnode{a1}{$(\Top,1)/\MNDd$}}  % top left
\rput(40,20){\rnode{a2}{$(\Top,1)/\MNDd$}}  % top right
\rput(0,0){\rnode{a3}{$\Top/\CATd$}}  % bottom left
\rput(40,0){\rnode{a4}{$\Top/\CATd$}}  % bottom right

\ncline{->}{a1}{a2} \naput{{\scriptsize $\DM$}} % top
\ncline{->}{a3}{a4} \nbput{{\scriptsize $\DC$}} % bottom
\ncline{->}{a1}{a3} \nbput{{\scriptsize $\Alg$}}% left
\ncline{->}{a2}{a4} \naput{{\scriptsize $\Alg$}} % right

\endpspicture
\]
commute, the first strictly and the second up to a canonical isomorphism.

\end{lemma}

\begin{proof}
Follows from Proposition~\ref{eightpointtwo}: compare the proof of Lemma~\ref{lemma:comparing-Fs}.
\end{proof}

We can now use $\DM$ iteratively to define the monads for incoherent $n$-categories; this is analogous to Definition~\ref{newfourpointtwelve}.

\begin{definition}\label{neweightpointten}

We define a sequence of objects 
\[\left(n\cat{-Gph}, \iT{n}, \iPi{n}\right) \in \cat{Top}\downarr \Algd\]
as follows.
\begin{itemize}
\item $\iT{0}$ is the identity monad on $\cat{Set} = 0\cat{-Gph}$, and
\[\iPi{0}\: \cat{Top} \lra \cat{Set} \iso 0\cat{-Gph}^{\iT{0}}\] 
is the
set-of-points functor, and

\item for $n \geq 1$,
\[
\left(n\cat{-Gph}, \iT{n}, \iPi{n}\right) 
\ = \
\DM\left((n - 1)\cat{-Gph}, \iT{n - 1}, \iPi{n - 1}\right).
\]
\end{itemize}
(We appear to be re-using the notation $\iPi{n}$ once again; this will be
justified shortly.)  We also define, for each $n \geq 1$, a morphism
\[
(U_n, \gamma_n)\: 
\left(n\cat{-Gph}, \iT{n}, \iPi{n}\right) 
\lra 
\left((n - 1)\cat{-Gph}, \iT{n - 1}, \iPi{n - 1}\right)
\]
in $\tdalgd$ as follows.
\begin{itemize}
\item $(U_1, \gamma_1)$ is the morphism of monads corresponding to the
commutative diagram
\[
\psset{unit=0.1cm,labelsep=2pt,nodesep=2pt}
\pspicture(0,-5)(30,40)

%   a0
% a1 a2
% a3 a4

%%%%%%%%%% top
\rput(15,35){\rnode{a0}{$\Top$}}  % top 

\rput(-25.5,17.5){$\cat{$(\Set, \iPi{0} E)$-Cat} \iso $}
\rput(44,17.5){$\iso \Set$}

\rput(0,18){\rnode{a1}{$\cat{$1$-Gph}^{\iT{1}}$}}  % top left
\rput(30,18){\rnode{a2}{$\cat{$0$-Gph}^{\iT{0}}$}}  % top right
\rput(0,0){\rnode{a3}{\cat{$1$-Gph}}}  % bottom left
\rput(30,0){\rnode{a4}{\cat{$0$-Gph;}}}  % bottom right

\ncline{->}{a0}{a1} \nbput[labelsep=-1pt,npos=0.45]{{\scriptsize $$}} % left diag
\ncline{->}{a0}{a2} \naput[labelsep=2pt,npos=0.53]{{\scriptsize $$}} % right diag

\ncline{->}{a1}{a2} \naput{{\scriptsize $\cat{ob}$}} % top
\ncline{->}{a3}{a4} \nbput{{\scriptsize $\cat{ob}$}} % bottom
\ncline{->}{a1}{a3} \nbput{{\scriptsize $$}} % left
\ncline{->}{a2}{a4} \naput{{\scriptsize $$}} \nbput{{\scriptsize $\iso$}} % right

\endpspicture
\]
note that this is in $\tdalgd$, so we may apply $\DM$.

\item For $n \geq 2$, $(U_n, \gamma_n) = \DM(U_{n - 1}, \gamma_{n - 1})$.
\end{itemize}

\end{definition}

Note that Lemma~\ref{eightpointfour} together with induction tells us that 
\begin{itemize}
\item the category of algebras for each $\iT{n}$ is \cat{$n$-iCat}, and 
\item each functor 
\[
\iPi{n}\: \cat{Top} \lra n\cat{-Gph}^{\iT{n}} \iso n\cat{-iCat}
\]
appearing in the sequence of objects $(n\cat{-Gph}, \iT{n}, \iPi{n})$ is the
same functor $\iPi{n}$ as defined in Section~\ref{sec:trimcat}.
\end{itemize}
Furthermore when morphisms of monads are
viewed as commutative squares, the diagram
\begin{equation}\label{diagfive}
\cdots
\map{(U_2, \gamma_2)} (1\cat{-Gph}, \iT{1}, \iPi{1})
\map{(U_1, \gamma_1)} (0\cat{-Gph}, \iT{0}, \iPi{0})
\end{equation}
becomes the commutative diagram
\begin{equation}\label{newdiagtwenty}
\psset{unit=0.1cm,labelsep=2pt,nodesep=4pt}
\pspicture(0,-4)(100,55)

% x1 x2 x3
% y1 y2 y3

\rput(0,45){$\cdots$}
\rput(0,10){$\cdots$}

\rput(0,45){\rnode{x0}{\white{$\DM^3\1$}}}

\rput(30,45){\rnode{x1}{
$\left(\pspicture(-7,7.5)(6,18)
% a1 
% a2 
\rput(0,15){\rnode{a1}{$\Top$}}  %  top
\rput(0,2){\rnode{a2}{$\cat{$2$-iCat}$}}  % bottom
\ncline{->}{a1}{a2} \naput{{\scriptsize $\iPi{2}$}} % top

\endpspicture\right)$}}

\rput(65,45){\rnode{x2}{
$\left(\pspicture(-7,7.5)(6,18)
% a1 
% a2 
\rput(0,15){\rnode{a1}{$\Top$}}  %  top
\rput(0,2){\rnode{a2}{$\cat{$1$-iCat}$}}  % bottom
\ncline{->}{a1}{a2} \naput{{\scriptsize $\iPi{1}$}} % top

\endpspicture\right)$}}

\rput(100,45){\rnode{x3}{
$\left(\pspicture(-7,7.5)(6,18)
% a1 
% a2 
\rput(0,15){\rnode{a1}{$\Top$}}  %  top
\rput(0,2){\rnode{a2}{$\cat{$0$-iCat}$}}  % bottom
\ncline{->}{a1}{a2} \naput{{\scriptsize $\iPi{0}$}} % top

\endpspicture\right)$}}

\psset{nodesep=4pt}
\ncline{->}{x0}{x1} \naput{{\scriptsize $U_3$}}
\ncline{->}{x1}{x2} \naput{{\scriptsize $U_2$}}
\ncline{->}{x2}{x3} \naput{{\scriptsize $U_1$}}

%%%%%%%%%%%%%%%%%%%%%%%%%%%%%%%%%%%%%%%%%%%%%%%%%%%%%%%%bottom row

\rput(3,10){\rnode{y0}{\white{$\DM^3\1$}}}

\rput(30,10){\rnode{y1}{
$\left(\pspicture(-7,7.5)(6,18)
% a1 
% a2 
\rput(0,15){\rnode{a1}{$\Top$}}  %  top
\rput(0,2){\rnode{a2}{$\cat{$2$-Gph}$}}  % bottom
\ncline{->}{a1}{a2} \naput{{\scriptsize $\iPi{2}$}} % top

\endpspicture\right)$}}

\rput(65,10){\rnode{y2}{
$\left(\pspicture(-7,7.5)(6,18)
% a1 
% a2 
\rput(0,15){\rnode{a1}{$\Top$}}  %  top
\rput(0,2){\rnode{a2}{$\cat{$1$-Gph}$}}  % bottom
\ncline{->}{a1}{a2} \naput{{\scriptsize $\iPi{1}$}} % top

\endpspicture\right)$}}

\rput(100,10){\rnode{y3}{
$\left(\pspicture(-7,7.5)(6,18)
% a1 
% a2 
\rput(0,15){\rnode{a1}{$\Top$}}  %  top
\rput(0,2){\rnode{a2}{$\cat{$0$-Gph}$}}  % bottom
\ncline{->}{a1}{a2} \naput{{\scriptsize $\iPi{0}$}} % top

\endpspicture\right)$.}}

\psset{nodesep=2pt}
\ncline{->}{y0}{y1} \naput{{\scriptsize $U_3$}}
\ncline{->}{y1}{y2} \naput{{\scriptsize $U_2$}}
\ncline{->}{y2}{y3} \naput{{\scriptsize $U_1$}}

\psset{nodesepA=5pt,nodesepB=3pt}
\ncline{->}{x1}{y1} \naput{{\scriptsize $$}} 
\ncline{->}{x2}{y2} \naput{{\scriptsize $$}} 
\ncline{->}{x3}{y3} \naput{{\scriptsize $$}}

\endpspicture
\end{equation}

\begin{lemma} %8.5

Diagram~(\ref{newdiagtwenty}) induces a forgetful morphism
\[
\left(\pspicture(-6.5,7.5)(6.5,18)

% a1 
% a2 
\rput(0,15){\rnode{a1}{$\Top$}}  %  top
\rput(0,2){\rnode{a2}{$\cat{$\omg$-Cat}$}}  % bottom

\ncline[nodesep=2pt]{->}{a1}{a2} \naput{{\scriptsize $\Pi_\omg$}} % top

\endpspicture\right)
\lra
\left(\pspicture(-6.9,7.5)(6.9,18)

% a1 
% a2 
\rput(0,15){\rnode{a1}{$\Top$}}  %  top
\rput(0,2){\rnode{a2}{$\cat{$\omg$-Gph}$}}  % bottom

\ncline[nodesep=2pt]{->}{a1}{a2} \naput{{\scriptsize $\Pi_\omg$}} % top

\endpspicture\right)
\]
with underlying functor
\[\cat{$\omg$-Cat} \lra \cat{$\omg$-Gph}.\]
This functor is monadic.

\end{lemma}

\begin{proof}
This follows from the fact that each forgetful functor $n\cat{-iCat} \lra
n\cat{-Gph}$ is monadic, using the same argument as in the proof of
Theorem~F.2.2 of~\cite{lei8}.
\end{proof}

\begin{definition}

Write $\Tom$ for the induced monad on $\omg\cat{-Gph}$.  Thus, $\Tom$ is the
monad for Trimble $\omg$-categories, and we have the functor
\[
\Piom\: \cat{Top} \lra \omg\cat{-Gph}^{\Tom}.
\]

\end{definition}

We will show later that $\Tom$ preserves coproducts and that $\Piom$ preserves
both coproducts and finite products.

For each $n \geq 0$, the truncation morphisms participate in a commutative
square

\[
\psset{unit=0.1cm,labelsep=2pt,nodesep=2pt}
\pspicture(50,55)

% x1 x2 
% y1 y2 

\rput(0,45){\rnode{x1}{
$\left(\pspicture(-6,7.5)(6,18)
% a1 
% a2 
\rput(0,15){\rnode{a1}{$\Top$}}  %  top
\rput(0,2){\rnode{a2}{$\cat{$\omg$-Cat}$}}  % bottom
\ncline{->}{a1}{a2} \naput{{\scriptsize $\Pi_\omg$}} % top

\endpspicture\right)$}}

\rput(40,45){\rnode{x2}{
$\left(\pspicture(-6,7.5)(6,18)
% a1 
% a2 
\rput(0,15){\rnode{a1}{$\Top$}}  %  top
\rput(0,2){\rnode{a2}{$\cat{$n$-iCat}$}}  % bottom
\ncline{->}{a1}{a2} \naput{{\scriptsize $\iPi{n}$}} % top

\endpspicture\right)$}}

\rput(0,10){\rnode{y1}{
$\left(\pspicture(-7,7.5)(7,18)
% a1 
% a2 
\rput(0,15){\rnode{a1}{$\Top$}}  %  top
\rput(0,2){\rnode{a2}{$\cat{$\omg$-Gph}$}}  % bottom
\ncline{->}{a1}{a2} \naput{{\scriptsize $\Pi_\omg$}} % top

\endpspicture\right)$}}

\rput(40,10){\rnode{y2}{
$\left(\pspicture(-6,7.5)(6,18)
% a1 
% a2 
\rput(0,15){\rnode{a1}{$\Top$}}  %  top
\rput(0,2){\rnode{a2}{$\cat{$n$-Gph}$}}  % bottom
\ncline{->}{a1}{a2} \naput{{\scriptsize $\iPi{n}$}} % top

\endpspicture\right)$}}

\psset{nodesep=2pt}
\ncline{->}{x1}{x2} \naput{{\scriptsize $\cat{tr}_n$}}
\ncline{->}{y1}{y2} \nbput{{\scriptsize $\cat{tr}_n$}}

\psset{nodesepA=2pt,nodesepB=2pt}
\ncline{->}{x1}{y1}
\ncline{->}{x2}{y2}

\endpspicture
\]
This gives graph truncation the structure of a lax morphism of monads, say
\[
(\tr_n, \tau_n)\: (\omg\cat{-Gph}, \Tom) \lra (n\cat{-Gph}, \iT{n}),
\]
commuting with the morphisms from $(\cat{Top}, 1)$.

We now embark on the proof of one of our two main theorems, characterizing
the monad $\Tom$ for Trimble $\omg$-categories as the terminal coalgebra
for $\DM$.

\begin{theorem} %8.8
\label{thm:trim-monad-tc}

The object $(\omg\cat{-Gph}, \Tom, \Piom)$ of $\cat{Top} \downarr \Algd$,
consisting of the monad $\Tom$ for Trimble $\omg$-categories together with
the fundamental $\omg$-groupoid functor $\Piom$, is the terminal coalgebra
for $\DM$.

\end{theorem}

\begin{proof}

We use Ad\'amek's theorem.  We need to consider the limit in $\tdalgd$ of
\begin{equation}\label{diagsix}
\psset{unit=0.1cm,labelsep=2pt,nodesep=2pt}
\pspicture(0,-1)(80,5)

% x1 x2 x3

\rput(0,0){$\cdots$}

\rput(0,0){\rnode{x0}{\white{$\DG^3\1$}}}
\rput(30,0){\rnode{x1}{$\DM^3\1$}}
\rput(55,0){\rnode{x2}{$\DM^2\1$}}
\rput(80,0){\rnode{x3}{$\DM\1$}}

\psset{nodesep=4pt}
\ncline{->}{x0}{x1} \naput{{\scriptsize $\DG^3!$}}
\ncline{->}{x1}{x2} \naput{{\scriptsize $\DG^2!$}}
\ncline{->}{x2}{x3} \naput{{\scriptsize $\DG!$}}

\endpspicture
\end{equation}

\vv{1em}

\noi where $\1$ is the terminal object $(\1, 1, !)$ of $\tdalgd$. (We also continue to write the terminal object of $\CAT$ as $\1$.)  We need to show
\begin{enumerate}
 \item the limit is $(\cat{$\omg$-Gph}, \Tom, \Piom)$, and
\item the limit is preserved by $\DM$.
\end{enumerate}

First we note that the above diagram (\ref{diagsix}) is isomorphic to our diagram of $n$-dimensional monads and truncations; formally, there are canonical isomorphisms $(I_n, \iota_n)_{n \geq 0}$ in $\tdalgd$ making the following diagram commute
\[
\psset{unit=0.1cm,labelsep=2pt,nodesep=2pt}
\pspicture(100,45)

% x1 x2 x3
% y1 y2 y3

\rput(0,40){$\cdots$}
\rput(0,10){$\cdots$}

\rput(0,40){\rnode{x0}{\white{$\DM^3\1$}}}
\rput(30,40){\rnode{x1}{$\DM^3\1$}}
\rput(65,40){\rnode{x2}{$\DM^2\1$}}
\rput(100,40){\rnode{x3}{$\DM\1$}}

\psset{nodesep=4pt}
\ncline{->}{x0}{x1} \naput{{\scriptsize $\DM^3!$}}
\ncline{->}{x1}{x2} \naput{{\scriptsize $\DM^2!$}}
\ncline{->}{x2}{x3} \naput{{\scriptsize $\DM!$}}

\rput(3,10){\rnode{y0}{\white{$\DM^3\1$}}}

\rput(30,10){\rnode{y1}{
$\left(\pspicture(-8,7.5)(7,18)
% a1 
% a2 
\rput(0,15){\rnode{a1}{$\Top$}}  %  top
\rput(0,2){\rnode{a2}{$\cat{$2$-Gph}^{\iT{2}}$}}  % bottom
\ncline{->}{a1}{a2} \naput{{\scriptsize $\iPi{2}$}} % top

\endpspicture\right)$}}

\rput(65,10){\rnode{y2}{
$\left(\pspicture(-8,7.5)(7,18)
% a1 
% a2 
\rput(0,15){\rnode{a1}{$\Top$}}  %  top
\rput(0,2){\rnode{a2}{$\cat{$1$-Gph}^{\iT{1}}$}}  % bottom
\ncline{->}{a1}{a2} \naput{{\scriptsize $\iPi{1}$}} % top

\endpspicture\right)$}}

\rput(100,10){\rnode{y3}{
$\left(\pspicture(-8,7.5)(7,18)
% a1 
% a2 
\rput(0,15){\rnode{a1}{$\Top$}}  %  top
\rput(0,2){\rnode{a2}{$\cat{$0$-Gph}^{\iT{0}}$}}  % bottom
\ncline{->}{a1}{a2} \naput{{\scriptsize $\iPi{0}$}} % top

\endpspicture\right)$}}

\psset{nodesep=2pt}
\ncline{->}{y0}{y1} \naput{{\scriptsize $U_3$}}
\ncline{->}{y1}{y2} \naput{{\scriptsize $U_2$}}
\ncline{->}{y2}{y3} \naput{{\scriptsize $U_1$}}

\psset{nodesepA=5pt,nodesepB=3pt}
\ncline{->}{x1}{y1} \naput{{\scriptsize $(I_2,\iota_2)$}} \nbput{{\scriptsize $\iso$}}
\ncline{->}{x2}{y2} \naput{{\scriptsize $(I_1,\iota_1)$}} \nbput{{\scriptsize $\iso$}}
\ncline{->}{x3}{y3} \naput{{\scriptsize $(I_0,\iota_0)$}} \nbput{{\scriptsize $\iso$}}

\endpspicture
\]
with the underlying functors $I_n$ as in the proof of Theorem~\ref{thm:gph-tc}, as usual.  Also as usual, this is true by a straightforward induction: we have
\[
\DM\1 = (\cat{$\1$-Gph}, 1, !) \iso (\cat{$0$-Gph},1, !) 
\]
and as the rightmost square of the diagram commutes, the whole diagram commutes. 

This shows that the the limit we need to take is over the following diagram in $\tdalgd$
\begin{equation}\label{diagseven} 
\cdots 
\map{(U_3, \gamma_3)} (\cat{$2$-Gph}, \iT{2}, \iPi{2}) 
\map{(U_2, \gamma_2)} (\cat{$1$-Gph}, \iT{1}, \iPi{1}) 
\map{(U_1, \gamma_0)} (\cat{$0$-Gph}, \iT{0}, \iPi{0})
\end{equation}
so our aim is to show that the following cone is a limit in $\tdalgd$.
\begin{equation}\label{diageight}
\psset{unit=0.095cm,labelsep=2pt,nodesep=2pt}
\pspicture(0,-7)(114,24)
% standard limit diagram
%   a
% x3 x2 x1 x0

\rput(20,20){\rnode{a}{$(\cat{$\omg$-Gph}, \Tom, \Piom)$}}

\rput(0,0){$\cdots$}

\rput(0,0){\rnode{x3}{\white{$\DG^3\1$}}}
\rput(30,0){\rnode{x2}{$(\cat{$2$-Gph}, \iT{2}, \iPi{2})$}}
\rput(70,0){\rnode{x1}{$(\cat{$1$-Gph}, \iT{1}, \iPi{1})$}}
\rput(108,0){\rnode{x0}{$(\cat{$0$-Gph}, \iT{0}, \iPi{0})$}}

\psset{nodesep=2pt}
\ncline{->}{x3}{x2} \naput{{\scriptsize $$}}
\ncline{->}{x2}{x1} \nbput{{\scriptsize $$}}
\ncline{->}{x1}{x0} \nbput{{\scriptsize $$}}

\ncline{->}{a}{x0} \naput[npos=0.65,labelsep=1pt]{{\scriptsize $(\mbox{tr}_0, \tau_0)$}}
\ncline{->}{a}{x1} \nbput[npos=0.55,labelsep=1pt]{{\scriptsize $(\mbox{tr}_1, \tau_1)$}}
\ncline{->}{a}{x2} \nbput[npos=0.65,labelsep=1pt]{{\scriptsize $(\mbox{tr}_2, \tau_2)$}}
\ncline{->}{a}{x3} \naput{{\scriptsize $$}}

\endpspicture
\end{equation}
So far we only know it is a cone in $\tdalg$.  We make use of the isomorphism
\[\tdalgd \iso (\Top,1)/\MNDd.\]
Now the forgetful functor
\[(\Top,1)/\MNDd \tra \MNDd\]
creates limits (Proposition \ref{onepointten}) so it is enough for us to show
\numAlph
\begin{enumerate}
 \item the cone above is in fact in $(\Top,1)/\MNDd$, not just $(\Top,1)/\MND$, and
\item the underlying cone is a limit in $\MNDd$.
\end{enumerate}
The only complicated aspect of this is checking the correct distributivity/(co)limit preserving conditions.  We proceed in steps.
\numarabic
\begin{enumerate}
 \item First note that by definition the diagram (\ref{diagseven}) whose limit we are taking is in $\tdalgd$, so in particular the underlying diagram
\[ 
\cdots 
\map{(U_3, \gamma_3)} (\cat{$2$-Gph}, \iT{2}) 
\map{(U_2, \gamma_2)} (\cat{$1$-Gph}, \iT{1}) 
\map{(U_1, \gamma_1)} (\cat{$0$-Gph}, \iT{0})
\]
is in $\MNDd$.

\item Also note that the underlying cone of diagram (\ref{diageight}) is in
  $\MNDwk$, that is, all the truncation maps are weak.  This can be seen in
  the same way as the analogous result for the strict case (see proof of
  Theorem~\ref{fourpointsixteen})---the map $!$ is certainly weak, and
  $\DM$ is seen to preserve weakness just as for $\FM$. 

\item Morever the underlying cone  in $\CAT$ of diagram (\ref{diageight}) is
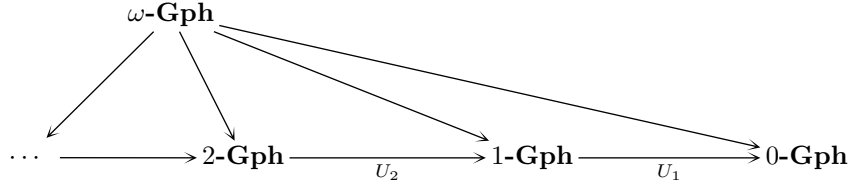
\[
\psset{unit=0.095cm,labelsep=2pt,nodesep=2pt}
\pspicture(0,-4)(114,22)
% standard limit diagram
%   a
% x3 x2 x1 x0

\rput(20,20){\rnode{a}{\cat{$\omg$-Gph}}}

\rput(0,0){$\cdots$}

\rput(0,0){\rnode{x3}{\white{$\DG^3\1$}}}
\rput(30,0){\rnode{x2}{$\cat{$2$-Gph}$}}
\rput(70,0){\rnode{x1}{$\cat{$1$-Gph}$}}
\rput(108,0){\rnode{x0}{$\cat{$0$-Gph}$}}

\psset{nodesep=2pt}
\ncline{->}{x3}{x2} \naput{{\scriptsize $$}}
\ncline{->}{x2}{x1} \nbput{{\scriptsize ${U_2}$}}
\ncline{->}{x1}{x0} \nbput{{\scriptsize ${U_1}$}}

\ncline{->}{a}{x0} \naput{{\scriptsize $$}}
\ncline{->}{a}{x1} \naput{{\scriptsize $$}}
\ncline{->}{a}{x2} \naput{{\scriptsize $$}}
\ncline{->}{a}{x3} \naput{{\scriptsize $$}}

\endpspicture
\]
so we know it is a limit in $\CAT$ and at least a cone in $\CATd$.

\item So we can apply Proposition~\ref{twopointseven} to conclude that the vertex $(\cat{$\omg$-Gph}, \Tom)$ is in $\MNDd$, and that the cone 
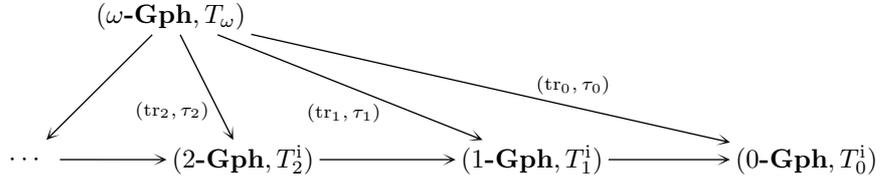
\[
\psset{unit=0.095cm,labelsep=2pt,nodesep=2pt}
\pspicture(0,-5)(114,22)
% standard limit diagram
%   a
% x3 x2 x1 x0

\rput(20,20){\rnode{a}{$(\cat{$\omg$-Gph}, \Tom)$}}

\rput(0,0){$\cdots$}

\rput(0,0){\rnode{x3}{\white{$\DG^3\1$}}}
\rput(30,0){\rnode{x2}{$(\cat{$2$-Gph}, \iT{2})$}}
\rput(70,0){\rnode{x1}{$(\cat{$1$-Gph}, \iT{1})$}}
\rput(108,0){\rnode{x0}{$(\cat{$0$-Gph}, \iT{0})$}}

\psset{nodesep=2pt}
\ncline{->}{x3}{x2} \naput{{\scriptsize $$}}
\ncline{->}{x2}{x1} \nbput{{\scriptsize $$}}
\ncline{->}{x1}{x0} \nbput{{\scriptsize $$}}

\ncline{->}{a}{x0} \naput[npos=0.65,labelsep=1pt]{{\scriptsize $(\mbox{tr}_0, \tau_0)$}}
\ncline{->}{a}{x1} \nbput[npos=0.55,labelsep=1pt]{{\scriptsize $(\mbox{tr}_1, \tau_1)$}}
\ncline{->}{a}{x2} \nbput[npos=0.65,labelsep=1pt]{{\scriptsize $(\mbox{tr}_2, \tau_2)$}}
\ncline{->}{a}{x3} \naput{{\scriptsize $$}}

\endpspicture
\]
is a limit in both $\MNDd$ and $\MND$.  So we have shown (B).

\end{enumerate}

We now prove (A).  We already know that the underlying cone is in $\MNDd$ and that the sequential diagram is in $(\Top,1)/\MNDd$, so we just need to show that the vertex $(\cat{$\omg$-Gph}, \Tom, \Piom)$ is in $(\Top,1)/\MNDd$, and the only missing fact is that the functor 
\[\Top \map{\Piom} \cat{$\omg$-Cat}\] is in $\CATd$ not just $\CAT$.  We do
this by showing it is induced by the universal property of
$\cat{$\omg$-Cat}$ in $\CATd$. 

Recall that $\Alg$ and $\Algd$ each have a right adjoint and so preserve limits (Propositions~\ref{twopointthree} and \ref{twopointsix}); thus the diagram of categories of algebras
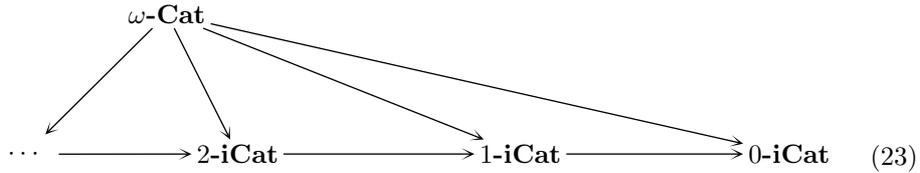
\begin{equation}\label{diagten}
\psset{unit=0.093cm,labelsep=2pt,nodesep=3pt}
\pspicture(0,3)(122,24)

%   a
% x0 x1 x2 x3 x4

\rput(20,20){\rnode{a}{\cat{$\omg$-Cat}}}

\rput(0,0){$\cdots$}

\rput(0,0){\rnode{x3}{\white{$\DG^3\1$}}}
\rput(30,0){\rnode{x2}{$\cat{$2$-iCat}$}}
\rput(70,0){\rnode{x1}{$\cat{$1$-iCat}$}}
\rput(108,0){\rnode{x0}{$\cat{$0$-iCat}$}}

\psset{nodesep=2pt}
\ncline{->}{x3}{x2} \naput{{\scriptsize $$}}
\ncline{->}{x2}{x1} \naput{{\scriptsize $$}}
\ncline{->}{x1}{x0} \naput{{\scriptsize $$}}

\ncline{->}{a}{x0} \naput{{\scriptsize $$}}
\ncline{->}{a}{x1} \naput{{\scriptsize $$}}
\ncline{->}{a}{x2} \naput{{\scriptsize $$}}
\ncline{->}{a}{x3} \naput{{\scriptsize $$}}

\endpspicture
\end{equation}

\vv{1em}

\noi is a limit in both $\CATd$ and $\CAT$ (as usual we need the latter for the next proof); compare with Section~\ref{batanin}.  Furthermore we have another cone in $\CATd$ coming from the underlying sequential diagram, that is we have
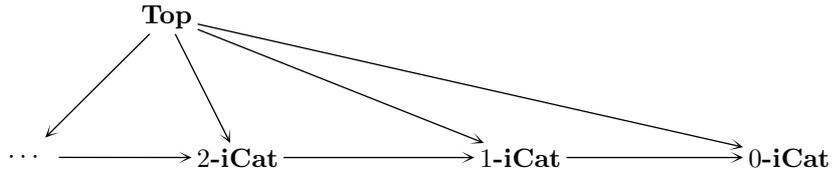
\[
\psset{unit=0.093cm,labelsep=2pt,nodesep=3pt}
\pspicture(0,-5)(122,24)

%   a
% x0 x1 x2 x3 x4

\rput(20,20){\rnode{a}{\cat{Top}}}

\rput(0,0){$\cdots$}

\rput(0,0){\rnode{x3}{\white{$\DG^3\1$}}}
\rput(30,0){\rnode{x2}{$\cat{$2$-iCat}$}}
\rput(70,0){\rnode{x1}{$\cat{$1$-iCat}$}}
\rput(108,0){\rnode{x0}{$\cat{$0$-iCat}$}}

\psset{nodesep=2pt}
\ncline{->}{x3}{x2} \naput{{\scriptsize $$}}
\ncline{->}{x2}{x1} \naput{{\scriptsize $$}}
\ncline{->}{x1}{x0} \naput{{\scriptsize $$}}

\ncline{->}{a}{x0} \naput{{\scriptsize $$}}
\ncline{->}{a}{x1} \naput{{\scriptsize $$}}
\ncline{->}{a}{x2} \naput{{\scriptsize $$}}
\ncline{->}{a}{x3} \naput{{\scriptsize $$}}

\endpspicture
\]
Now the functor $\Piom$ makes everything commute (see diagram (\ref{diagnine}), Definition~\ref{sevenpointnine}), so must be the factorisation induced by the universal property of the above limit in $\CATd$. Thus we see that $\Piom$ is also in $\CATd$ as required.

This completes the proof that (\ref{diageight}) is a limit;  it remains to prove that this limit is preserved by 
\[\DM \: (\Top,1)/\MNDd \tra (\Top,1)/\MNDd.\]
Since the forgetful functor
\[(\Top,1)/\MNDd \tra \MNDd\]
reflects limits, it is again enough to show that the underlying cone in $\MNDd$ is a limit.  We use Proposition~\ref{twopointseven} again.  We know:

\begin{enumerate}
 \item The diagram is in $\MNDwk$ as $\DM$ preserves weakness.
\item The underlying cone is in $\CATd$ as it underlies a cone in $\MNDd$.
\item The underlying cone is a limit in $\CAT$
\[
\psset{unit=0.095cm,labelsep=2pt,nodesep=2pt}
\pspicture(114,24)
% standard limit diagram
%   a
% x3 x2 x1 x0

\rput(20,20){\rnode{a}{$\cat{$\omg$-Gph} \iso \cat{($\omg$-Gph)-Gph}$}}

\rput(0,0){$\cdots$}

\rput(0,0){\rnode{x3}{\white{$\DG^3\1$}}}
\rput(30,0){\rnode{x2}{$\cat{$2$-Gph}$}}
\rput(70,0){\rnode{x1}{$\cat{$1$-Gph}$}}
\rput(108,0){\rnode{x0}{$\cat{$0$-Gph}$}}

\psset{nodesep=2pt}
\ncline{->}{x3}{x2} \naput{{\scriptsize $$}}
\ncline{->}{x2}{x1} \nbput{{\scriptsize ${U_2}$}}
\ncline{->}{x1}{x0} \nbput{{\scriptsize ${U_1}$}}

\ncline{->}{a}{x0} \naput{{\scriptsize $$}}
\ncline{->}{a}{x1} \naput{{\scriptsize $$}}
\ncline{->}{a}{x2} \naput{{\scriptsize $$}}
\ncline{->}{a}{x3} \naput{{\scriptsize $$}}

\endpspicture
\]

\end{enumerate}
so the cone in question is a limit in $\MNDd$ (and $\MND$) as required.
\end{proof}

We now deduce our other main theorem: that the category $\omg\cat{-Cat}$
of Trimble $\omg$-categories, together with the fundamental
$\omg$-groupoid functor $\Piom$, is the terminal coalgebra for $\DC$.

\begin{prfof}{Theorem~\ref{thm:trim-cat-tc}}
We use Ad\'amek's Theorem.  We have already seen (Lemma~\ref{lemma:DC-ladder}) that it suffices to show
that $\left(\cat{Top} \tmap{\Piom} \omg\cat{-Cat}\right)$ is the limit in
$\cat{Top}/\CATp$ of the diagram
\begin{equation}
\label{eq:DC-sequence}
\cdots 
\map{U_3}
\left(\pspicture(-6,7.5)(6,18)

% a1 
% a2 
\rput(0,15){\rnode{a1}{$\Top$}}  %  top
\rput(0,2){\rnode{a2}{$\cat{$2$-iCat}$}}  % bottom

\ncline[nodesep=2pt]{->}{a1}{a2} \naput{{\scriptsize $\iPi{2}$}} % top

\endpspicture\right)
\map{U_2}
\left(\pspicture(-6,7.5)(6,18)

% a1 
% a2 
\rput(0,15){\rnode{a1}{$\Top$}}  %  top
\rput(0,2){\rnode{a2}{$\cat{$1$-iCat}$}}  % bottom

\ncline[nodesep=2pt]{->}{a1}{a2} \naput{{\scriptsize $\iPi{1}$}} % top

\endpspicture\right)
\map{U_1}
\left(\pspicture(-6,7.5)(6,18)

% a1 
% a2 
\rput(0,15){\rnode{a1}{$\Top$}}  %  top
\rput(0,2){\rnode{a2}{$\cat{$0$-iCat}$}}  % bottom

\ncline[nodesep=2pt]{->}{a1}{a2} \naput{{\scriptsize $\iPi{0}$}} % top

\endpspicture\right)
\end{equation}
and that $\DC$ preserves this limit.

As the functor
\[\Top/\CATp \tra \CATp\]
reflects limits, it suffices to show that the cone
\begin{equation}\label{newdiagtwentytwo}
\psset{unit=0.093cm,labelsep=2pt,nodesep=3pt}
\pspicture(0,-5)(122,24)

%   a
% x0 x1 x2 x3 x4

\rput(20,20){\rnode{a}{\cat{$\omg$-Cat}}}

\rput(0,0){$\cdots$}

\rput(0,0){\rnode{x0}{\white{$\DG^3\1$}}}
\rput(22,0){\rnode{x1}{$\cat{$n$-iCat}$}}
\rput(60,0){\rnode{x2}{$\cat{$(n-1)$-iCat}$}}
\rput(92,0){\rnode{x3}{$\cdots$}}
\rput(115,0){\rnode{x4}{$\cat{$0$-iCat}$}}

\psset{nodesep=2pt}
\ncline{->}{x0}{x1} \naput{{\scriptsize $$}}
\ncline{->}{x1}{x2} \naput{{\scriptsize $$}}
\ncline[nodesepB=10pt]{->}{x2}{x3} \naput{{\scriptsize $$}}
\ncline[nodesepA=10pt]{->}{x3}{x4} \naput{{\scriptsize $$}}

\ncline{->}{a}{x0} \naput{{\scriptsize $$}}
\ncline{->}{a}{x1} \naput{{\scriptsize $$}}
\ncline{->}{a}{x2} \naput{{\scriptsize $$}}
\ncline{->}{a}{x4} \naput{{\scriptsize $$}}

\endpspicture
\end{equation}
is a limit in $\CATp$.  Further, as
\[\CATp \hra \CAT\]
reflects limits, it suffices to show that this is a limit in \CAT, which is true as $\Alg$ preserves limits (in fact this limit was exhibited in the previous proof).

To show that $\DC\: \cat{Top}/\CATp \lra \cat{Top}/\CATp$ preserves this limit,
we use the commutative diagram 
\[\psset{unit=0.08cm,labelsep=0pt,nodesep=3pt}
\pspicture(0,-5)(90,29)

% a1 a2 a3 down left
% b1 b2 b3

\rput(0,25){\rnode{a1}{$(\Top,1)/\MNDd$}} % top left
\rput(49,25){\rnode{a2}{$(\Top,1)/\MNDd$}} % top middle
\rput(90,25){\rnode{a3}{$\MND$}} % top right

\rput(0,0){\rnode{b1}{$\Top/\CATd$}}   % bottom left
\rput(49,0){\rnode{b2}{$\Top/\CATd$}}  % bottom middle
\rput(90,0){\rnode{b3}{$\CAT$}}  % bottom right

\psset{nodesep=3pt,labelsep=2pt,arrows=->}
\ncline{a1}{a2}\naput{{\scriptsize $\DM$}} % top
\ncline{a2}{a3}\naput{{\scriptsize $U$}} % top

\ncline{b1}{b2}\nbput{{\scriptsize $\DC$}} % bottom
\ncline{b2}{b3}\nbput{{\scriptsize $U$}} % bottom

\ncline{a1}{b1}\nbput{{\scriptsize $\Alg$}} % left
\ncline{a2}{b2}\nbput{{\scriptsize $\Alg$}} % right
\ncline{a3}{b3}\naput{{\scriptsize $\Alg$}} % right

\endpspicture\]
(Lemma~\ref{lemma:comparing-Ds}).  The image of the limit
cone~(\ref{newdiagtwentytwo}) under $\DC$ is a cone in $\cat{Top}/\CATd$, which we
wish to show is a limit cone in $\cat{Top}/\CATp$.  By the usual Propositions~\ref{prop:coslice-creates} and~\ref{prop:little-refl}, it is
enough to show that \emph{its} underlying cone in $\CAT$ is a limit.  This
cone is 

\[
 U\DC\left(\pspicture(-27,6.5)(28,19)
\rput(0,7){
$\left(\pspicture(-6,7.5)(6,18)
% a1 
% a2 
\rput(0,15){\rnode{a1}{$\Top$}}  %  top
\rput(0,3){\rnode{a2}{$\cat{$\omg$-Cat}$}}  % bottom
\ncline[nodesepB=2pt]{->}{a1}{a2} \naput{{\scriptsize $\Pi_\omg$}} % top

\endpspicture\right)
%%%%%%%%%%%%%%%%%%%%%%%%%%%%%%%%%%%%%%%%%%%
\map{\tr_n}
%%%%%%%%%%%%%%%%%%%%%%%%%%%%%%%%%%%%%%%%%%5
\left(\pspicture(-6,7.5)(6,18)
% a1 
% a2 
\rput(0,15){\rnode{a1}{$\Top$}}  %  top
\rput(0,3){\rnode{a2}{$\cat{$n$-iCat}$}}  % bottom
\ncline[nodesepB=2pt]{->}{a1}{a2} \naput{{\scriptsize $\iPi{n}$}} % top

\endpspicture\right)$}

\endpspicture
\right)_{n \geq 0}
\]

\[ \iso U \DC \Alg \left(\pspicture(-30,6.5)(31,19)
\rput(0,7){
$\left(\pspicture(-9,7.5)(8.5,18)
% a1 
% a2 
\rput(0,15){\rnode{a1}{$\Top$}}  %  top
\rput(0,3){\rnode{a2}{$\cat{$\omg$-Gph}^{T_\omg}$}}  % bottom
\ncline[nodesepB=0pt]{->}{a1}{a2} \naput{{\scriptsize $\Pi_\omg$}} % top

\endpspicture\right)
%%%%%%%%%%%%%%%%%%%%%%%%%%%%%%%%%%%%%%%%%%%
\map{(\tr_n, \tau_n)}
%%%%%%%%%%%%%%%%%%%%%%%%%%%%%%%%%%%%%%%%%%5
\left(\pspicture(-7.5,7.5)(7,18)
% a1 
% a2 
\rput(0,15){\rnode{a1}{$\Top$}}  %  top
\rput(0,3){\rnode{a2}{$\cat{$n$-Gph}^{\iT{n}}$}}  % bottom
\ncline[nodesepB=-1pt]{->}{a1}{a2} \naput{{\scriptsize $\iPi{n}$}} % top

\endpspicture\right)$}

\endpspicture
\right)_{n \geq 0}
\]
But $U\DC\Alg \iso \Alg U \DM$, and in the proof of
Theorem~\ref{thm:trim-monad-tc}, we showed that
\[
 U\DM\left(\pspicture(-27,6.5)(28,19)
\rput(0,7){
$\left(\pspicture(-6,7.5)(6,18)
% a1 
% a2 
\rput(0,15){\rnode{a1}{$\Top$}}  %  top
\rput(0,3){\rnode{a2}{$\cat{$\omg$-Cat}$}}  % bottom
\ncline[nodesepB=2pt]{->}{a1}{a2} \naput{{\scriptsize $\Pi_\omg$}} % top

\endpspicture\right)
%%%%%%%%%%%%%%%%%%%%%%%%%%%%%%%%%%%%%%%%%%%
\map{\tr_n}
%%%%%%%%%%%%%%%%%%%%%%%%%%%%%%%%%%%%%%%%%%5
\left(\pspicture(-6,7.5)(6,18)
% a1 
% a2 
\rput(0,15){\rnode{a1}{$\Top$}}  %  top
\rput(0,3){\rnode{a2}{$\cat{$n$-iCat}$}}  % bottom
\ncline[nodesepB=2pt]{->}{a1}{a2} \naput{{\scriptsize $\iPi{n}$}} % top

\endpspicture\right)$}

\endpspicture
\right)_{n \geq 0}
\]
is a limit cone in $\MND$.  Since $\Alg\: \MND \lra \CAT$ preserves limits
(Proposition~\ref{prop:FTM-adjns}), the result follows.
\end{prfof}

% \bigskip

Finally we characterise $\Tom$, the monad for Trimble weak $\omg$-categories, directly.  First we define a series of monads $P_k$ on $\wGph$ inductively as follows; the idea is that $P_k$ is the monad for ``composition along bounding $k$-cells''.

\begin{itemize}
 \item $P_0 = \fc{(\wGph, \Piom E)}$
\item for $k \geq 1$, $P_k = {(P_{k-1})}_* = {\big(\fc{(\wGph, \Piom E)}\big)}_{k*}$
\end{itemize}
where the subscript $k*$ is shorthand for $*$ applied $k$ times.

Now, we cannot define $\Tom$ globally using these monads, as it would need to be an ``infinite composite''.  However, at each dimension $n$, the action of $\Tom$ on $n$-cells is given by a finite composite of the $P_k$.

\begin{theorem}
 $\Tom$ is given as follows.  For an $\omg$-graph $A$ we have
\begin{itemize}
 \item $(PA)_0 = A_0$, and
\item for $n \geq 1$, $(\Tom A)_n = (P_0 P_1 \cdots P_{n-1} A)_n$.
\end{itemize}
\end{theorem}

\begin{proof}
 By the limit cone (\ref{diageight}) we know
\[(\Tom A)_n = (\iT{n} \tr_n A)_n\]
and by the definition of $\iT{n}$ (Definition~\ref{neweightpointten}) we know
\begin{itemize}
 \item $\iT{0} = \id$
\item $\iT{n} = (\iT{n-1})^+ = \fc{(\cat{$(n-1)$-Gph}, \iPi{n-1} E)} \circ {(\iT{n-1})}_*$.
\end{itemize}
Unravelling this we get
\[\iT{n} = P^{(n)}_0 P^{(n)}_1 \cdots P^{(n)}_{n-1}\]
where $P^{(n)}_k = (\fc{(\cat{$(n-k+1)$-Gph}, \iPi{n-k+1} E)})_{k*}$.

Thus we need to show that 
\[(P_0 P_1 \cdots P_{n-1} A)_n = (P^{(n)}_0 P^{(n)}_1 \cdots P^{(n)}_{n-1} \tr_n(A))_n\]
Thus it suffices to show that for an $n$-graph $B$ and any $k < n$
\[(P_k B)_n = {(P^{(n)}_k B)}_n\]
that is, on $n$-cells
\[{\big(\fc{(\wGph, \Piom E)}\big)}_{k*} = \big(\fc{(\cat{$(n-k+1)$-Gph}, \iPi{n-k+1} E)}\big)_{k*}\]
i.e.\ on $(n-k)$-cells
\[{\fc{(\wGph, \Piom E)}} = \big(\fc{(\cat{$(n-k+1)$-Gph}, \iPi{n-k+1} E}).\]
This follows from Remark~\ref{sixprecise}, which tells us that 
\[\tr_r(\Piom E) = \iPi{r} E.\]
\end{proof}

Finally note that in the language of \cite{che16} we now have an iterative theory of (incoherent)
$n$-categories parametrised by the series of operads 
\[\iPi{0} E, \ \iPi{1} E, \ \iPi{2} E, \ldots .\]
This emphasises the way in which the single operad $E$ is used to parametrise every type of composition in Trimble's theory of higher-dimensional categories.

%%%%%%%%%%%%%%%%%%%%%%%%%%%%%%%%%%%%%%%%%%%%%%%%%%%%%%%%%%%%%%%%%%%%%%%%%
%%%%%%%%%%%%%%%%%%%%%%%%%%%%%%%%%%%%%%%%%%%%%%%%%%%%%%%%%%%%%%%%%%%%%%%%%
% Appendix: Limits of monads
%%%%%%%%%%%%%%%%%%%%%%%%%%%%%%%%%%%%%%%%%%%%%%%%%%%%%%%%%%%%%%%%%%%%%%%%%
%%%%%%%%%%%%%%%%%%%%%%%%%%%%%%%%%%%%%%%%%%%%%%%%%%%%%%%%%%%%%%%%%%%%%%%%%

\appendix

\section{Appendix: Limits of monads}
\label{sec:app}

In this appendix we gather together some technical results about  monads.  
Our main purpose is to prove that the limits we need to calculate in
$\MNDd$ behave in the way one might expect.  
We begin by proving Proposition~\ref{twopointsix}, that the chain of adjunctions in Proposition~\ref{prop:FTM-adjns}
restricts to a chain of adjunctions
\[
\psset{unit=0.11cm,labelsep=2pt,nodesep=3pt}
\pspicture(0,22)

% a2
% a1

\rput(0,0){\rnode{a1}{$\CATd$}}  % named node, with something placed there
\rput(0,20){\rnode{a2}{$\MNDd$}}  % named node, with something placed there

\ncline[nodesepA=3pt,nodesepB=3pt]{->}{a1}{a2}\ncput*[labelsep=0pt]{{\scriptsize $\cat{Inc}$}} % right

\nccurve[angleA=215,angleB=145,ncurv=0.9]{->}{a2}{a1}\nbput{{\scriptsize $\Und$}} \naput[labelsep=5.5pt]{{\scriptsize $\ladj$}}
\nccurve[angleA=-35,angleB=35,ncurv=0.9]{->}{a2}{a1}\naput{{\scriptsize $\Alg$}} \nbput[labelsep=5.5pt]{{\scriptsize $\ladj$}}

\endpspicture
\]

\begin{prfof}{Proposition~\ref{twopointsix}} 
We have to show that all three functors restrict to the subcategories
concerned, and that given an object in one of the subcategories, the
associated unit and counit morphisms also lie in the subcategory.  The only
nontrivial cases are that $\Alg$ restricts to a functor $\MNDd \lra \CATd$ and
that, for $(\cl{V}, T) \in \MNDd$, the counit morphism
\[
\cat{Inc}\Alg(\cl{V}, T) = (\cl{V}^T, \cat{id})
\ltra
(\cl{V}, T)
\]
lies in $\MNDd$.

To prove these statements, we use the fact that for any $(\cl{V}, T) \in
\MND$, the forgetful functor $G^T\: \cl{V}^T \lra \cl{V}$ creates limits, and
creates all colimits that $T$ preserves.  For $(\cl{V}, T) \in \MNDd$, then,
$G^T$ creates finite products and small coproducts.  It follows that
$\cl{V}^T$ has all finite products and small coproducts, and since $G^T$
reflects isomorphisms, $\cl{V}^T$ is infinitely distributive.  It also follows
that $G^T$ preserves and reflects finite products and small coproducts, from
which it can be shown that the image under $\Alg$ of a morphism in $\MNDd$ is
in $\CATd$.  Finally, the underlying functor of the counit morphism above is
$G^T$, which preserves finite products and small coproducts.
\end{prfof} 

We now build up some technical lemmas.

\begin{lemma}
\label{lemma:lim-2-lim}

Let $\scat{I}$ be a category and $R\: \scat{I} \lra \CAT$ a functor.  Then any
limit cone on the diagram $R$ is also a 2-limit cone.

\end{lemma}
(We use ``2-limit'' in its traditional, strict, sense.)

\begin{proof}
See Kelly \cite{kel6}.  An elementary direct
proof is also straightforward.
\end{proof}

Given a class $\cl{K}$ of colimits and a class $\cl{L}$ of limits, denote by
$\CAT_{\cl{K}, \cl{L}}$ the subcategory of $\CAT$ consisting of the categories
with, and functors preserving, $\cl{K}$-colimits and $\cl{L}$-limits.  The
following lemma is easily verified and well known.

\begin{lemma}
\label{lemma:KL-refl}

The inclusion $\CAT_{\cl{K}, \cl{L}} \hra \CAT$ reflects limits. 
\proofbox

\end{lemma}

The lemma can also be explained in terms of 2-monad theory.  For a 2-monad
$S$ on a 2-category $\cl{C}$, let $\Alg^\mathrm{wk}(S)$ be the category of
$S$-algebras and their weak (pseudo) morphisms.  It can be shown that the
forgetful functor $\Alg^\mathrm{wk}(S) \lra \cl{C}$ reflects 2-limits.
(Since the present explanation is not essential for our purpose, we omit
the proof.)  For a suitably-chosen 2-monad $S$ on $\CAT$ we have an
equivalence $\Alg^\mathrm{wk}(S) \catequiv \CAT_{\cl{K}, \cl{L}}$, under
which the forgetful functor corresponds to the inclusion into $\CAT$.  The
lemma follows.

Let $\CATc$ be the subcategory of $\CAT$ consisting of the categories with,
and functors preserving, small coproducts.  

\begin{proposition}
\label{prop:three-refl}\label{twopointeight}\label{prop:little-refl}

The inclusions
\[
\CATp \hra \CAT,
\qquad
\CATc \hra \CAT,
\qquad
\CATd \hra \CAT
\]
all reflect limits.

\end{proposition}

\begin{proof}
Lemma~\ref{lemma:KL-refl} proves the first two cases.  It also implies that
the inclusion $\CATpc \hra \CAT$ reflects limits, where $\CATpc$ is
the subcategory of $\CAT$ consisting of the categories with, and functors
preserving, finite products and small coproducts.  The inclusion $\CATd
\hra \CATpc$ also reflects limits (being full and faithful),
proving the third case.
\end{proof}

For a 2-category $\cl{C}$, let $\cat{Mnd}(\cl{C})$ be the category of monads
in $\cl{C}$ and their lax morphisms, and let $\cat{Mnd}^\mathrm{wk}(\cl{C})$
be the subcategory consisting of only the weak morphisms.

\begin{proposition}
\label{prop:gen-reflection}

Let $\cl{C}$ be a 2-category, $\scat{I}$ a category, and $(A_i, T_i)_{i \in
\scat{I}}$ a diagram in $\cat{Mnd}^\mathrm{wk}(\cl{C})$.  Let
\begin{equation}
\label{eq:gen-mnd-cone}
\left(
(A, T) \map{(P_i, \pi_i)} (A_i, T_i)
\right)_{i \in \scat{I}}
\end{equation}
be a cone in $\cat{Mnd}^\mathrm{wk}(\cl{C})$.  Suppose that
\begin{equation}
\label{eq:gen-und-cone}
\left(
A \map{P_i} A_i
\right)_{i \in \scat{I}}
\end{equation}
is a 2-limit cone in $\cl{C}$.  Then~(\ref{eq:gen-mnd-cone}) is a limit cone
in $\cat{Mnd}(\cl{C})$.
\end{proposition}

This result, whose proof follows shortly, has a partial explanation in
terms of 2-monad theory.  For many 2-categories $\cl{C}$, including $\CAT$,
there is a 2-monad $S$ on $\cl{C}$ such that (with the obvious notation)
\[
\Alg^\mathrm{wk}(S) \iso \cat{Mnd}^\mathrm{wk}(\cl{C}),
\qquad
\Alg^\mathrm{lax}(S) \iso \cat{Mnd}(\cl{C})
\]
and $S$ has rank.  The forgetful functor $\Alg^\mathrm{wk}(S) \lra \cl{C}$
reflects 2-limits, so the proposition will follow as long as the inclusion
$\Alg^\mathrm{wk}(S) \hra \Alg^\mathrm{lax}(S)$ preserves limits.
We know from~\cite{bkp1} that this inclusion has a left biadjoint, and so
preserves \emph{bi}limits, and that the inclusion $\Alg^\mathrm{str}(S)
\hra \Alg^\mathrm{lax}(S)$ of \emph{strict} algebras has a left
adjoint, and so preserves limits.  These facts are not quite what we need, but
at least make the result unsurprising.  In any case, the following proof is
elementary.

\begin{proof}
Let 
\[
\left(
(B, S) \map{(Q_i, \kappa_i)} (A_i, T_i)
\right)_{i \in \scat{I}}
\]
be a cone in $\cat{Mnd}(\cl{C})$.  Since~(\ref{eq:gen-und-cone}) is a 2-limit,
and in particular a limit, there is a unique morphism $\overline{Q}\: B \lra A$
such that 
\[
Q_i = \left( B \map{\overline{Q}} A \map{P_i} A_i \right)
\]
for all $i \in \scat{I}$.

We claim that there exists a unique 2-cell $\overline{\kappa}$ such that

\[\psset{unit=0.08cm,labelsep=0pt,nodesep=3pt}
\pspicture(0,-4)(70,22)

\rput(0,10){
\pspicture(40,22)

% a1 a2 a3
% b1 b2 b3

\rput(0,20){\rnode{a1}{$B$}} % top left
\rput(20,20){\rnode{a2}{$A$}} % top mid
\rput(40,20){\rnode{a3}{$A_i$}} % top right

\rput(0,0){\rnode{b1}{$B$}}   % bottom left
\rput(20,0){\rnode{b2}{$A$}}  % bottom mid
\rput(40,0){\rnode{b3}{$A_i$}}  % bottom right

\psset{nodesep=3pt,labelsep=2pt,arrows=->}
\ncline{a1}{a2}\naput{{\scriptsize $\bar{Q}$}} % top left
\ncline{b1}{b2}\nbput{{\scriptsize $\bar{Q}$}} % bottom left
\ncline{a2}{a3}\naput{{\scriptsize $P_i$}} % top right
\ncline{b2}{b3}\nbput{{\scriptsize $P_i$}} % bottom right

\ncline{a1}{b1}\nbput{{\scriptsize $S$}} % left
\ncline{a2}{b2}\naput{{\scriptsize $T$}} % middle
\ncline{a3}{b3}\naput{{\scriptsize $T_i$}} % right

\pnode(13,13){c1}
\pnode(7,7){c2}
\ncline[doubleline=true,arrowinset=0.6,arrowlength=0.8,arrowsize=0.5pt 2.1]{c1}{c2} \naput[npos=0.4,labelsep=0.4pt]{{\scriptsize $\bar{\kappa}$}}

\pnode(33,13){a3}
\pnode(27,7){b3}
\ncline[doubleline=true,arrowinset=0.6,arrowlength=0.8,arrowsize=0.5pt 2.1]{a3}{b3} \naput[npos=0.4,labelsep=0.7pt]{{\scriptsize $\pi_i$}}

\endpspicture}

\rput(36,8){$=$}

\rput(60,10){
\pspicture(20,22)

% a1 a2
% b1 b2

\rput(0,20){\rnode{a1}{$B$}} % top left
\rput(20,20){\rnode{a2}{$A_i$}} % top right

\rput(0,0){\rnode{b1}{$B$}}   % bottom left
\rput(20,0){\rnode{b2}{$A_i$}}  % bottom right

\psset{nodesep=3pt,labelsep=2pt,arrows=->}
\ncline{a1}{a2}\naput{{\scriptsize $Q_i$}} % top
\ncline{b1}{b2}\nbput{{\scriptsize $Q_i$}} % bottom
\ncline{a1}{b1}\nbput{{\scriptsize $S$}} % left
\ncline{a2}{b2}\naput{{\scriptsize $T_i$}} % right

\psset{labelsep=0.5pt}
\pnode(13,13){a3}
\pnode(7,7){b3}
\ncline[doubleline=true,arrowinset=0.6,arrowlength=0.8,arrowsize=0.5pt 2.1]{a3}{b3} \naput[npos=0.4]{{\scriptsize $\kappa_i$}}

\endpspicture}

\endpspicture\]
for all $i \in \scat{I}$.  Indeed, this equation is equivalent to

\[\psset{unit=0.08cm,labelsep=0pt,nodesep=3pt}
\pspicture(0,-15)(80,22)

\rput(0,10){
\pspicture(40,22)

% a1 a2 a3
% b1 b2 b3

\rput(0,20){\rnode{a1}{$B$}} % top left
\rput(20,20){\rnode{a2}{$A$}} % top mid
%\rput(40,20){\rnode{a3}{$A_i$}} % top right

\rput(0,0){\rnode{b1}{$B$}}   % bottom left
\rput(20,0){\rnode{b2}{$A$}}  % bottom mid
\rput(40,0){\rnode{b3}{$A_i$}}  % bottom right

\psset{nodesep=2pt,labelsep=2pt,arrows=->}
\ncline{a1}{a2}\naput{{\scriptsize $\bar{Q}$}} % top left
\ncline{b1}{b2}\nbput{{\scriptsize $\bar{Q}$}} % bottom left
%\ncline{a2}{a3}\naput{{\scriptsize $P_i$}} % top right
\ncline{b2}{b3}\nbput{{\scriptsize $P_i$}} % bottom right

\ncline{a1}{b1}\nbput{{\scriptsize $S$}} % left
\ncline{a2}{b2}\naput{{\scriptsize $T$}} % middle
%\ncline{a3}{b3}\naput{{\scriptsize $T_i$}} % right

\pnode(11.5,11.5){c1}
\pnode(8.5,8.5){c2}
\ncline[nodesep=0pt,doubleline=true,arrowinset=0.6,arrowlength=0.8,arrowsize=0.5pt 2.1]{c1}{c2} \naput[npos=0.4,labelsep=0.4pt]{{\scriptsize $\bar{\kappa}$}}

\endpspicture}

\rput(36,8){$=$}

\rput(70,10){
\pspicture(40,22)

% a1 a2 a3
% b1 b2 b3
%    c

\rput(0,20){\rnode{a1}{$B$}} % top left
\rput(20,20){\rnode{a2}{$A$}} % top mid
\rput(40,20){\rnode{a3}{$A$}} % top right

\rput(0,0){\rnode{b1}{$B$}}   % bottom left
\rput(20,7){\rnode{b2}{$A_i$}}  % bottom mid
\rput(40,-5){\rnode{b3}{$A_i$}}  % bottom right

\rput(15,-10){\rnode{c}{$A$}}  % bottom right

\psset{nodesep=2pt,labelsep=2pt,arrows=->}
\ncline{a1}{a2}\naput{{\scriptsize $\bar{Q}$}} % top left
\ncline{b1}{b3}\naput[labelsep=1pt,npos=0.5]{{\scriptsize $Q_i$}} % bottom left
\ncline{a2}{a3}\naput{{\scriptsize $T$}} % top right
\ncline{b2}{b3}\nbput[npos=0.4,labelsep=0pt]{{\scriptsize $T_i$}} % bottom right

\ncline{a1}{b1}\nbput{{\scriptsize $S$}} % left
\ncline{a2}{b2}\naput{{\scriptsize $P_i$}} % middle
\ncline{a3}{b3}\naput{{\scriptsize $P_i$}} % right

\ncline[nodesepA=2pt,nodesepB=0pt]{a1}{b2}\nbput[labelsep=0pt]{{\scriptsize $Q_i$}} % right
\ncline[nodesepA=2pt,nodesepB=0pt]{b1}{c}\nbput[labelsep=0pt,npos=0.4]{{\scriptsize $\bar{Q}$}} % right
\ncline{c}{b3}\nbput{{\scriptsize $P_i$}} % right

{\psset{doubleline=true,arrowinset=0.6,arrowlength=0.5,arrowsize=0.5pt 2.1,nodesep=0pt}
\rput[c](8,4){\pcline{->}(3,3)(0,0) \naput[labelsep=0pt]{{\scriptsize $\kappa_i$}}}}

{\psset{doubleline=true,arrowinset=0.6,arrowlength=0.5,arrowsize=0.5pt 2.1,nodesep=0pt}
\rput[c](28,9){\pcline{->}(3,3)(0,0) \naput[labelsep=-2pt,npos=0.3]{{\scriptsize $\pi_i^{-1}$}}}}

\rput(13,16){{\scriptsize $=$}}
\rput(16,-5){{\scriptsize $=$}}

\endpspicture}

\endpspicture\]
for all $i \in \scat{I}$ (using the fact that $(P_i, \pi_i)$ is a \emph{weak}
morphism of monads).  By the 2-limit property of~(\ref{eq:gen-und-cone}), the
claim holds as long as the 2-cells
\[
\psset{unit=0.1cm,labelsep=2pt,nodesepA=2pt,nodesepB=1pt}
\left(
\pspicture(-2,9)(22,20)

\rput(0,10){\rnode{a1}{$B$}}  % named node, with something placed there
\rput(20,10){\rnode{a2}{$A_i$}}  % named node, with something placed there

\pcline[linewidth=0.6pt,doubleline=true,arrowinset=0.6,arrowlength=0.8,arrowsize=0.5pt 2.1]{->}(10,13)(10,7)  \naput{{\scriptsize $$}}

\ncarc[arcangle=45]{->}{a1}{a2}\naput{{\scriptsize $P_iT\bar{Q}$}}
\ncarc[arcangle=-45]{->}{a1}{a2}\nbput{{\scriptsize $P_i\bar{Q}S$}}

\endpspicture\right)_{i \in \bI}
\]
on the right-hand side form a morphism of cones.  This is a straightforward
check, proving the claim.

It is also straightforward to check that the pair $(\overline{Q},
\overline{\kappa})$ satisfies the axioms for a morphism of monads.  The result
follows. 
\end{proof}

We now deduce a result that we have used repeatedly for calculating limits
in $\MNDd$ and related slice categories.  To a first approximation it says
that the forgetful functor $\MND \lra \CAT$ reflects limits.  One might
expect this by some kind of 2-monadicity, an object of $\MND$ being an
object of $\CAT$ equipped with extra structure.  But since the morphisms in
$\MND$ are \emph{lax}, and since we are interested in $\MNDd$ as well as
$\MND$, the statement is more complicated.

\begin{proposition}\label{prop:reflection}\label{twopointseven}

Let $\scat{I}$ be a category, let $(\cl{A}_i, T_i)_{i \in \scat{I}}$ be a
diagram in $\MNDdwk$, and let
\begin{equation}
\label{eq:refl-mnd-cone}
\left(
(\cl{A}, T) \map{(P_i, \pi_i)} (\cl{A}_i, T_i)
\right)_{i \in \scat{I}}
\end{equation}
be a cone in $\MNDwk$.  Suppose that the cone
\begin{equation}
\label{eq:refl-und-cone}
\left(
\cl{A} \map{P_i} \cl{A}_i
\right)_{i \in \scat{I}}
\end{equation}
lies in $\CATd$ and is a limit cone in $\CAT$.  Then:
\begin{enumerate}
\item 
\label{part:reflection-dist}
$(\cl{A}, T) \in \MNDd$

\item
\label{part:reflection-limit}
the cone~(\ref{eq:refl-mnd-cone}) is a limit cone in both $\MND$ and $\MNDd$.

\end{enumerate}

\end{proposition}

\begin{proof}
For~(\ref{part:reflection-dist}), we have to show that $T\: \cl{A} \lra \cl{A}$
preserves small coproducts.  Composing $T$ with the
cone~(\ref{eq:refl-und-cone}) gives a cone
\begin{equation}
\label{eq:composite-cone}
\left(
\cl{A} \map{P_i T} \cl{A}_i
\right)_{i \in \scat{I}}
\end{equation}
in $\CAT$.  In fact this cone lies in $\CATc$, since for each $i \in \scat{I}$
we have an isomorphism
\[\psset{unit=0.1cm,labelsep=0pt,nodesep=3pt}
\pspicture(20,22)

% a1 a2
% b1 b2

\rput(0,20){\rnode{a1}{$\cA$}} % top left
\rput(20,20){\rnode{a2}{$\cA_i$}} % top right

\rput(0,0){\rnode{b1}{$\cA$}}   % bottom left
\rput(20,0){\rnode{b2}{$\cA_i$}}  % bottom right

\psset{nodesep=3pt,labelsep=2pt,arrows=->}
\ncline{a1}{a2}\naput{{\scriptsize $P_i$}} % top
\ncline{b1}{b2}\nbput{{\scriptsize $P_i$}} % bottom
\ncline{a1}{b1}\nbput{{\scriptsize $T$}} % left
\ncline{a2}{b2}\naput{{\scriptsize $T_i$}} % right

\psset{labelsep=1.5pt}
\pnode(13,13){a3}
\pnode(7,7){b3}
\ncline[doubleline=true,arrowinset=0.6,arrowlength=0.8,arrowsize=0.5pt 2.1]{a3}{b3} \naput[npos=0.4]{{\scriptsize $\pi_i$}} \nbput[npos=0.4]{{\scriptsize $\iso$}}

\endpspicture\]
and the functors $T_i$ and $P_i$ are, by hypothesis, both
coproduct-preserving.  

The cone (\ref{eq:refl-und-cone}) is a limit in $\CATc$ as well as
$\CAT$, by Proposition~\ref{prop:three-refl}.  Hence $T$, as the unique
morphism of cones from~(\ref{eq:composite-cone}) to~(\ref{eq:refl-und-cone})
in $\CAT$, preserves coproducts.

For~(\ref{part:reflection-limit}), we know from Lemma~\ref
{lemma:lim-2-lim} and
Proposition~\ref{prop:gen-reflection} that the cone~(\ref{eq:refl-mnd-cone})
is a limit in $\MND$.  We have to show that it is a limit in $\MNDd$.  In
other words, we have to show that given an object $(\cl{B}, S) \in \MNDd$ and
a morphism 
\[
(\overline{Q}, \overline{\kappa})\: (\cl{B}, S) \lra (\cl{A}, T)
\]
in $\MND$ such that each composite
\[
(\cl{B}, S) \map{(\overline{Q}, \overline{\kappa})}
(\cl{A}, T) \map{(P_i, \pi_i)}
(\cl{A}_i, T_i)
\]
lies in $\MNDd$, the original morphism $(\overline{Q}, \overline{\kappa})$
also lies in $\MNDd$.  But a morphism in $\MND$ lies in $\MNDd$ if and only if
its underlying functor lies in $\CATd$, so this follows from the fact that the
inclusion $\CATd \hra \CAT$ reflects limits
(Proposition~\ref{prop:three-refl}). 
\end{proof}

%\bibliography{../bib/bib1208}

\begin{references*}

% {10}

\bibitem{ada1}
Ji\v{r}\'\i\ Ad\'amek.
\newblock Free algebras and automata realization in the language of categories.
\newblock {\em Math.\ Univ.\ Carolinae}, 15:589--602, 1974.

\bibitem{ada2}
Ji\v{r}\'\i\ Ad\'amek.
\newblock Introduction to coalgebra.
\newblock {\em Theory and Applications of Categories}, 14:157--199, 2005.

\bibitem{bat1}
M.~A. Batanin.
\newblock Monoidal globular categories as a natural environment for the theory
  of weak $n$-categories.
\newblock {\em Adv. Math.}, 136(1):39--103, 1998.

\bibitem{batweb1}
Michael Batanin and Mark Weber.
\newblock Algebras of higher operads as enriched categories.
\newblock {\em Applied Categorical Structures}, 19(1):93--135, 2011.

\bibitem{bur1}
A.~Burroni.
\newblock ${T}$-cat{\'e}gories (cat{\'e}gories dans un triple).
\newblock {\em Cahiers de Topologie et G{\'e}om{\'e}trie Diff{\'e}rentielle
  Cat{\'e}goriques}, 12:215--321, 1971.

\bibitem{che19}
Eugenia Cheng.
\newblock The universal loop space operad and generalisations, 2009.
\newblock Talk at the British Topology Meeting and CT2010, submitted.

\bibitem{che16}
Eugenia Cheng.
\newblock Comparing operadic theories of $n$-category.
\newblock {\em Homology, Homotopy and Applications}, 13(2):217--249, 2011.
\newblock Also E-print {\tt arXiv:0809.2070}.

\bibitem{cg2}
Eugenia Cheng and Nick Gurski.
\newblock Towards an $n$-category of cobordisms.
\newblock {\em Theory and Applications of Categories}, 18:274--302, 2007.

\bibitem{cl1}
Eugenia Cheng and Aaron Lauda.
\newblock Higher dimensional categories: an illustrated guide book, 2004.
\newblock Available via \href{http://eugeniacheng.com/guidebook}{eugeniacheng.com/guidebook}

\bibitem{kel4}
G.~M. Kelly.
\newblock A unified treatment of transfinite constructions for free algebras,
  free monoids, colimits, associated sheaves and so on.
\newblock {\em Bull. Austral. Math. Soc.}, 22(1):1--83, 1980.

\bibitem{kel6}
G.~M. Kelly.
\newblock Elementary observations on 2-categorical limits.
\newblock {\em Bulletin of the Australian Mathematical Society}, 39:301--317,
  1989.

\bibitem{lac3}
Stephen Lack.
\newblock A 2-categories companion.
\newblock In John Baez and Peter May, editors, {\em $n$-Categories: Foundations
  and Applications}, volume 149, pages 105--191. IMA, 2009.
\newblock E-print {\tt arXiv:math.CT/0702535}.

\bibitem{lam1}
Joachim Lambek.
\newblock A fixpoint theorem for complete categories.
\newblock {\em Mathematische Zeitschrift}, 103:151--161, 1968.

\bibitem{lei7}
Tom Leinster.
\newblock A survey of definitions of $n$-category.
\newblock {\em Theory and Applications of Categories}, 10:1--70, 2002.

\bibitem{lei8}
Tom Leinster.
\newblock {\em Higher Operads, Higher Categories}.
\newblock Number 298 in London Mathematical Society Lecture Note Series.
  Cambridge University Press, 2004.

\bibitem{mac1}
Saunders {Mac Lane}.
\newblock {\em Categories for the Working Mathematician}, volume~5 of {\em
  Graduate Texts in Mathematics}.
\newblock Springer-Verlag, New York, second edition, 1998.

\bibitem{may2}
Peter May.
\newblock {\em The Geometry of Iterated Loop Spaces}, volume 271 of {\em
  Lecture Notes in Mathematics}.
\newblock Springer, 1972.

\bibitem{may1}
Peter May.
\newblock Operadic categories, ${A}_\infty$-categories and $n$-categories,
  2001.
\newblock notes of a talk given at Morelia, Mexico.

\bibitem{bkp1}
G.~M. Kelly A. J.~Power {R. Blackwell}.
\newblock Two-dimensional monad theory.
\newblock {\em Journal of Pure and Applied Algebra}, 59(1):1--41, 1989.

\bibitem{str1}
Ross Street.
\newblock The formal theory of monads.
\newblock {\em Journal of Pure and Applied Algebra}, 2:149--168, 1972.

\bibitem{tri1}
Todd Trimble.
\newblock What are `fundamental $n$-groupoids'?, 1999.
\newblock Seminar at DPMMS, Cambridge, 24 August 1999.

\bibitem{web2}
Mark Weber.
\newblock Multitensors as monads on categories of enriched graphs, 2011.
\newblock {\em Theory and Applications of Categories}, 28:857--932, 2013.
\newblock E-print {\tt arXiv:1106.1977}.

\end{references*}

\ed